\documentclass[a4paper,oneside,11pt]{article}
\usepackage[a4paper]{geometry}

\title{Integrable measure equivalence rigidity of right-angled Artin groups via quasi-isometry}
\author{Camille Horbez and Jingyin Huang}
\date{\today}

\usepackage{aeguill}                 
\usepackage{graphicx}
\usepackage{amsmath,amsfonts,amssymb,amsthm}
\usepackage{pstricks} 
\usepackage{epstopdf}
\usepackage{srcltx}
\usepackage[utf8]{inputenc} 
\usepackage{hyperref} 
\usepackage[english]{babel} 
\usepackage{comment}

\usepackage[all]{xy}
\usepackage{bbm}
\begin{document}
	%\tableofcontents 
	\newtheorem{de}{Definition}[section]
	\newtheorem{dec}[de]{Definition-Construction}
	\newtheorem{theo}[de]{Theorem} 
	\newtheorem{prop}[de]{Proposition}
	\newtheorem{lemma}[de]{Lemma}
	\newtheorem{cor}[de]{Corollary}
	\newtheorem{propd}[de]{Proposition-Definition}
	\newtheorem{conj}[de]{Conjecture}
	\newtheorem{claim}[de]{Claim}
	\newtheorem*{claim2}{Claim}
	\newtheorem{obs}[de]{Observation}
	
	\newtheorem{theointro}{Theorem}
	\newtheorem*{defintro}{Definition}
	\newtheorem{corintro}[theointro]{Corollary}

	\theoremstyle{remark}
	\newtheorem{rk}[de]{Remark}
	\newtheorem{ex}[de]{Example}
	\newtheorem{question}[de]{Question}
	\newtheorem{assumption}[de]{Standing assumption}
	
	\normalsize
	
	\newcommand{\st}{\mathrm{st}}
	\newcommand{\lk}{\mathrm{lk}}
	\newcommand{\Bij}{\mathrm{Bij}}
	\newcommand{\Ga}{\Gamma}
	\newcommand{\ga}{\gamma}
	\newcommand{\Htaut}{(\mathrm{H}_{\mathrm{taut}})}
	\newcommand{\Hicc}{(\mathrm{H}_{\mathrm{icc}})}
	\newcommand{\Htrans}{(\mathrm{H}_{\mathrm{trans}})}
	\newcommand{\cala}{\mathcal{A}}
	\newcommand{\calb}{\mathcal{B}}
	\newcommand{\calh}{\mathcal{H}}
	\newcommand{\calg}{\mathcal{G}}
	\newcommand{\calk}{\mathcal{K}}
	\newcommand{\call}{\mathcal{L}}
	\newcommand{\calo}{\mathcal{O}}
	\newcommand{\calp}{\mathcal{P}}
	\newcommand{\calw}{\mathcal{W}}
	\newcommand{\calz}{\mathcal{Z}}
	\newcommand{\caln}{\mathcal{N}}
	\newcommand{\G}{\mathsf G}
	\newcommand{\sfH}{\mathsf H}
	\newcommand{\Out}{\operatorname{Out}}
	\newcommand{\Aut}{\mathrm{Aut}}
	\newcommand{\stab}{\mathrm{Stab}}
	\newcommand{\Isom}{\mathrm{Isom}}
	\newcommand{\Maps}{\mathrm{Maps}}
	\newcommand{\dunion}{\sqcup}
	\newcommand{\B}{\mathbb B}
	\newcommand{\Bu}{\B}
	\newcommand{\fp}{\operatorname{FP}}
	\newcommand{\ba}{\mathbb A}
	\newcommand{\Prob}{\mathrm{Prob}}
	\newcommand{\Stab}{\mathrm{Stab}}
	\newcommand{\Conv}{\mathrm{Conv}}
	\newcommand{\reg}{\mathrm{reg}}
	\newcommand{\sing}{\mathrm{sing}}
	\newcommand{\thin}{\mathrm{thin}}
	\newcommand{\cb}{\overline{\B}^R}
	\newcommand{\Perm}{\operatorname{Perm}}
	\newcommand{\sfv}{\mathsf{v}}
	\newcommand{\sfw}{\mathsf{w}}
 \newcommand{\Homeo}{\mathrm{Homeo}}

	\newcommand{\actson}{\curvearrowright}
	\newcommand{\gr}{\mathrm{gr}}
	
	\makeatletter
	\edef\@tempa#1#2{\def#1{\mathaccent\string"\noexpand\accentclass@#2 }}
	\@tempa\rond{017}
	\makeatother
	
	\newcommand{\Ccom}[1]{\Cmod\marginpar{\color{red}\tiny #1 --ch}} 
	\newcommand{\Ccomm}[1]{\Jmod\marginpar{\color{blue}\tiny #1 --jh}}
	\newcommand{\Cmod}{$\textcolor{red}{\clubsuit}$} 
	\newcommand{\Jmod}{$\textcolor{blue}{\spadesuit}$}
	
		\maketitle

		\begin{abstract}
			Let $G$ be a right-angled Artin group with $|\Out(G)|<+\infty$. We prove that if a countable group $H$ with bounded finite subgroups is measure equivalent to $G$, with an $L^1$-integrable measure equivalence cocycle towards $G$, then $H$ is finitely generated and quasi-isometric to $G$. In particular, through work of Kleiner and the second-named author, $H$ acts properly and cocompactly on a $\mathrm{CAT}(0)$ cube complex which is quasi-isometric to $G$ and equivariantly projects to the right-angled building of $G$. 
		
			As a consequence of work of the second-named author, we derive a superrigidity theorem in integrable measure equivalence for an infinite class of right-angled Artin groups, including those whose defining graph is an $n$-gon with $n\ge 5$. In contrast, we also prove that if a right-angled Artin group $G$ with $|\Out(G)|<+\infty$ splits non-trivially as a product, then there does not exist any locally compact group which contains all groups $H$ that are $L^1$-measure equivalent to $G$ as lattices, even up to replacing $H$ by a finite-index subgroup and taking the quotient by a finite normal subgroup.
		\end{abstract}

\setcounter{tocdepth}{1}
\tableofcontents

\normalsize
  
		\section{Introduction}

  \subsection{Background, history and motivation}

Measure equivalence was introduced by Gromov \cite{Gro} as a measure-theoretic analogue to quasi-isometry. Two countable groups $G_1,G_2$ are \emph{measure equivalent} if there exists a standard (non-null) measure space $\Omega$ (called a \emph{coupling}), equipped with a measure-preserving action of $G_1\times G_2$, such that for every $i\in\{1,2\}$, the $G_i$-action on $\Omega$ is free and has a finite measure fundamental domain. Quasi-isometry between finitely generated groups has an analogous characterization, with $\Omega$ a (non-empty) locally compact topological space on which $G_1$ and $G_2$ have commuting actions, both properly discontinuous and cocompact. As a motivating example, lattices in the same locally compact second countable group are always measure equivalent -- if cocompact, they are quasi-isometric. Despite the analogy in definitions, there is no implication in either way. By a celebrated theorem of Ornstein--Weiss \cite{OW}, building on earlier work of Dye \cite{Dye1,Dye2}, all countably infinite amenable groups are measure equivalent -- they are far from being all quasi-isometric. Conversely, measure equivalence preserves Property~(T) \cite{Fur-me} or ratios of $\ell^2$-Betti numbers \cite{Gab-l2}, which are not quasi-isometry invariants. In contrast to the Ornstein--Weiss theorem, there has been a lot of effort in proving the quasi-isometric and measure equivalence rigidity of many important classes of groups. These include lattices in higher rank simple Lie groups \cite{Fur-me,KL,EF,Esk}, where Zimmer's cocycle superrigidity theorem \cite{Zim2} played a central role on the side of measure equivalence, and surface mapping class groups \cite{Kid-me,BKMM,Ham}.
 
By using a notion of \emph{uniform measure equivalence} to study the large-scale geometry of amenable groups \cite{Sha}, Shalom strengthened the bridge between measure equivalence and quasi-isometry. %By imposing an integrability condition on the , 
In the same spirit of imposing a quantitative control on the word length of a cocycle naturally associated to the measure equivalence, Bader--Furman--Sauer coined the notion of \emph{integrable measure equivalence}, reviewed at the beginning of Section~\ref{sec:results}. For this notion, they established new rigidity theorems for some rank $1$ lattices, including all lattices in $\mathrm{Isom}(\mathbb{H}_{\mathbb{R}}^n)$ with $n\ge 3$, and all cocompact lattices in $\mathrm{Isom}(\mathbb{H}_{\mathbb{R}}^2)$ -- for the latter rigidity fails for (standard) measure equivalence \cite{BFS}. The idea behind integrable measure equivalence finds its roots in the work of Margulis \cite{Mar}, where integrability conditions on lattices appear to be crucial in induction arguments, see also \cite{Sha2}. Integrable measure equivalence retains more geometric information about the group, like growth \cite{Aus} or the isoperimetric profile \cite{DKLMT}. On the ergodic side, integrability conditions on orbit equivalence cocycles (closely related to measure equivalence cocycles) already appeared in Belinskaya's theorem \cite{Bel} regarding actions of $\mathbb{Z}$, and have also been studied in connection to ergodic notions like entropy \cite{Aus2,KLi,KLi2}.

Outside the realm of Lie groups, there has been growing interest in understanding lattices in totally disconnected locally compact groups and their rigidity properties. These turn out to be very mysterious compared to the more classical lattices acting on symmetric spaces and buildings. A general classification theorem of lattice embeddings due to Bader--Furman--Sauer \cite[Theorem~A]{BFS2} highlights the importance of the totally disconnected case. Of particular relevance are lattices acting on $\mathrm{CAT}(0)$ cube complexes, among which right-angled Artin groups form a prototypical example. 

Right-angled Artin groups are also important for many other reasons. We mention in particular their connections to buildings \cite{Dav}, and the deep combinatorial tools from the work of Haglund--Wise \cite{HW,HW2} that famously led to Agol's solution to the virtual Haken conjecture \cite{Ago}. These also turn out to be crucial ingredients in quasi-isometry, measure equivalence and other forms of rigidity, as will be further demonstrated in this paper.

Given a finite simplicial graph $\Gamma$, the right-angled Artin group $G_\Gamma$ has a finite presentation with one generator per vertex of $\Gamma$, where two generators commute whenever the associated vertices are adjacent. There has been a lot of work regarding the quasi-isometry classification/rigidity of these groups, e.g.\ \cite{BN,BJN,BKS,Hua,Hua2,Marg}; some of them are quasi-isometrically rigid \cite{Hua-QI}. In previous work \cite{HH}, we initiated a study of right-angled Artin groups in measure equivalence. Contrary to the situation in quasi-isometry, and in contrast to the behaviour of certain other classes of Artin groups \cite{HH1}, they demonstrate a lack of rigidity, in the sense that the class of groups that are measure equivalent to $G_\Gamma$ is huge, for instance it contains all graph products of countably infinite amenable groups over $\Gamma$. In fact the line between rigidity and flexibility is quite subtle, see e.g.\ \cite{HHI} where we recover rigidity by imposing extra ergodicity assumptions on the coupling. 

In the present paper, we relate integrable measure equivalence and quasi-isometry for right-angled Artin groups with finite outer automorphism group, and derive a superrigidity theorem in some cases. The finiteness condition on the outer automorphism group naturally appears in rigidity questions; it is easily readable on the defining graph \cite{Ser,Lau} and is generic in a sense \cite{CF,Day}.

\subsection{Integrable measure equivalence versus quasi-isometry for RAAGs}\label{sec:results}

Let $G$ and $H$ be two countable groups, with $G$ finitely generated. Let $|\cdot|_G$ be a word length on $G$ with respect to some finite generating set. An \emph{$(L^1,L^0)$-measure equivalence coupling from $H$ to $G$} is a measure equivalence coupling $(\Omega,\mu)$ between $G$ and $H$ such that there exists a Borel fundamental domain $X_G$ for the $G$-action on $\Omega$, for which the measure equivalence cocycle $c:H\times X_G\to G$ (defined by letting $c(h,x)$ be the unique element $g\in G$ such that $ghx\in X_G$) satisfies \[\forall h\in H,~~  \int_{X_G}|c(h,x)|_G\; d\mu(x)<+\infty.\] The terminology $(L^1,L^0)$-measure equivalence coupling comes from the fact that we are only imposing an $L^1$-integrability condition from $H$ to $G$, not from $G$ to $H$ (this would not make sense as $H$ is not assumed finitely generated). In this respect, this notion is weaker than $L^1$-measure equivalence in the sense of Bader--Furman--Sauer \cite{BFS}, who imposed integrability in both directions.

We say that a group $H$ has  \emph{bounded finite subgroups} if there is a bound on the cardinality of its finite subgroups. Our main theorem is the following.
  
		\begin{theointro}\label{theointro:main}
			Let $G$ be a right-angled Artin group with $|\Out(G)|<+\infty$, let $H$ be a countable group with bounded finite subgroups. 
			
			If there exists an $(L^1,L^0)$-measure equivalence coupling from $H$ to $G$, then $H$ is finitely generated and quasi-isometric to $G$.
		\end{theointro}

 As such, the theorem fails if $H$ is allowed to have unbounded finite subgroups.  This comes from the existence of infinitely generated non-uniform lattices in the automorphism group of the universal cover of the Salvetti complex of $G$, whenever $G$ is a non-abelian right-angled Artin group (see \cite[Section~4.2]{HH}).  Another example of groups $H$ with unbounded finite subgroups having an $(L^1,L^0)$-measure equivalence coupling towards $G$ is given by graph products over the defining graph of $G$, with vertex groups isomorphic to $\oplus_{\mathbb{N}}\mathbb Z/2\mathbb Z$: indeed, the odometer gives a measure equivalence (in fact an orbit equivalence) between $\oplus_{\mathbb{N}}\mathbb Z/2\mathbb Z$ and $\mathbb{Z}$ with an $L^\infty$-integrability condition, and the integrability passes to graph products.
 %\Ccom{added ``and the integrability passes to graph products"} 
 We do not know any finitely generated examples, however.  This is related to the deep and important question raised in \cite[Question~33]{FHT}, asking which polyhedral complexes admit both finitely generated and non-finitely-generated non-uniform lattices.
 
% : this is related to the deep and important question regarding the existence of finitely generated non-uniform lattices acting on polyhedral complexes, see \cite[Question~33]{FHT} -- we mention that finitely generated non-uniform lattices in products of trees were constructed by Rémy \cite{Rem} using Kac--Moody groups. 

The integrability condition in Theorem~\ref{theointro:main} is also crucial, as it excludes examples coming from graph products of amenable groups.

We mention that Theorem~\ref{theointro:main} is already new even for $G=\mathbb{Z}$, when $H$ is not assumed finitely generated. Finitely generated groups $H$ with an $(L^1,L^0)$-measure equivalence coupling from $H$ to $\mathbb{Z}$ grow linearly by a theorem of Bowen \cite{Aus}, and are therefore virtually cyclic. But excluding the possibility that $H$ be infinitely generated (e.g.\ $H=\mathbb{Q}$, whose finitely generated subgroups are all isomorphic to $\mathbb{Z}$) requires a new argument. In the present work, this generalization to infinitely generated groups is not just for the sake of the greatest generality. Indeed,  when $G$ is an arbitrary right-angled Artin group with $|\Out(G)|<+\infty$, even if we start with a finitely generated group $H$, in the course of the proof, we will have to work with subgroups of $H$ that arise as point stabilizers for some $H$-action on a $\mathrm{CAT}(0)$ cube complex, and we will not know \emph{a priori} that these are finitely generated. In fact, proving finite generation will be an important task in the proof. 

For certain groups (like surface mapping class groups), there are separate rigidity statements in measure equivalence and quasi-isometry, which imply the conclusion of Theorem~\ref{theointro:main}. This is not the case however for the class of groups in Theorem~\ref{theointro:main}. 

\subsection{Consequences to superrigidity}

Groups that are quasi-isometric to a right-angled Artin group with finite outer automorphism group have been extensively studied \cite{HK,Hua-QI}. The following corollary follows from the combination of Theorem~\ref{theointro:main} and \cite[Corollary~6.4]{HK}.

\begin{corintro}\label{corintro:cube} 
Let $G$ be a right-angled Artin group with $|\Out(G)|<+\infty$, let $H$ be a countable group with  bounded finite subgroups. 
			
			If there exists an $(L^1,L^0)$-measure equivalence coupling from $H$ to $G$, then $H$ acts properly discontinuously, cocompactly, by cubical automorphisms on a $\mathrm{CAT}(0)$ cube complex which is quasi-isometric to $G$.
\end{corintro}

In a sense, this is analogous to Furman's theorem on lattices $G$ in higher-rank simple Lie groups \cite{Fur-me}, stating that any countable group that is measure equivalent to $G$, acts as a lattice (up to a finite kernel) on the corresponding symmetric space (or \cite{BFS} for $\mathrm{Isom}(\mathbb{H}_{\mathbb{R}}^n)$ under an integrability assumption). But there is a crucial difference, in that we need to allow the cube complex to depend on $H$, see Theorem~\ref{theointro:nonrigidity}. Nevertheless, all cube complexes that arise in Corollary~\ref{corintro:cube} are simple deformations of the universal cover of the Salvetti complex of $G$, and they have a very explicit description given in Section~\ref{sec:qi-criterion} (in particular, they all collapse onto the right-angled building of $G$).

We also mention that if the group $H$ (with bounded finite subgroups) is only assumed to be measure equivalent to $G$, with no integrability condition, then we can still derive that $H$ acts cocompactly on a $\mathrm{CAT}(0)$ cube complex (in fact on the right-angled building of $G$), with amenable stabilizers, see Theorem~\ref{theointro:me-cat}.

Under further assumptions on the defining graph of $G$, it is proved in \cite[Theorem~1.2]{Hua-QI}, using deep combinatorial insights on special cube complexes by Haglund--Wise \cite{HW2}, that all uniform lattices acting on all cube complexes arising in Corollary~\ref{corintro:cube} are virtually special, which is the key for proving that they are commensurable -- see also \cite{She} for recent progress on the relationship between virtual specialness and commensurability. Thus we obtain a superrigidity theorem in integrable measure equivalence for a class of right-angled Artin groups. 

More precisely, we say that a finite simplicial graph $\Gamma$ is \emph{star-rigid} if for every vertex $v\in V\Gamma$, the only automorphism of $\Gamma$ that fixes the star of $v$ pointwise is the identity. Examples of such graphs include the $n$-gon, for any $n\ge 3$, and any asymmetric graph. We recall that an \emph{induced square} in a simplicial graph $\Gamma$ is an embedded 4-cycle $C$ such that no two opposite vertices in $C$ are adjacent in $\Gamma$.
		
		\begin{corintro}\label{corintro:rigidity}
			Let $G$ be a right-angled Artin group with $|\Out(G)|<+\infty$, whose defining graph is star-rigid and does not contain any induced square. Let $H$ be any countable group with bounded finite subgroups.
			
			If there exists an $(L^1,L^0)$-measure equivalence coupling from $H$ to $G$, then $G$ and $H$ are commensurable up to a finite kernel.
		\end{corintro}

This is perhaps the first instance where passing through quasi-isometry is key for obtaining rigidity results on the side of (integrable) measure equivalence.
  
\begin{rk}
\label{rk:commensurability}
   The conclusion of Corollary~\ref{corintro:rigidity} is false whenever the defining graph of $G$ contains an induced square in view of \cite[Theorem~1.8]{Hua-QI} (regardless of whether $\Out(G)$ is finite or not). In this case, the universal cover $X$ of the Salvetti complex of $G$ contains a subcomplex which is a product of two trees. An appropriate irreducible lattice acting on this product of trees can then be extended to a cocompact lattice in $\Aut(X)$ that is not commensurable to $G$ even up to a finite kernel.  
\end{rk}

  \subsection{Lattice envelopes}\label{sec:lattice-embedding}

Let $G$ be a right-angled Artin group with $|\Out(G)|<+\infty$, and let $\mathcal C_G$ be the class of all countable groups $H$ with  bounded finite subgroups having an $(L^1,L^0)$-measure equivalence coupling towards $G$. Motivated by results in Lie groups as discussed before, we ask the following questions. Does there exist a locally compact second countable group $\mathfrak G$ 
\begin{enumerate}
\item such that every $H\in\mathcal C_G$ has a lattice representation into $\mathfrak G$ with finite kernel?
\item such that every $H\in\mathcal C_G$ has a finite-index subgroup $H^0\subseteq H$ that has a lattice representation into $\mathfrak G$ with finite kernel?
\end{enumerate}
It turns out that the answer to the first question is negative whenever $|\Out(G)|<+\infty$, by \cite[Theorem~6.11]{HK} (whose proof constructs a group $H$ that is commensurable to $G$ and has no lattice representation with finite kernel in the same locally compact group as $G$). On the other hand, the second question has a positive answer for the groups appearing in Corollary~\ref{corintro:rigidity}. But our next theorem provides a class of right-angled Artin groups $G$ where even the second question as a negative answer, namely: the answer to the second question is negative whenever $|\Out(G)|<+\infty$ and $G$ splits as a product of two non-cyclic groups. In fact, for these groups $G$, the second question even has a negative answer if $\mathcal{C}_G$ is replaced with the (possibly smaller) class of all groups $H$ that are \emph{strongly commable} with $G$, in the sense that there exist finitely generated groups $G=G_1,\dots,G_k=H$ such that for every $i\in\{1,\dots, k-1\}$, the groups $G_i$ and $G_{i+1}$ are uniform lattices in a common locally compact second countable group -- this definition is a variation over Cornulier's notion of \emph{commability} \cite{Cor}.

\begin{theointro}\label{theointro:nonrigidity}
Let $G$ be the  direct product of two non-cyclic right-angled Artin groups with finite outer automorphism groups (hence $|\Out(G)|<+\infty$). 

Then there does not exist any locally compact group $\mathfrak G$ such that any torsion-free countable group $H$ which is strongly commable with $G$, has a finite-index subgroup $H^0$ with a lattice embedding into $\mathfrak G$.
\end{theointro}

 To prove Theorem~\ref{theointro:nonrigidity}, we construct an infinite sequence of torsion-free countable groups $H_n$ that are all strongly commable to $G$, and which cannot all virtually embed as lattices in a common locally compact group. We do not know, however, whether we can find two torsion-free countable groups $H_1,H_2$, both commable with $G$, such that no finite-index subgroups $H_1^0$ of $H_1$ and $H_2^0$ of $H_2$ have a lattice embedding in a common locally compact group.

Our proof of Theorem~\ref{theointro:nonrigidity} crucially relies on the celebrated construction by Burger--Mozes of simple groups which are uniform lattices in products of trees \cite{BM}. This forms a sharp contrast with Corollary~\ref{corintro:rigidity} which relies on virtual specialness of $H$ -- the examples leading to Theorem~\ref{theointro:nonrigidity} are very far from being virtually special. %fail virtually specialness in an aggressive manner. 
Theorem~\ref{theointro:nonrigidity} can also be viewed as a much stronger version of Remark~\ref{rk:commensurability}, where one goes from lack of commensurability to lack of (virtual) common locally compact model.

Theorem~\ref{theointro:nonrigidity} also contrasts with Kida's proof of the measure equivalence superrigidity of products of mapping class groups \cite{Kid-me}, which he derives from the superrigidity of mapping class groups together with the work of Monod--Shalom \cite{MS} on rigidity for products of negatively curved ($\mathcal{C}_{\mathrm{reg}}$) groups. The difference with our work is that the form of rigidity established by Kida for mapping class groups is even stronger in that he proves that every self measure equivalence coupling factors through the tautological one by left/right multiplication on the (extended) mapping class group. This stronger form of rigidity fails in our context -- we will come back to this while discussing the proof of our main theorem. 
 
Nevertheless, we can still get information on the possible lattice envelopes of a non-cyclic right-angled Artin group $G$ with $|\Out(G)|<+\infty$. More generally, if $H$ is a countable group with bounded finite subgroups which is measure equivalent to $G$, then any lattice embedding of $H$ is cocompact, and in fact every lattice envelope $\mathfrak{H}$ of $H$ is totally disconnected up to a compact kernel (Theorem~\ref{theo:cocompact}). And if there is an $(L^1,L^0)$-measure equivalence coupling from $H$ to $G$, then more can be said: in this case, there exists a uniformly locally finite $\mathrm{CAT}(0)$ cube complex $Y$ quasi-isometric to $G$ (having the same description as in Corollary~\ref{corintro:cube}), and a continuous homomorphism $\mathfrak{H}\to\Aut(Y)$ with compact kernel and cocompact image (Theorem~\ref{theo:lattice-embedding}). We mention that some of these results follow alternatively from the general work of Bader--Furman--Sauer on lattice envelopes \cite{BFS2}. We take a different route, in closer relation to Furman's ideas for exploiting measure equivalence rigidity towards classification of lattice embeddings \cite{Fur-lattice}. On the topic of lattice envelopes of right-angled Artin groups, we also refer the reader to the recent work of Caprace--de Medts \cite{CdM}.

\subsection{Discussion of the proof of the main theorem}

A common strategy towards proving the measure equivalence rigidity of a group $G$, initiated by Furman \cite{Fur-me} and further developed by Monod--Shalom \cite{MS}, Kida \cite{Kid-me}, Bader--Furman--Sauer \cite{BFS}, consists in showing that every self-coupling of $G$ factors through $G$ itself, or more generally through a locally compact group $\mathfrak{G}$ that contains $G$ as a lattice. However, this cannot hold in our setting even for integrable self-couplings, as this would imply that every countable group that is integrably measure equivalent to $G$ has a lattice representation with finite kernel into $\mathfrak{G}$. This would contradict our discussion in Section~\ref{sec:lattice-embedding}.

More specifically, Kida's strategy for measure equivalence rigidity of mapping class groups \cite{Kid-me} relies on having a graph $\mathcal{C}$  on which $G$ acts (in his case, the curve graph of the surface), with two properties. 
First, vertex stabilizers of $\mathcal{C}$ are ``recognized'' in an appropriate sense by any self-coupling $\Omega$ of $G$, which allows to build an equivariant map $\Omega\to\Aut(\mathcal{C})$. Second, the graph $\mathcal{C}$ is combinatorially rigid in the sense that $\Aut(\mathcal{C})$ is virtually isomorphic to $G$ -- this is ensured by a theorem of Ivanov for the curve graph \cite{Iva}. 

For right-angled Artin groups, there cannot exist a graph $\mathcal{C}$ that has both properties because of the discussion in Section~\ref{sec:lattice-embedding}. However the first half of Kida's strategy still works, and was carried in our previous work \cite{HH}. More precisely, Kim and Koberda introduced in \cite{KK} an analogue of the curve graph, which they called the \emph{extension graph} $\Gamma^e$ -- but the Polish group $\Aut(\Gamma^e)$ is much bigger than $G_\Gamma$ itself and not locally compact, in fact it contains the permutation group of countably many elements $\mathfrak{S}_\infty$ as a subgroup \cite[Corollary~4.20]{Hua}. In earlier work \cite{HH}, we proved the following fact (see also Lemma~\ref{lemma:self-coupling}), under our standing assumption that $|\Out(G_\Gamma)|<+\infty$.
\\
\\
\textbf{Fact.} For every self-coupling $\Sigma$ of $G_\Gamma$, there is a $(G_\Gamma\times G_\Gamma)$-equivariant measurable map $\Sigma\to\Aut(\Gamma^e)$ -- where the action on $\Aut(\Gamma^e)$ is by left/right multiplication. 
\\

 This is the starting point of the present paper. From there our proof has three steps.
\\
\\
\textbf{Step 1: An action of $H$ on the right-angled building of $G_\Gamma$.} This first step does not use any integrability assumption. In addition to the extension graph, another important geometric object attached to $G_\Gamma$ is its right-angled building $\B_\Gamma$, a $\mathrm{CAT}(0)$ cube complex introduced by Davis in \cite{Dav} which encodes the arrangement of flats in $G_\Gamma$. Its vertices are exactly the \emph{standard flats} in $G_\Gamma$, i.e.\ left cosets of free abelian subgroups coming from complete subgraphs of $\Gamma$ -- a vertex corresponding to a left coset of $\mathbb{Z}^k$ is said to have \emph{rank} $k$. And there is an edge between two vertices representing cosets $C,C'$ if $C\subseteq C'$ and this is a codimension $1$ inclusion. There is a natural way of filling in higher-dimensional cubes -- we refer to Section~\ref{sec:extension-graph} for more details and Remark~\ref{rk:building} on how the right-angled building is related to the classical buildings. We prove the following theorem, which answers, at least partly, the first question in \cite[p.~1028]{HH}.

\begin{theointro}\label{theointro:me-cat}
    Let $G$ be a right-angled Artin group with $|\Out(G)|<+\infty$, let $\B$ be the right-angled building of $G$, and let $H$ be a group which is measure equivalent to $G$.

    Then $H$ acts on $\B$ with amenable vertex stabilizers, and cocompactly provided that $H$ has bounded finite subgroups.  
\end{theointro}

We now say a word about its proof. It turns out that $\Aut(\B_\Gamma)$ is naturally isomorphic to $\Aut(\Gamma^e)$, see Section~\ref{sec:autos}. So the above fact ensures that every self-coupling of $G_\Gamma$ factors through $\Aut(\B_\Gamma)$. By a general argument from \cite{Kid-amalgam,BFS}, given a measure equivalence coupling $\Omega$ between $G_\Gamma$ and $H$, we obtain a representation of $H$ in $\Aut(\B_\Gamma)$ (i.e.\ an action of $H$ on $\B_\Gamma$) and a $(G_\Gamma\times H)$-equivariant measurable map $\theta:\Omega\to\Aut(\B_{\Gamma})$.

We then elaborate on an argument of Kida \cite[Section~5.2]{Kid-amalgam} and formulate a general framework that enables, given a vertex $v\in V\B_\Gamma$, to induce a measure equivalence coupling $\Omega_v$ between the stabilizers $G_v$ (for the original action of $G_\Gamma$) and $H_v$ (for the $H$-action obtained as above). This coupling $\Omega_v$ is (up to a small technicality) the preimage, under $\theta$, of the full stabilizer of $v$ in $\Aut(\B_{\Gamma})$. This general framework is established in Section~\ref{sec:me}.

We mention that there are several subtleties for implementing the above strategy. In particular, it is important for all our arguments to know that the inclusion of $G_\Gamma$ in $\Aut(\B_\Gamma)$ is \emph{strongly ICC}, i.e.\ that the Dirac measure at $\mathrm{id}$ is the only probability measure on $\Aut(\B_\Gamma)$ which is invariant under the conjugation by every element of $G_\Gamma$. The proof of this fact relies heavily on the proximal dynamics (in the sense of Furstenberg \cite{Furs}) of the action of $G_\Gamma$ on the Roller compactification of $\B_\Gamma$, relying on tools established by Fern\'os \cite{Fer} and Kar--Sageev \cite{KS}. 
\\
\\
\textbf{Step 2: From amenable to virtually cyclic stabilizers of rank $1$ vertices.}
This is the only place where we use our integrability assumption. We use it to show that every vertex of $\B_\Gamma$ whose stabilizer for the action of $G_\Gamma$ is cyclic, also has a virtually cyclic stabilizer for the action of $H$. For this, using an argument of Escalier and the first-named author \cite{EH}, we observe that if the measure equivalence coupling $\Omega$ is $L^1$-integrable from $H$ to $G$, then the induced measure equivalence coupling $\Omega_v$ between the vertex stabilizers $G_v,H_v$ is also $L^1$-integrable from $H_v$ to $G_v\approx\mathbb{Z}$. 

At this point, if we knew that $H_v$ were finitely generated, then we could apply Bowen's theorem stating that growth is preserved by $L^1$-integrable measure equivalence \cite[Appendix~B]{Aus}, and deduce that $H_v$ is virtually cyclic. The main difficulty is that we do not know \emph{a priori} that vertex stabilizers for the $H$-action are finitely generated, even if we had assumed $H$ to be finitely generated to start with. It is therefore crucial for us to extend Bowen's theorem and prove the following, which specializes Theorem~\ref{theointro:main} to the case where $G=\mathbb{Z}$ (in fact Theorem~\ref{theo:fg} gives a slightly more precise version phrased in terms of integrable embeddings).

\begin{theointro}[{see Theorem~\ref{theo:fg}}]\label{theointro:z}
Let $H$ be a countable group with bounded finite subgroups, and assume that there is an $(L^1,L^0)$-measure equivalence coupling from $H$ to $\mathbb{Z}$.

Then $H$ is virtually cyclic.
\end{theointro}

\medskip

\noindent \textbf{Step 3: Control of factor actions and conclusion.}
At this point we have actions of $H$ and $G_\Gamma$ acting on the same complex $\B_\Gamma$ by cubical automorphisms, with all the cell stabilizers being virtually isomorphic.\footnote{In reality, we will only prove that stabilizers of rank $1$ vertices in $H$ and $G_{\Gamma}$ are virtually isomorphic, which is enough for our purposes, but this could be extended to higher-rank vertices.} However, this is far from concluding that $H$ and $G_\Gamma$ are virtually isomorphic or even quasi-isometric. 

By restricting the action $H\actson \B_\Gamma$ to rank 0 vertices of $G_\Gamma$, we obtain an action of $H$ on $G_{\Gamma}$ by \emph{flat-preserving bijections}, which means that every element $h\in H$ acts on $G_{\Gamma}$ by sending any standard flat $F$ bijectively onto another standard flat $F'$. A flat-preserving bijection is in generally not an isometry (or even a quasi-isometry) -- in fact, given a standard line $\ell$ (i.e.\ a 1-dimensional standard flat) in $G_\Gamma$, any permutation of elements of $\ell$ extends to a flat-perserving bijection of $G_\Gamma$. On the other hand, a flat-preserving bijection is a quasi-isometry if its restriction to each standard line is a quasi-isometry, with uniform constants. 
The upshot of Step 3 is to prove that the action $H\actson G_\Gamma$  coarsely preserves the order along each standard line up to a carefully chosen conjugation. As the permutation on parallel standard lines would interfere with each other, to make this precise, we use the notion of a \emph{factor action} introduced in \cite{HK}, which we now recall.

Let $\sfv\in V\Gamma^e$ be a vertex, corresponding to a cyclic subgroup $gG_v g^{-1}$, with $g\in G_\Gamma$ and $v\in V\Gamma$. The stabilizer $H_\sfv$ (for the $H$-action on $\Aut(\Gamma^e)\approx\Aut(\B_\Gamma)$) preserves the product region $P_\sfv=g(G_v\times G_{\lk(v)})$, which is also the union of all standard lines in $G_\Gamma$ with stabilizer $gG_vg^{-1}$ (here $G_{\lk(v)}$ is the subgroup generated by elements in $V\Gamma\setminus\{v\}$ that commute with $v$). Geometrically we think of $P_\sfv$ as the union of all standard lines that are parallel to $gG_v$. 
The action of $H_\sfv$ on $P_\sfv$ preserves its product decomposition, thereby inducing a \emph{factor action} of $H_\sfv$ on $Z_\sfv\approx gG_vg^{-1}$. The main ingredient of Step 3 is the following result connecting measure equivalence to quasi-isometry in our setting, established in Section~\ref{sec:factor-action}.

\medskip
\noindent
\textbf{Key property.} Let $\alpha:H\actson\B_\Gamma$ be an action obtained from an $(L^1,L^0)$-measure equivalence coupling from $H$ to $G_\Gamma$, and let $\sfv\in V\Gamma^e$. Then the factor action $H_\sfv\actson Z_\sfv$ is conjugate to an action of $H_\sfv$ on $\mathbb{Z}$ by uniform quasi-isometries.
\medskip

Once the key property is established, it follows from \cite{HK} that the action $H\actson G_\Gamma$ is conjugate to an action by uniform quasi-isometries, from which it is not hard to deduce $H$ and $G_\Gamma$ are quasi-isometric.

\subsection{Structure of the paper}

In Section~\ref{sec:prelim}, we provide background on right-angled Artin groups and associated geometric objects; in particular, we review Kim and Koberda's  extension graph, and the right-angled building, and compare their automorphism groups. In Section~\ref{sec:qi-criterion}, we review the work of Kleiner and the second-named author, and use it to establish a sufficient criterion to ensure that a group is quasi-isometric to a right-angled Artin group with finite outer automorphism group. Section~\ref{sec:me} contains general constructions and lemmas regarding measure equivalence couplings. Section~\ref{sec:proximal} uses proximal dynamics to establish the strong ICC property for $\Aut(\B_\Gamma)$. In Section~\ref{sec:action-with-amenable-stab}, we combine the measure equivalence framework and our previous work \cite{HH} to prove Theorem~\ref{theointro:me-cat}: every group that is measure equivalent to a right-angled Artin group $G$ with $|\Out(G)|<+\infty$, acts on the right-angled building of $G$ with amenable stabilizers. In Section~\ref{sec:integrability}, we exploit our integrability assumption; we prove Theorem~\ref{theointro:z} and use it to get control on the stabilizers for the $H$-action on the right-angled building of $G$. The required control on factor actions, described in Step~3 of the above sketch, is established in Section~\ref{sec:factor-action}, and we complete the proof of our main theorem (Theorem~\ref{theointro:main}) in Section~\ref{sec:proof}. Section~\ref{sec:lattice-envelope} contains our theorems regarding lattice envelopes of groups that are measure equivalent to a right-angled Artin group. Finally, we prove Theorem~\ref{theointro:nonrigidity}, regarding the lack of virtual common locally compact model for groups that are strongly commable to $G$, when $G$ splits as a product, in Section~\ref{sec:non-rigidity}.

\paragraph*{Acknowledgments.}
We thank Uri Bader, Pierre-Emmanuel Caprace, Amandine Escalier, Alex Furman, Matthieu Joseph, François Le Maître and Romain Tessera for fruitful related conversations. We are also grateful to Yves Cornulier for helpful comments on an earlier version of this paper. Finally, we thank the referee for their helpful comments on our paper.

The first-named author was funded by the European Union (ERC, Artin-Out-ME-OA, 101040507). Views and opinions expressed are however those of the authors only and do not necessarily reflect those of the European Union or the European Research Council. Neither the European Union nor the granting authority can be held responsible for them.  A CC-BY public copyright license has been applied by the
authors to the present document and will be applied to all subsequent
versions arising from this
submission, in accordance with the grant’s open access conditions.

The second-named author was funded by a Sloan fellowship.

This project was started at the Institut Henri Poincaré (UAR 839 CNRS-Sorbonne Université) during the trimester program \emph{Groups acting on fractals, Hyperbolicity and Self-similarity}. Both authors thank the IHP for its hospitality and support (through LabEx CARMIN, ANR-10-LABX-59-01).

		\section{Preliminaries on the geometry of right-angled Artin groups}\label{sec:prelim}
		
		\subsection{Right-angled Artin groups}
		
  Let $\Gamma$ be a finite simplicial graph, i.e.\ $\Gamma$ has no loop-edge and no multiple edges between vertices. The \emph{right-angled Artin group} with defining graph $\Gamma$, denoted by $G_\Gamma$, is the group defined by the following presentation:
		\begin{center}
			$\langle V\Gamma |\ [v,w]=1$ if $v$ and $w$ are joined by an edge $\rangle$.
		\end{center} 
		The reader is referred to \cite{Cha} for an introduction to right-angled Artin groups. 
		
		Recall that $\Lambda\subset\Gamma$ is an \textit{induced subgraph} if vertices of $\Lambda$ are adjacent in $\Lambda$ if and only if they are adjacent in $\Ga$. Each induced subgraph $\Lambda\subset\Ga$ yields an injective homomorphism $G_{\Lambda}\to G_\Ga$ whose image is called 
		a \emph{standard parabolic subgroup} of type $\Lambda$. A \emph{parabolic subgroup} of $G_\Gamma$ is a conjugate of some standard parabolic subgroup. A \emph{standard coset} of type $\Lambda$ is a left coset of the form $gG_{\Lambda}$.
		A \emph{standard abelian subgroup} of $G_\Gamma$ is a standard parabolic subgroup whose type is a complete subgraph. In particular, the trivial subgroup is a standard abelian subgroup, whose type is the empty set. A \emph{standard flat} of type $\Lambda$ in $G_\Ga$ is a left coset of a standard abelian subgroup of type $\Lambda$, and its \emph{dimension} is the rank of this abelian subgroup. One-dimensional standard flats are also called \emph{standard lines}. Zero-dimensional standard flats are exactly elements in $G_\Ga$.
		
		The \emph{star} of a vertex $v$ in $\Gamma$, denoted by $\st(v)$, is the induced subgraph spanned by $v$ and all its adjacent vertices. Its \emph{link} $\lk(v)$ is the induced subgraph spanned by all vertices that are adjacent to $v$. 
		
		The \emph{orthogonal} $\Lambda^\perp$ of an induced subgraph $\Lambda\subseteq \Gamma$ is the induced subgraph of $\Gamma$ spanned by all vertices in $V\Gamma\setminus V\Lambda$ which are adjacent to every vertex of $\Lambda$. For example $\lk(v)=\{v\}^\perp$. The following proposition was proved by Charney, Crisp and Vogtmann, building on work of Godelle \cite{God}.
		
		\begin{prop}[{\cite[Proposition~2.2]{CCV}}]\label{prop:normalizer}
			Let $\Gamma$ be a finite simplicial graph.
			\begin{enumerate}
				\item For every induced subgraph $\Lambda\subseteq\Gamma$, the normalizer of $G_{\Lambda}$ in $G_{\Gamma}$ is equal to $G_{\Lambda\cup \Lambda^{\perp}}$.
				\item Let $\Lambda_1,\Lambda_2\subseteq\Gamma$ be induced subgraphs. If $G_{\Lambda_1}$ and $G_{\Lambda_2}$ are conjugate, then $\Lambda_1=\Lambda_2$.
				\item Given two induced subgraphs $\Lambda_1$ and $\Lambda_2$ in $\Gamma$, if $gG_{\Lambda_1}g^{-1}\subset G_{\Lambda_2}$ for some $g\in G_\Gamma$, then there exists $h\in G_{\Lambda_2}$ such that $gG_{\Lambda_1}g^{-1}=hG_{\Lambda_1}h^{-1}$.
			\end{enumerate}
		\end{prop} 
		
		Two standard flats $F_1=g_1G_{\Lambda_1}$ and $F_2=g_2G_{\Lambda_2}$ in $G_\Gamma$ are \emph{parallel} if $\Lambda_1=\Lambda_2$ and  $g_1^{-1}g_2$ belongs to the normalizer of $G_{\Lambda_1}$ in $G_\Gamma$. Note that this definition does not depend on the choice of coset representatives $g_1$ and $g_2$. The \emph{parallel set} of a standard flat $F$, denoted $P_F$, is the union of all standard flats that are parallel to $F$. By Proposition~\ref{prop:normalizer}(1), if $F$ has type $\Lambda$, then $P_F=g G_{\Lambda\cup \Lambda^{\perp}}$ for some $g\in G_\Gamma$. The splitting $G_{\Lambda\cup \Lambda^{\perp}}=G_\Lambda\times G_{\Lambda^\perp}$ gives a splitting of the parallel set $P_F\cong F\times F^\perp$, moreover, this splitting gives a \emph{parallelism map} $p:F_1\to F_2$ between any two standard flats that are parallel. 

\paragraph*{The universal cover of the Salvetti complex.}

      The group $G_\Gamma$ acts geometrically, i.e.\ properly discontinuously and cocompactly, on a $\mathrm{CAT}(0)$ cube complex $X_\Gamma$, defined as follows. The 1-skeleton of $X_\Gamma$ is the Cayley graph $C_\Gamma$ of $G_\Gamma$ for its standard generating set $S$. It is equipped with the usual orientation and labeling of edges in a Cayley graph by elements of $S$. We then glue a square to each 4-cycle in the Cayley graph to obtain the 2-skeleton of $X_\Gamma$, then attach a $3$-cube to each copy of the boundary of a $3$-cube (with the obvious attaching maps), and more generally by induction on $k$, we glue a $k$-cube on each copy of the boundary of a $k$-cube. This process terminates after finitely many steps and results in a finite-dimensional $\mathrm{CAT}(0)$ cube complex \cite[Section~3]{CD}, on which $G_\Gamma$ acts geometrically. 
      We will identify $G_\Gamma$ with the 0-skeleton of $X_\Gamma$. 
   The \emph{Salvetti complex} of $G_\Gamma$, introduced in \cite{Sal} and denoted $S_\Gamma$, is defined to be the quotient $G_\Gamma\backslash X_\Gamma$ -- and $X_{\Gamma}$ is the universal cover of $S_\Gamma$. Note that the 2-skeleton of $S_\Gamma$ is exactly the presentation complex of $G_\Gamma$. Hence we label each edge of $S_\Gamma$ by a generator of $G_\Gamma$ and orient each edge. Each complete subgraph $\Delta$ of $\Gamma$ with $n$ vertices gives a copy of the $n$-dimensional torus in $S_\Gamma$ as a subcomplex, and $S_\Gamma$ is a union of these torus subcomplexes.

		\subsection{Extension graphs and right-angled buildings}\label{sec:extension-graph}
		
		\paragraph*{Extension graphs.} The following notion was introduced by Kim and Koberda.
		
		\begin{de}[Extension graph \cite{KK}]
			Let $\Gamma$ be a finite simplicial graph. The \emph{extension graph} $\Gamma^e$ of $\Gamma$ is the simplicial graph whose vertices are the infinite cyclic parabolic subgroups of $G_\Gamma$, where two vertices are adjacent if the corresponding parabolic subgroups commute. 
		\end{de}
		
	We emphasize that, unless $\Gamma$ is a complete graph, the extension graph $\Gamma^e$ is infinite and not locally finite (assuming $\Gamma$ is connected). The conjugation action of $G_\Gamma$ on itself induces an action of $G_\Gamma$ on $\Gamma^e$ by graph automorphisms. Every standard flat $F\subseteq G_{\Gamma}$ determines a (finite) complete subgraph $\Delta(F)\subseteq\Gamma^e$, in the following way: $\stab_{G_\Gamma}(F)$ is a parabolic subgroup, generated by a finite set of pairwise commuting infinite cyclic parabolic subgroups $\{Z_1,\dots,Z_k\}$, and we let $\Delta(F)$ to the complete subgraph of $\Gamma^e$ spanned by the vertices corresponding to $Z_1,\ldots,Z_k$. The number $k$ of vertices in $\Delta(F)$ is equal to the dimension of $F$. Given $\sfv\in V\Gamma^e$, a \emph{$\sfv$-line} in $G_\Gamma$ is a standard line $\ell$ with $\Delta(\ell)=\{\sfv\}$.
		
		Note that a collection of pairwise distinct cyclic parabolic subgroups $P_i=g_i\langle v_i\rangle g_i^{-1}$ for $1\le i\le n$ mutually commute if and only $v_i$ and $v_j$ are adjacent in $\Gamma$ for $i\neq j$ and there exists $g\in G_\Gamma$ such that for every $i\in\{1,\dots,n\}$, one has  $P_i=g\langle v_i\rangle g^{-1}$. This is a consequence of Proposition~\ref{prop:normalizer} and an induction on $n$.

		Thus, if $K$ is a complete subgraph of $\Gamma^e$, then there exists a standard flat $F=g\langle v_1\ldots,v_n\rangle$ such that $\Delta(F)=K$. 
		
		If two standard flats $F_1,F_2$ satisfy $\Delta(F_1)=\Delta(F_2)$, then they are parallel. Indeed, if $F_i=g_iG_{\Lambda_i}$ for every $i\in\{1,2\}$, then $\Delta(F_1)=\Delta(F_2)$ implies $g_1G_{\Lambda_1}g^{-1}_1=g_2G_{\Lambda_2}g^{-1}_2$. Then Proposition~\ref{prop:normalizer} implies that $F_1$ and $F_2$ are parallel. As a consequence, the map $\Delta$ from the collection of standard flats to the collection of cliques of $\Gamma^e$ induces a bijection between maximal standard flats in $G_\Ga$ and maximal cliques in $\Ga^e$. 
		
		Thus we can define $\Ga^e$ alternatively as follows. Vertices of $\Ga^e$ are in one-to-one correspondence with parallelism classes of standard lines in $G_\Ga$. Two vertices of $\Ga^e$ are adjacent if and only if there are  representatives of the associated parallelism classes that together span a $2$-dimensional standard flat.

		\paragraph*{Right-angled buildings.}
		
		Recall that an \emph{interval} in a partially ordered set $(\calp,\le)$ is a subset of the form $I_{a,b}=\{x\in\calp\mid a\le x\le b\}$ for some $a,b\in\calp$ with $a\le b$. If every interval in $\calp$ is a Boolean lattice of finite rank, then there exists a unique (up to isomorphism) cube complex $|\calp|$ whose poset of cubes is isomorphic to the poset of intervals of $\calp$, see e.g.\ \cite[Proposition~A.38]{AB}. We call $|\calp|$ the \emph{cubical realization} of $\calp$. 
		
		We will be particularly interested in the case when $\mathcal P$ is the poset of standard flats in $G_\Gamma$, ordered by inclusion. Note that if we take $g_1G_{\Lambda_1}\le g_2G_{\Lambda_2}$ in $\mathcal P$, then we can assume that $g_2=g_1$ up to changing coset representatives, and the interval between $g_1G_{\Lambda_1}$ and $g_2G_{\Lambda_2}$ consists of all cosets of form $g_1 G_{\Lambda}$ with $\Lambda_1\subset \Lambda\subset \Lambda_2$. In particular it is a Boolean lattice of finite rank. The following notion was introduced by Davis in \cite{Dav}.
		
		\begin{de}[Right-angled building]
			\label{def:building}
			Let $\Gamma$ be a finite simplicial graph. The \emph{right-angled building} of $G_\Gamma$, denoted by $\B_\Gamma$, is the cubical realization of the poset of standard flats in $G_\Gamma$. The left action of $G_\Gamma$ on itself induces an action of $G_\Gamma$ on $\B_\Gamma$ by cubical automorphisms.
		\end{de}
		
		In Figure~\ref{fig:building} below, we draw a subcomplex of $\B_\Gamma$ when $\Gamma$ is a pentagon with consecutive vertices denoted by $a,b,c,d,e$. The vertex in Figure~\ref{def:building} labeled by $ab$ corresponds to the identity left coset of the subgroup generated by $a$ and $b$. Similarly, we define other vertices in Figure~\ref{def:building}. The central vertex corresponds to the identity element in $G_\Gamma$. The subcomplex of $\B_\Gamma$ displayed in Figure~\ref{def:building} is a fundamental domain for the action $G_\Gamma\actson \B_\Gamma$.
		\begin{figure}[h]
	\centering
	\includegraphics[scale=0.9]{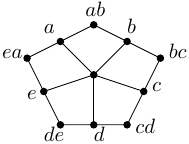}
	\caption{Fundamental domain of the pentagon building.}
	\label{fig:building}
\end{figure}			
  
  \begin{rk}
  \label{rk:building}
We comment on the terminology ``right-angled building'', even though we will not explicitly need the connection to buildings in the rest of the paper.  Davis \cite{Dav} explained that we can view the right-angled Artin group $G_\Gamma$ as a building in the sense of Tits (see \cite[Section 3]{Dav}) modeled on a reflection group $W_\Gamma$ which is the right-angled Coxeter group with the same defining graph. Moreover, $\B_\Gamma$ is a geometric model witnessing some classical geometric properties of buildings, for example, any two points in $\B_\Gamma$ are contained in a common ``apartment'' which is a convex subcomplex isomorphic to the canonical $\mathrm{CAT}(0)$ cube complex (called the \emph{Davis complex}) associated with $W_\Gamma$. We refer to \cite[Section~5]{Dav} for a more detailed discussion.
  \end{rk}

		By \cite[Corollary 11.7]{Dav} (attributed to Meier), the cube complex $\B_\Gamma$ is $\mathrm{CAT}(0)$. There is a one-to-one correspondence between $k$-cubes in $\B_\Gamma$ and intervals of the form $I_{F_1,F_2}$ where $F_1\subseteq F_2$ are two standard flats in $\Gamma$ with $\dim(F_2)-\dim(F_1)=k$. In particular, vertices of $\B_\Gamma$ correspond to standard flats in $G_\Gamma$. The \emph{rank} of a vertex $v\in V\B_\Gamma$ is the dimension of the corresponding standard flat. We label every vertex of $\B_\Gamma$ by the type of the corresponding standard flat. The vertex set of $\B_\Gamma$ inherits a partial order from the poset of standard flats. 
		
		The action of $G_\Gamma$ on $\B_\Gamma$ preserves the labellings of vertices. This action is cocompact, but not proper -- the stabilizer of a cube is isomorphic to $\mathbb Z^{n}$ where $n$ is the rank of the minimal vertex in this cube.
		
		Recall that the \emph{join} of two simplicial graphs $\Gamma_1$ and $\Gamma_2$, denoted $\Gamma_1\circ \Gamma_2$, is the simplicial graph obtained from the disjoint union $\Gamma_1\sqcup \Gamma_2$ by adding an edge between any vertex in $\Gamma_1$ and any vertex in $\Gamma_2$.
		It follows from the definition that if $\Gamma=\Gamma_1\circ \Gamma_2$, then $\B_\Gamma\cong \B_{\Gamma_1}\times \B_{\Gamma_2}$.

\begin{lemma}\label{lemma:rank}
			Every cubical automorphism of $\B_\Gamma$ preserves ranks of vertices.
		\end{lemma}

\begin{proof}
Let $x\in V\B_\Gamma$ be a vertex. Recall that the \emph{link} of $x$ in $\B_\Gamma$, denoted by $\lk(x,\B_\Gamma)$, is the simplicial complex formed by intersecting an $\varepsilon$-sphere around $x$ with $\B_\Gamma$, for $\varepsilon>0$ small enough. In particular, there is a bijection between vertices of $\B_\Gamma$ which are adjacent to $x$ and vertices in $\lk(x,\B_\Gamma)$. We claim that any two vertices of $\B_\Gamma$ with different ranks must have non-isomorphic links in $\B_\Gamma$, which implies the lemma.

Now we prove the claim.
Let $\lk^+(x,\B_\Gamma)$ (resp.\ $\lk^-(x,\B_\Gamma)$) be the full subcomplex of $\lk(x,\B_\Gamma)$ spanned by vertices which are larger than $x$ (resp.\ smaller than $x$). A simplex in $\lk^+(x,\B_\Gamma)$ corresponds to the corner of a cube $C^+\subset \B_\Gamma$ corresponding to an interval of form $[x,y]$ for some vertex $y\ge x$. A simplex in $\lk^-(x,\B_\Gamma)$ corresponds to the corner of a cube $C^-\subset \B_\Gamma$ corresponding to an interval of form $[z,x]$ for some vertex $z\le x$. Then $C^+$ and $C^-$ span a bigger cube in $\B_\Gamma$ isometric to $C^+\times C^-$ corresponding to the interval $[z,y]$.
Thus $\lk(x,\B_\Gamma)$ is the join of $\lk^+(x,\B_\Gamma)$ and $\lk^-(x,\B_\Gamma)$. Note that $\lk^+(x,\B_\Gamma)$ is a finite complex. Suppose $x$ has rank $k$ and corresponds to a standard flat $F$ of dimension $k$. Vertices in $\lk^-(x,\B_\Gamma)$ are in one-to-one correspondence with co-dimension 1 standard flats in $F$. Thus these vertices can be divided into $k$ different classes $V_1,\dots,V_k$, corresponding to the $k$ parallelism classes of co-dimension 1 standard flats in $F$. Each $V_i$ is an infinite set. Moreover, $\lk^-(x,\B_\Gamma)$ is a join of $k$ discrete sets $V_1*V_2*\cdots* V_k$. Thus the claim follows.
\end{proof}

  \subsection{Automorphisms of the extension graph and the right-angled building}\label{sec:autos}
  %\Ccomm{add a sentence about the topology} 
  We endow $\Aut(\B_\Gamma)$ and $\Aut(\Gamma^e)$ with the compact-open topology, making them Polish groups.		
		A bijection $f:G_\Gamma\to G_\Gamma$ is \emph{flat-preserving} if both $f$ and $f^{-1}$ send any standard flat bijectively onto another standard flat. Let $\Bij_{\fp}(G_\Gamma)$ be the group of flat-preserving bijections of $G_\Gamma$,  again equipped with the compact-open topology, or equivalently the topology of pointwise convergence, which makes it a Polish group. Lemma~\ref{lemma:rank} implies that for every $f\in\Aut(\B_\Gamma)$, the restriction of $f$ to the set of rank $0$ vertices is a flat-preserving bijection of $G_\Gamma$. Conversely, any flat-preserving bijection of $G_\Gamma$ induces an automorphism of the poset of standard flats, hence an automorphism of $\B_\Gamma$. This yields an isomorphism of Polish groups $\Aut(\B_\Gamma)\simeq \Bij_{\fp}(G_\Gamma)$. We now explain how to identify $\Aut(\B_\Gamma)$ to $\Aut(\Gamma^e)$ when $|\Out(G_\Gamma)|<+\infty$, as will be recorded in Lemma~\ref{lemma:iso} below.

\paragraph{From $\Aut(\B_\Gamma)$ to $\Aut(\Gamma^e)$.}  We define a map $\Phi:\Aut(\B_\Gamma)\to\Aut(\Gamma^e)$ as follows. Take $\alpha\in\Aut(\B_\Gamma)$. The restriction of $\alpha$ to the set of rank $0$ vertices of $\B_\Gamma$ is a flat-preserving bijection $g:G_\Gamma\to G_\Gamma$. We claim that $g$ sends parallel standard lines to parallel standard lines. Indeed, let $\ell_1$ and $\ell_2$ be parallel standard lines. If $\ell_1$ and $\ell_2$ are contained in a common $2$-dimensional standard flat, then the claim is obvious. In general, it follows from Proposition~\ref{prop:normalizer} that there is a finite chain of standard lines starting from $\ell_1$ and ending at $\ell_2$ such that consecutive members in the chain are parallel and contained in a common 2-dimensional standard flat. The claim thus follows. 
 
The above claim implies that $g$ induces a bijection of the set of parallelism classes of standard lines of $G_\Gamma$. Recall that vertices of $\Gamma^e$ correspond to parallelism classes of standard lines. Thus $g$ induces a bijection of $V\Gamma^e$, and by construction this bijection preserves adjacency. We let $\Phi(\alpha)$ be this automorphism of $\Gamma^e$. 

\paragraph{From $\Aut(\Gamma^e)$ to $\Aut(\B_\Gamma)$.} Conversely, assuming that $|\Out(G_\Gamma)|<+\infty$, we now build a map $\Theta:\Aut(\Gamma^e)\to\Aut(\B_\Gamma)$, which will be an inverse to $\Phi$.

Let $\alpha\in\Aut(\Gamma^e)$. Let $p\in G_\Gamma$, and let $F_1,\dots,F_n$ be the maximal standard flats that contain $p$. Each standard flat $F_i$ corresponds to a maximal clique $C_i$ in $\Gamma^e$, and $\alpha(C_i)$ in turn corresponds to a unique maximal standard flat $F'_i$ of $G_\Gamma$. As $|\Out(G_\Gamma)|<+\infty$, it follows from \cite[Lemmas~4.12 and~4.17]{Hua} that $F'_1\cap\dots\cap F'_n$ is a singleton $\{p'\}$. Letting $\alpha_\ast(p)=p'$ defines a map $\alpha_\ast:G_\Gamma\to G_\Gamma$. By construction, the $\alpha_{\ast}$-image of any maximal standard flat is contained in a maximal standard flat. By considering $(\alpha^{-1})_\ast$, we see that in fact $\alpha_\ast$ sends every maximal standard flat bijectively onto another standard flat.

We claim that in fact $\alpha_\ast$ is flat-preserving. For this it is enough to prove that it sends standard lines to standard lines -- an inductive argument then shows it for standard flats of higher dimension. As $|\Out(G_\Gamma)|<+\infty$, every vertex $v\in V\Gamma$ is the intersection of the maximal cliques in $\Gamma$ containing $v$ (otherwise the link of $v$ would be contained in the star of another vertex). As standard flats containing a standard line of type $v$ are in one-to-one correspondence with cliques in $\Gamma$ that contain $v$, every standard line $\ell$ is the intersection of all the maximal standard flats containing $\ell$. Hence $\alpha_*$ and $(\alpha^{-1})_*$ send standard lines bijectively onto standard lines, as claimed. 

We now let $\Theta(\alpha)=\alpha_\ast$, viewed as an automorphism of $\B_\Gamma$. 

One readily verifies that $\Phi$ and $\Theta$ are continuous group homomorphisms, and that they are inverses of each other, which we record in the following statement.

\begin{lemma}
\label{lemma:iso}
Suppose that $|\Out(G_\Gamma)|<+\infty$. Then $\Phi:\Aut(\B_\Gamma)\to \Aut(\Gamma^e)$ is an isomorphism of Polish groups with inverse $\Theta$, and $\Phi(G_\Gamma)=G_\Gamma$ under the natural embeddings of $G_\Gamma$ in $\Aut(\B_\Gamma)$ and in $\Aut(\Gamma^e)$. \qed
\end{lemma}

\begin{rk}\label{rk:iso}
 In the sequel, the last part of the lemma will be used in the following way. Letting $G_\Gamma\times G_\Gamma$ act by left/right multiplication on both $\Aut(\B_\Gamma)$ and $\Aut(\Gamma^e)$, the map $\Phi$ is $(G_\Gamma\times G_\Gamma)$-equivariant. 
\end{rk}

		\section{Blow-up buildings and a quasi-isometry criterion}\label{sec:qi-criterion}
	
		In this section, we review work of Kleiner and the second named author \cite{HK} and use it to provide a criterion ensuring that a group is quasi-isometric to a right-angled Artin group $G=G_\Gamma$ with $|\Out(G)|<+\infty$, see Theorem~\ref{theo:QI}. 

  	\paragraph*{Factor actions.} Take a vertex $\sfv\in V\Gamma^e$, and let $P_\sfv$ be the union of all $\sfv$-lines in $G_\Gamma$.  Proposition~\ref{prop:normalizer} implies that $P_\sfv$ is a left coset of the form $gG_{\st(v)}$ for some $v\in V\Gamma$. Let $Z_\sfv$ be the collection of left cosets of $G_{\lk(v)}$ in $P_\sfv$ and $\mathcal L_\sfv$ be the collection of $\sfv$-lines (i.e.\ left cosets of $G_\sfv$ in $P_\sfv$). There are natural projections $\pi_1: P_\sfv\to Z_\sfv$ and $\pi_2:P_\sfv\to \mathcal L_\sfv$, and we can identify $P_\sfv$ and $Z_\sfv\times \call_\sfv$ via the bijection $(\pi_1,\pi_2)$.  
  
%  The splitting $G_{\st(v)}=\langle v\rangle\times G_{\lk(v)}$ induces a direct product decomposition
%		$P_\sfv=g\langle v\rangle\times gG_{\lk(v)}$. Let 
  
 % Let $Z_\sfv=g\langle v\rangle$ and $Q_\sfv=gG_{\lk(v)}$.
		
		Let $H$ be a group. Any action $\alpha: H\to\Aut(\B_\Gamma)$ induces an $H$-action by flat-preserving bijections on $G_\Gamma$, as well as an $H$-action on $\Gamma^e$ by graph automorphisms, as explained in Section~\ref{sec:autos}.
  %As flat-preserving bijections send parallel standard lines to parallel standard lines, there is also an induced $H$-action on $\Gamma^e$ by graph automorphisms. 
  Let $\sfv\in V\Gamma^e$ and let $H_\sfv$ be its $H$-stabilizer. Then $H_\sfv$ preserves $P_\sfv$, and acts on it by flat-preserving bijections sending $\sfv$-lines to $\sfv$-lines. As a consequence, the $H_\sfv$-action on $P_\sfv$ preserves the product decomposition $Z_\sfv\times \call_\sfv$ described above. In particular, there is an induced action $\alpha_\sfv:H_\sfv\to\Bij(Z_\sfv)$, where $\Bij(Z_\sfv)$ is the group of all bijections of $Z_\sfv$. 

The following notion, from \cite[Definition~5.32]{HK}, will be crucial in the present work.

		\begin{de}[Factor action]
			\label{def:factor action}
		Given an action $H\actson\B_{\Gamma}$ of a group $H$ by cubical automorphisms, and $\sfv\in V\Gamma^e$, the induced action $\alpha_\sfv:H_\sfv\to\Bij(Z_\sfv)$ is called the \emph{factor action} of $\alpha$ associated to $\sfv$.
		\end{de}

The goal of the present section is to derive the following theorem from the work of Kleiner and the second named author \cite{HK}.

  		\begin{theo}
			\label{theo:QI}
	 Let $G=G_\Gamma$ be a right-angled Artin group with $|\Out(G)|<+\infty$. Let $H$ be a group. Assume that $H$ has an action $\alpha:H\actson G$ by flat-preserving bijections satisfying the following conditions: 
			\begin{enumerate}
				\item the action has finitely many orbits and finite stabilizers;
				%\item the induced action $H\actson V\Gamma^e$ has finitely many orbits;
				\item for every $\sfv\in V\Gamma^e$, the factor action $\alpha_\sfv:H_\sfv\actson Z_\sfv$ is conjugate to an action on $\mathbb{Z}$ by uniform quasi-isometries.
			\end{enumerate}
			Then $H$ is finitely generated and quasi-isometric to $G$. 
   
   Moreover, $H$ acts geometrically (i.e.\ properly and cocompactly by automorphisms) on a CAT(0) cube complex $Y$ with an $H$-equivariant surjective cubical map $Y\to \B_\Gamma$ (for the $H$-action on $\B$ induced by $\alpha$) such that the preimage of any rank $k$ vertex is isomorphic to the Euclidean space $\mathbb E^k$ with its usual cubulation.
\end{theo}

The cube complex $Y$ has a more explicit description, see Section~\ref{subsec:blowup building}. It is closely related to the canonical cube complex $X_\Gamma$ associated with $G_\Gamma$ (the universal cover of its Salvetti complex) in the sense that $Y$ is obtained from $X_\Gamma$ by replacing each standard flat in $X_\Gamma$ by what we call a branched flat, and gluing these branched flats in a similar pattern as how standard flats in $X_\Gamma$ are glued together. 

In the rest of this section, we will first give a more precise description of blow-up buildings, then prove Theorem~\ref{theo:QI}.

  \subsection{Blow-up buildings}
  \label{subsec:blowup building}
		\begin{de}[Branched lines and flats]
			A metric simplicial graph $\beta$ is a \emph{branched line} (see Figure~\ref{fig:branched} left) if there exists $C>0$ such that $\beta$ is obtained from $\mathbb{R}$ (equipped with its simplicial structure given by subdividing at integer points) by gluing, at every integer $n\in\mathbb{Z}$, at most $C$ edges and at least one edge of length $1$, denoted $e_{n,1},\dots,e_{n,k}$, gluing the origin of each $e_{n,i}$ at $n$.  
			
			Valence one vertices of $\beta$ are called the \emph{tips} of $\beta$, and their set is denoted by $t(\beta)$. The copy of $\mathbb{R}$ is called the \emph{core} of $\beta$.
			
			A \emph{branched flat} $F$ is a product of finitely many branched lines $\beta_1,\dots,\beta_k$. A \emph{tip} of $F$ is a tuple $(t_1,\dots,t_k)$, where each $t_i$ is a tip of $\beta_i$; we denote their set by $t(F)$. The core of $F$ is the product of the cores of the $\beta_i$.
		\end{de}
	
		\begin{figure}[h]
	\centering
	\includegraphics[scale=0.7]{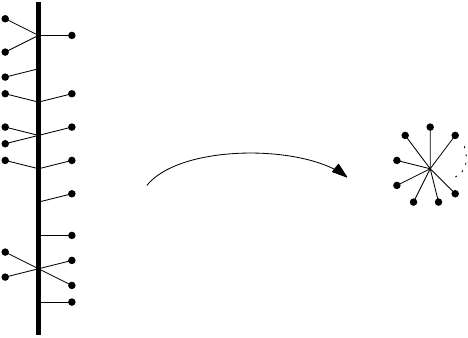}
	\caption{A branched line and a projection map.}
	\label{fig:branched}
\end{figure}	
		
		The following construction is a special case of \cite[Sections 5.2 and~5.3]{HK}, see also \cite[Section~3.1]{Hua-QI}. 
		
		\begin{de}[Blow-up datum]
			\label{definition:blow up}
			A \emph{blow-up datum} is a family of surjections $(g_\ell:\ell\to\mathbb{Z})_\ell$, where $\ell$ varies over the set of standard lines of $G_\Gamma$, such that
			\begin{enumerate}
				\item whenever $\ell_1,\ell_2$ are parallel, with parallelism map $p:\ell_2\to\ell_1$, then $g_{\ell_2}=g_{\ell_1}\circ p$;
				\item for every $\ell$, there exists $C_\ell>0$ such that $|g_\ell^{-1}(n)|\le C_\ell$ for every $n\in\mathbb{Z}$.
			\end{enumerate}
		We say that the blow-up datum $(g_\ell)_\ell$ is \emph{uniformly locally finite} if the constant $C_\ell$ can be chosen independent of $\ell$.
		\end{de}

 Recall that in our convention, a standard line is defined as a subset of $G_\Gamma$ -- in particular it is discrete.

Associated to any blow-up datum $(g_\ell)_\ell$ is a family $(\beta_\sfv)_{\sfv\in V\Gamma^e}$ of branched lines, and a family of maps $(f_\ell)_\ell$, defined in this way. For every $\sfv\in V\Gamma^e$, we first choose a $\sfv$-line $\ell_\sfv$. We then let $\beta_\sfv$ be the simplicial graph $(\ell_\sfv\times [0,1])\sqcup \mathbb R/{\sim}$, where $(x,1)\sim g_{\ell_\sfv}(x)$ for any $x\in \ell_\sfv$. The inclusion map $\ell_\sfv\to\beta_\sfv$ yields a bijection $f_{\ell_\sfv}:\ell_\sfv\to t(\beta_\sfv)$. Now, for every standard line $\ell$, denoting by $\sfv$ the type of $\ell$, there is a parallelism map $p:\ell\to\ell_\sfv$, and we let $f_{\ell}=f_{\ell_\sfv}\circ p$. We say that the family $(f_\ell)_\ell$ of bijections constructed in this way is \emph{adapted} to the blow-up datum $(g_\ell)_\ell$.

		\paragraph*{Blow-up buildings.}
		
		Let $(g_\ell)_\ell$ be a blow-up datum, and let $(f_\ell)_\ell$ be an adapted family of bijections. We now associate to $(g_\ell),(f_\ell)$ a cube complex $Y$, as follows. % (depending on the blow up data) as follows.
		First, to every standard flat $F\subseteq G_\Gamma$, we associate a space $\beta_F$ as follows:
		\begin{enumerate}
			\item if $\Delta(F)=\emptyset$ (i.e.\ $F$ is a 0-dimensional standard flat), we let $\beta_F$ be a point;
			\item if $\Delta(F)\neq\emptyset$, writing $F=\prod_{\sfv\in V(\Delta(F))}\ell_{\sfv}$, where each $\ell_{\sfv}\subset F$ is a standard $\sfv$-line, we let  $\beta_{F}=\prod_{\sfv\in V(\Delta(F))}\beta_{\sfv}$.
		\end{enumerate}
		Whenever $F'\subseteq F$ are two standard flats, we can write \[F'=\prod_{\sfv\in V(\Delta(F'))}\ell_{\sfv}\times\prod_{\sfv\in V(\Delta(F))\setminus V(\Delta(F'))}\{x_{\sfv}\},\] where each $x_{\sfv}$ is a vertex in $\ell_{\sfv}$. Then we define an isometric embedding $\beta_{F'}\hookrightarrow \beta_F$ as follows:
		\[\beta_{F'}=\prod_{\sfv\in V(\Delta(F'))}\beta_\sfv\cong \prod_{\sfv\in V(\Delta(F'))}\beta_{\sfv}\times\prod_{\sfv\in V(\Delta(F)\setminus \Delta(F'))}\{f_{\ell_{\sfv}}(x_{\sfv})\}\hookrightarrow \prod_{\sfv\in V(\Delta(F))}\beta_{\sfv}=\beta_F.\]
		
		\begin{de}[Blow-up building]
			The space $Y$ obtained from the disjoint union of the branched lines $\beta_F$ by identifying $\beta_{F'}$ as a subset of $\beta_F$ whenever $F'\subseteq F$, according to the above isometric embeddings, is called the \emph{blow-up building} associated to $(g_\ell)_\ell$, $(f_\ell)_\ell$.
		\end{de}

\subsection{Properties of blow-up buildings}
		We now define a projection map $\pi:Y\to \B_\Gamma$. Note that for each standard line $\ell\subset G$, we can define a map $\pi:\beta_\ell\to \B_\Gamma$ by sending the core of $\beta_\ell$ to the rank 1 vertex in $\B_\Gamma$ associated with $\ell$, sending each vertex in $t(\beta_\ell)$ to the associated rank $0$ vertex in $\B_\Gamma$, and extending linearly (see Figure~\ref{fig:branched}). More generally, let $F=\prod_{\sfv\in V(\Delta(F))}\ell_{\sfv}$ be a standard flat, where each $\ell_{\sfv}\subset F$ is a standard $\sfv$-line, and let  $\beta_{F}=\prod_{\sfv\in V(\Delta(F))}\beta_{\sfv}$ be the associated branched flat. We define $\pi:\beta_F\to\B_\Gamma$ as follows. Every vertex $x\in\beta_F$ lies in the core of a unique subcomplex of the form $\beta_{F'}$, with $F'\subseteq F$, and we let $\pi(x)$ be the vertex of $\B_\Gamma$ associated to $F'$. 
  These maps defined on each $\beta_F$ are compatible with the gluing pattern, hence induce a map $\pi:Y\to \B_\Gamma$. Note that the restriction of $\pi:Y\to\B_\Gamma$ to each cube is either an isometry or collapses the cube to a cube of smaller dimension (by collapsing some of the interval factors). We say that a vertex $y\in Y$ has \emph{rank} $k$ if $\pi(y)$ has rank $k$.
		We record the following properties of $Y$. 
		\begin{enumerate}
			\item The natural map $\beta_F\to Y$ is injective for each standard flat $F\subset G$, see \cite[Lemma~5.16]{HK}. The image of this embedding is called a \textit{standard branched flat}. From now on we slightly abuse notation and again denote by $\beta_F$ its image in $Y$. The core of a standard branched flat is called a \textit{standard flat}. The map sending a standard flat $F$ to the core of $\beta_F$ is a one-to-one correspondence between standard flats in $G$ and standard flats in $Y$.
			\item Given any two standard flats $F_1,F_2$ of $G$, one has $\beta_{F_1}\cap \beta_{F_2}=\beta_{F_{1}\cap F_{2}}$, see \cite[Lemma 8.1]{HK}. Thus if the cores of $\beta_{F_1}$ and $\beta_{F_2}$ have nontrivial intersection, then $\beta_{F_1}=\beta_{F_2}$: indeed, if $F'\subsetneq F$, then $\beta_{F'}$ does not intersect the core of $\beta_F$; so we must have $F_1\cap F_2=F_1=F_2$. In particular, different standard flats in $Y_\Ga$ are disjoint.
			\item There exists a unique injective map $f:G\to Y$ whose restriction to any standard line $\ell$ coincides with $f_\ell$. This map $f$ sends the vertex set of each standard flat of $G$ bijectively to the tips of a standard branched flat. The image of $f$ is exactly the set of $0$-dimensional standard flats in $Y$. 
		\end{enumerate}
		
In fact, the space $Y$ is a $\mathrm{CAT}(0)$ cube complex by \cite{HK}.

		\begin{lemma}[{\cite[Corollary 5.30]{HK}}]
			\label{lem:qi}
			Let $(g_\ell)_\ell$ be a uniformly locally finite blow-up datum, let $(f_\ell)_\ell$ be an adapted family of bijections, and let $Y$ be the associated blow-up building. 
			
			Then $G$ and $Y$ are quasi-isometric.
		\end{lemma}

		Each $\sfv$-line $\ell\subset G_\Gamma$ has a canonical identification with $Z_\sfv$. So the factor action $\alpha_\sfv:H_\sfv\actson Z_\sfv$ can also be viewed as an action $\alpha_{\sfv,\ell}:H_\sfv\actson\ell$.

		\begin{de}\label{de:compatible}
			Let $(g_\ell)_\ell$ be a blow-up datum, and $(f_\ell)_\ell$ be an adapted family of bijections. Let $\alpha:H\actson G_\Gamma$ be an action of a group $H$ by flat-preserving bijections. We say that $\alpha$ and $(g_\ell),(f_\ell)$ are \emph{compatible} if there exists a family of isometric actions $(\alpha'_{\sfv}:H_\sfv\to\Isom(\beta_\sfv))_{\sfv\in V\Gamma^e}$, such that 
			\begin{enumerate}
				\item for each $\sfv$-line $\ell$, the map $f_{\ell}:\ell\to t(\beta_\sfv)\subset \beta_\sfv$ is $(\alpha_{\sfv,\ell},\alpha'_\sfv)$-equivariant;
				\item if $h\in H$ sends a $\sfv$-line $\ell$ to a $\sfw$-line $\ell'$, then the map $f_{\ell'}\circ h\circ f^{-1}_{\ell}:t(\beta_\sfv)\to t(\beta_\sfw)$ extends to an $(\alpha'_\sfv,\alpha'_\sfw)$-equivariant isometry between $\beta_\sfv$ and $\beta_\sfw$.
			\end{enumerate}
		\end{de}
		
		Since flat-preserving bijections of $G_\Gamma$ are naturally in one-to-one correspondence with cubical automorphisms of $\B_\Gamma$, we will also say that an action $\alpha:H\to\Aut(\B_\Gamma)$ is \emph{compatible} with $(g_\ell),(f_\ell)$ if the corresponding $H$-action on $G_\Gamma$ by flat-preserving bijections is. The following lemma is a consequence of \cite[Lemma~5.25]{HK}.
		
		\begin{lemma}
			\label{lem:action}
			Let $\alpha:H\actson \B_\Gamma$ be an action by cubical automorphisms. Let $(g_\ell)_\ell$ be a blow-up datum, let $(f_\ell)_\ell$ be an adapted family of bijections, and let $Y$ be the associated blow-up building. 
			
			If $\alpha$ is compatible with $(g_\ell),(f_\ell)$, then there exists an action $\alpha':H\actson Y$ by cellular isometries such that the map $\pi:Y\to\B_\Gamma$ is $(\alpha',\alpha)$-equivariant.
		\end{lemma}

 \subsection{Proof of the quasi-isometry criterion} 
		\begin{proof}[Proof of Theorem~\ref{theo:QI}]
			The first step is to choose a blow-up datum which is compatible with the $H$-action.

   We claim that the action $H\actson V\Gamma^e$ has finitely many orbits. Indeed, as each point of $G$ is contained in finitely many standard lines, and the action $\alpha$ is flat-preserving, the action $\alpha:H\actson G$ has finitely many orbits of standard lines. Hence the claim follows.
   
   Let $\{\sfv_1,\dots,\sfv_n\}$ be a (finite) set of representatives of the orbits of vertices for the action $H\actson \Gamma^e$. By assumption, the action $\alpha_{\sfv_i}$ is conjugate to an action on $\mathbb{Z}$ by uniform quasi-isometries. Therefore, by \cite[Proposition~6.3]{HK}, it is semi-conjugate to an action on $\mathbb{Z}$ by isometries. More precisely, there exist an isometric action $\gamma_{\sfv_i}:H_{\sfv_i}\actson \mathbb Z$, an $(\alpha_{\sfv_i},\gamma_{\sfv_i})$-equivariant surjection $g_i: Z_{\sfv_i}\to \mathbb Z$, and $C_i>0$ satisfying $|g^{-1}_i(n)|\le C_i$ for every $n\in\mathbb{Z}$.  
			
			For every standard line $\ell$, we now define a map $g_\ell:\ell\to\mathbb{Z}$, following the construction in \cite[Section~5.6]{HK}. First, if $\ell$ is a $\sfv_i$-line for some $i\in\{1,\dots,n\}$, we let $p_i:P_{\sfv_i}\to  Z_{\sfv_i}$ be the projection map, and let $g_\ell=(g_i\circ p_i)_{|\ell}$. In general $\ell$ is a $\sfw$-line for some $\sfw\in V\Gamma^e$. For every $\sfw\in V\Gamma^e$, we choose an element $h_\sfw\in H$ such that $h_\sfw\sfw=\sfv_i$ for some $i\in\{1,\dots,n\}$. %We first let $h\in H$ be such that $h(\sfw)=\sfv_i$ for some $i\in\{1,\dots,n\}$. 
   Now if $\ell$ is a $\sfw$-line, then $h_\sfw(\ell)$ is a $\sfv_i$-line, and we let $g_{\ell}=g_{h_\sfw(\ell)}\circ h_\sfw$.  As the maps $g_i$ are finite-to-one, the family $(g_\ell)_\ell$ is a uniformly locally finite blow-up datum.
			
	Let $(\beta_\sfv)_{\sfv\in V\Gamma^e}$ be the family of branched lines associated to the blow-up datum $(g_\ell)_\ell$, and let $(f_{\ell}:\ell\to t(\beta_\sfv))_\ell$ be an adapted family of bijections. For every $i\in\{1,\dots,n\}$, the equivariance property of the map $g_i$ gives an isometric action $\alpha'_{\sfv_i}:H_{\sfv_i}\actson \beta_{\sfv_i}$. More generally, for every $\sfw\in V\Gamma^e$, there is an isometric action $\gamma_\sfw:H_\sfw\actson\mathbb{Z}$ given by $\gamma_\sfw(h)(z)=\gamma_{\sfv_i}(h_\sfw hh_\sfw^{-1})(z)$ (where $i$ is such that $h_\sfw\sfw=\sfv_i$). Then for every $\sfw$-line $\ell$, the map $g_\ell:\ell\to\mathbb{Z}$ constructed above is $(\alpha_{\sfw,\ell},\gamma_\sfw)$-equivariant. So it yields an isometric action $\alpha'_\sfw:H_\sfw\actson \beta_\sfw$, such that the map $f_\ell:\ell\to t(\beta_\sfw)$ is $(\alpha_{\sfw,\ell},\alpha'_\sfw)$-equivariant. One also checks that the second compatibility condition in Definition~\ref{de:compatible} is satisfied. 
			
			Let $Y$ be the blow-up building associated to $(g_\ell)_\ell,(f_\ell)_\ell$. By Lemma~\ref{lem:qi}, the space $Y$ is quasi-isometric to $G$. By Lemma~\ref{lem:action}, there is an action $\alpha':H\actson Y$ by cellular isometries, such that the map $\pi:Y\to\B_\Gamma$ is $(\alpha',\alpha)$-equivariant.
			
			We now prove that the action $\alpha'$ is proper and cocompact. This will imply that $H$ is finitely generated and quasi-isometric to $G$, as desired. By definition of $\pi$, every rank $0$ vertex of $\B_\Gamma$ has a unique preimage under $\pi$. As the map $\pi:Y\to \B_\Gamma$ is $(\alpha',\alpha)$-equivariant, it follows that the action $\alpha'$ has finitely many orbits of rank 0 vertices. Note that every vertex $y$ of $Y$ of rank at least $1$ is adjacent to at least one vertex of lower rank (this follows by considering the standard flat containing $y$ and the associated standard branched flat). Thus there exists $C>0$ such that any vertex in $Y$ is at most distance $C$ from a rank 0 vertex. As $Y$ is uniformly locally finite, it follows that $\alpha'$ has finitely many orbits of vertices.  Using again that $Y$ is uniformly locally finite, this is enough to ensure that there are only finitely many orbits of cells, so $\alpha'$ is cocompact. As $H$ acts cocompactly on a uniformly locally finite complex, to show that the action is proper, it suffices to show that the stabilizer of each vertex is finite. The case of rank 0 vertices follows from the first assumption of the theorem, the equivariance of $\pi$, and the fact that $\pi$ is a bijection between rank $0$ vertices of $Y$ and $\B_\Gamma$. The equivariance of $\pi$ ensures that the $H$-action on $Y$ preserves the rank of vertices. Therefore, the stabilizer of each vertex $y\in Y$ of rank at least $1$  permutes the non-empty finite set of vertices of lower rank which are adjacent to $y$. By induction on the rank, we thus deduce that the stabilizer of every vertex is finite.

		The moreover part of the theorem follows from \cite[Corollary 6.5]{HK}. Actually, the cube complex is exactly $Y$.
		\end{proof}

		\section{Measure equivalence couplings}\label{sec:me}

In this section, we first review the definition and framework of measure equivalence and couplings. We then establish a few general statements that will be specialized to the context of right-angled Artin groups in later sections.

\subsection{Review on measure equivalence}\label{sec:intro-me}

Recall from the introduction that a \emph{measure equivalence coupling} between two countable groups $\G$ and $\sfH$ is a standard measure space $\Omega$ (of positive measure) equipped with a measure-preserving action of $\G\times \sfH$ such that both factor actions $\G\actson \Omega$ and $\sfH\actson \Omega$ are free and have a finite measure fundamental domain. Here, a \emph{fundamental domain} for the action of $\G$ on $\Omega$ is a Borel subset $X_\G\subseteq\Omega$ such that $\G\cdot X_\G=\Omega$ up to null sets, and for every nontrivial element $g\in \G$, the intersection $X_\G\cap gX_\G$ is a null set. Two countable groups $\G,\sfH$ are \emph{measure equivalent} if there exists a measure equivalence coupling between $\G$ and $\sfH$. This turns out to be an equivalence relation on the set of countable groups, see \cite[Section~2]{Fur-me}.

 There is a dictionary between measure equivalence and stable orbit equivalence, that was established by Furman \cite{Fur-oe}. Let us briefly mention what we will need from this dictionary. Let $\Omega$ be a measure equivalence coupling between two countable groups $\mathsf{G}$ and $\mathsf{H}$, and let $X_{\mathsf{G}},X_{\mathsf{H}}$ be respective fundamental domains for the actions of $\mathsf{G},\mathsf{H}$ on $\Omega$ whose intersection $U$ has positive measure (these always exist, as shown by translating $X_{\mathsf{H}}$ if needed). There are induced actions $\G\actson X_\sfH$ and $\sfH\actson X_\G$, defined (on conull subsets) through the identifications $X_\sfH\approx \sfH\backslash\Omega$ and $X_\G\approx \G\backslash\Omega$. To distinguish these actions, when $g\in\G$ and $x\in X_\sfH$, we will write $gx\in\Omega$ for the image of $x$ under the action of $g$ on $\Omega$, and $g\cdot x\in X_\sfH$ for its image under the induced action of $g$ on $X_\sfH$. More concretely $g\cdot x$ is the unique element of $X_\sfH$ in the same $\sfH$-orbit as $gx$ (uniqueness is ensured almost everywhere using that $X_\sfH$ is a fundamental domain for the $\sfH$-action on $\Omega$). The orbits of the two induced actions $\mathsf{G}\actson X_{\mathsf{H}}$ and $\mathsf{H}\actson X_{\mathsf{G}}$ have the same intersection with $U$ (up to a null set): indeed if $x,g\cdot x\in U$ for some $g\in\G$, then there exists $h\in\sfH$ such that $hgx\in U$; as the actions of $\G$ and $\sfH$ on $\Omega$ commute, we have $ghx\in U$, showing that $h\cdot x=g\cdot x$.

There is a natural cocycle $c:\G\times X_H\to \sfH$, defined by letting $c(g,x)$ be the unique element $h\in H$ such that $hgx\in X_\sfH$. Likewise there is a cocycle $\sfH\times X_\G\to\G$. These are called the \emph{measure equivalence cocycles} associated to $\Omega$ and to the fundamental domains $X_\sfH$, $X_\G$. Here the cocycle relation means that $c(g_1g_2,x)=c(g_1,g_2x)c(g_2,x)$ for every $g_1,g_2\in\G$ and almost every $x\in X_\sfH$. We also mention that changing the fundamental domain $X_\sfH$ to another one $X_\sfH'$ changes $c$ to a cocycle $c'$ which is \emph{cohomologous}, i.e.\ such that there exists a measurable map $\varphi:X\to\sfH$ such that for every $g\in\G$ and almost every $x\in X_\sfH$, if we denote by $x'\in X'_\sfH$ the unique element in the same $\sfH$-orbit as $x$, then $c'(g,x')=\varphi(gx)c(g,x)\varphi(x)^{-1}$.

 The above can also be reformulated in the language of measured groupoids, see e.g.\ \cite[Section~2.2]{Kid-survey} or \cite[Section~3]{GH} for an introduction. Every measure-preserving action of a countable group $\G$ on a standard probability space $X$ gives rise to a measured groupoid $\G\ltimes X$ over $X$: as a Borel set this is $\G\times X$, and the composition law is given by $(h,gx)(g,x)=(hg,x)$, see e.g.\ \cite[Example~2.20]{Kid-survey} for more details. Every element $\gamma$ in a measured groupoid $\calg$ over $X$ has a source $s(\gamma)$ and a range $r(\gamma)$ in $X$: in the above example $s(g,x)=x$ and $r(g,x)=g\cdot x$. The measured groupoid $\G\ltimes X$ is naturally equipped with a measurable cocycle (i.e.\ a homomorphism of measured groupoids) $\rho_\G:\calg\to\G$, defined by letting $\rho_\G(g,x)=g$. Also, for every measured groupoid $\calg$ over a standard probability space $X$ and every positive measure Borel subset $U\subseteq X$, we can consider the restricted measured groupoid $\calg_{|U}$, consisting of all elements $\gamma\in\calg$ with $s(\gamma),r(\gamma)\in U$.
 
 Coming back to the above situation of a measure equivalence coupling $\Omega$ between $\G$ and $\sfH$, the measured groupoids $\calg_1,\calg_2$ coming from the respective actions $\mathsf{G}\actson X_{\mathsf{H}}, \mathsf{H}\actson X_{\mathsf{G}}$ have isomorphic restrictions to $U=X_{\G}\cap X_\sfH$ (where isomorphism is understood up to restricting to a conull Borel subset).

\subsection{Reduction of couplings}\label{sec:coupling-taut}

In this section, we review work of Kida \cite{Kid-me,Kid-amalgam} and Bader--Furman--Sauer \cite{BFS}. Let $L$ be a Polish group, and let $\G$ be a countable subgroup of $L$. Then $L$ is equipped with an action of $\G\times\G$ by left-right multiplication, namely $(g_1,g_2)\cdot \ell=g_1\ell g_2^{-1}$. Throughout the paper, we will use the following definition.

\begin{de}[Strongly ICC]\label{de:icc}
Let $L$ be a Polish group, and let $\G$ be a countable subgroup of $L$. The inclusion $\G\subseteq L$ is \emph{strongly ICC} if the Dirac mass at $\mathrm{id}$ is the unique Borel probability measure on $L$ which is invariant under the conjugation by every element of $\G$.
\end{de}

\begin{theo}[{Kida \cite[Theorem~3.5]{Kid-amalgam}, Bader--Furman--Sauer \cite[Theorem~2.6]{BFS}}]\label{theo:taut}
Let $L$ be a Polish group, and let $\G$ be a countable subgroup of $L$, such that the inclusion $\G\subseteq L$ is strongly ICC. Assume that for every self measure equivalence coupling $\Sigma$ of $\G$, there exists a measurable almost $(\G\times\G)$-equivariant\footnote{Whenever we say that a map is \emph{almost equivariant}, we mean that the equivariance relation holds almost everywhere.} map $\Sigma\to L$.

   Let $\sfH$ be a countable group that is measure equivalent to $\G$, and let $\Omega$ be a measure equivalence coupling between $\G$ and $\sfH$. 
			
			Then there exist a homomorphism $\iota:\sfH\to L$ with finite kernel, and a measurable almost $(\G\times \sfH)$-equivariant map $\theta:\Omega\to L$, i.e.\ for a.e.\ $\omega\in\Omega$, and any $(g,h)\in\G\times \sfH$, one has $\theta((g,h)\cdot \omega)=g\theta(\omega)\iota(h)^{-1}$.
\end{theo}

 Our assumption on $\G$ is \emph{coupling rigidity} in the sense of Kida \cite[Definition~3.3]{Kid-amalgam}, or \emph{tautness} in the sense of Bader--Furman--Sauer \cite[Definition~1.3]{BFS}. Notice that the latter notion of tautness of the self-coupling $\Sigma$ requires the equivariant map $\Sigma\to L$ to be essentially unique. However, uniqueness is automatically ensured by the strong ICC assumption, see \cite[Lemma~A.8(1)]{BFS}.

\begin{lemma}[{Kida \cite[Lemma~5.8]{Kid-me}}]\label{lemma:extension}
    Let $L$ be a Polish group, and let $\G,\hat\G$ be countable subgroups of $L$, with $\G$ normal in $\hat\G$ and of finite index in $\hat\G$. Assume that the inclusion $\G\subseteq L$ is strongly ICC. 
    
    Let $\Sigma$ be a self measure equivalence coupling of $\hat\G$, and let $\Phi:\Sigma\to L$ be a measurable map which is almost $(\G\times\G)$-equivariant.

    Then $\Phi$ is almost $(\hat\G\times\hat\G)$-equivariant.
\end{lemma}

\begin{proof}
    This is almost exactly \cite[Lemma~5.8]{Kid-me}, except that the group $L$ is not supposed to be discrete -- however the proof is exactly the same, upon replacing the ICC condition in Kida's lemma by the strong ICC property.
\end{proof}

\subsection{Restricting couplings to stabilizers}\label{sec:me-stab}
  
When $K$ is a polyhedral complex with countably many cells, the group $\Aut(K)$, equipped with the pointwise convergence topology, is a Polish group. A faithful action of a countable group $\G$ on $K$ enables to view $\G$ as a subgroup of $L=\Aut(K)$. 

In the present section, we will formulate two general lemmas about measure equivalence couplings that involve a Polish group $L$, and specialize them to the context where $L=\Aut(K)$. They will enable us, starting with two countable subgroups $\G,\sfH$ of $\Aut(K)$, and a measure equivalence coupling between $\G$ and $\sfH$ that factors through the natural one on $\Aut(K)$, to induce measure equivalence couplings at the level of stabilizers of vertices of $K$, and also get a control on orbits. In later sections, this will often be applied to the action of the right-angled Artin group on its right-angled building. The results appearing in the present section are inspired by work of Kida \cite[Section~5]{Kid-amalgam}.

		\begin{prop}\label{prop:ME-coupling-restriction}
			Let $L$ be a Polish group and $\G$ be a countable subgroup of $L$. Let $L'\subseteq L$ be a Borel subgroup such that $\G\cdot L'=L$.
			
			Let $\sfH$ be a countable group that is measure equivalent to $\G$, let $(\Omega,\mu)$ be a measure equivalence coupling between $\G$ and $\sfH$, and assume we are given a homomorphism $\iota:\sfH\to L$ and a measurable almost $(\G\times \sfH)$-equivariant map $\theta:\Omega\to L$, where the $(\G\times\sfH)$-action on $L$ is via $(g,h)\cdot \ell=g\ell\iota(h)^{-1}$.
			
			Then the groups $\G'=\G\cap L'$ and $\sfH'=\iota^{-1}(L')$ are measure equivalent. More precisely $\Omega'=\theta^{-1}(L')$ is a measure equivalence coupling between $\G'$ and $\sfH'$. In addition, for every subgroup $\mathsf{K}$ of either $\G$ or $\sfH$, every Borel fundamental domain for the action of $\mathsf{K}\cap\G'$ (or $\mathsf{K}\cap\sfH'$) on $\Omega'$ is contained in a Borel fundamental domain for the action of $\mathsf{K}$ on $\Omega$. 
		\end{prop}

\begin{proof} 
			By definition $\Omega'$ is a $(\G'\times \sfH')$-invariant Borel subset of $\Omega$. 
   
   Let $\mathsf{K}\subseteq\mathsf{H}$ be a subgroup, let $\mathsf{K}'=\mathsf{K}\cap\sfH'$, and let $Y'$ be a Borel fundamental domain for the action of $\mathsf{K}'$ on $\Omega'$: this exists because the $\sfH$-action on $\Omega$ has one. We claim that for any $k\in \mathsf{K}\setminus\{1\}$, one has $\mu(kY'\cap Y')=0$. Indeed, for a.e.\ $x,y\in Y'$, if $y=kx$ for some $k\in \mathsf{K}$, then $\theta(y)=\theta(x)\iota(k)^{-1}$. As $\theta(x),\theta(y)\in L'$, we have $k\in \mathsf{K}'$. The claim thus follows from the fact that $Y'$ is a fundamental domain for the action of $\mathsf{K}'$ on $\Omega'$. The same argument applies to subgroups of $\G$.
			
			When $\mathsf{K}=\mathsf{H}$, the above claim implies in particular that every Borel fundamental domain for the action of $\sfH'$ on $\Omega'$ has finite measure. Likewise, any Borel fundamental domain for the action of $\G'$ on $\Omega'$ has finite measure.
   
%   that $Y'$ is contained in a Borel fundamental domain for the action of $\sfH$ on $\Omega$, therefore it has finite measure. Similarly, any Borel fundamental domain $X'$ for the $\G'$-action on $\Omega'$ is contained in a Borel fundamental domain for the action of $\G$ on $\Omega$, and therefore has finite measure. 
			
			In order to conclude that $\G'$ and $\sfH'$ are measure equivalent, there remains to prove that $\mu(\Omega')>0$. Since $L=\G\cdot L'$, the space $\Omega$ is covered by the countably many subsets $\theta^{-1}(gL')=g\theta^{-1}(L')$, with $g$ varying in $\G$. As the $\G$-action on $\Omega$ is measure-preserving, the subsets $\theta^{-1}(gL')$ all have the same measure, and therefore this measure is positive. In particular $\Omega'=\theta^{-1}(L')$ has positive measure, as desired.
			%we can find a finite or countable subset $F\subseteq G$ containing the identity element $e$, such that $L=\sqcup_{f\in F}f\cdot L'$. For every $f\in F$, we let $\Omega_f=\theta^{-1}(fL')$. Since the $\G$-action on $\Omega$ is measure-preserving, all subsets $\Omega_f$ have the same measure, and since $\G$ is countable this measure is positive. In particular $\Omega'=\Omega_e$ has positive measure.
			%Given $v'\in VK$, let $L_{v\to v'}\subseteq L$ be the Borel subset consisting of elements that send $v$ to $v'$. Let $\Omega_{v\to v'}=\theta^{-1}(L_{v\to v'})$. Since the $L$-orbit of $v$ coincides with its $\G$-orbit by assumption, we have \[\Omega=\coprod_{v'\in \G\cdot v}\Omega_{v\to v'}.\] Since the $\G$-action on $\Omega$ is measure-preserving, all subsets $\Omega_{v\to v'}$ have the same measure. Since there are countably many such subsets, they must have positive measure. In particular, the measure of $\Omega_v=\Omega_{v\to v}$ is positive. This completes our proof
		\end{proof}
  
		% \begin{proof}
		% 	By definition $\Omega'$ is a $(\G'\times \sfH')$-invariant Borel subset of $\Omega$. Let $\mathsf{K}\subseteq\mathsf{H}$ be a subgroup, let $\mathsf{K}'=\mathsf{K}\cap\sfH'$, and let $Y'$ be a Borel fundamental domain for action of $\sfH'$ on $\Omega'$: this exists because the $\sfH$-action on $\Omega$ has one. We claim that for any $h\in \sfH$, one has $\mu(hY'\cap Y')=0$. Indeed, for a.e.\ $x,y\in Y'$, if $y=hx$ for some $h\in \sfH$, then $\theta(y)=\theta(x)\iota(h)^{-1}$. As $\theta(x),\theta(y)\in L'$, we have $h\in \sfH'$. The claim thus follows from the fact that $Y'$ is a fundamental domain for the action of $\sfH'$ on $\Omega'$.
			
		% 	The above claim implies that $Y'$ is contained in a Borel fundamental domain for the action of $\sfH$ on $\Omega$, therefore it has finite measure. Similarly, any Borel fundamental domain $X'$ for the $\G'$-action on $\Omega'$ is contained in a Borel fundamental domain for the action of $\G$ on $\Omega$, and therefore has finite measure. 
			
		% 	In order to conclude that $\G'$ and $\sfH'$ are measure equivalent, there remains to prove that $\mu(\Omega')>0$. Since $L=\G\cdot L'$, the space $\Omega$ is covered by the countably many subsets $\theta^{-1}(gL')=g\theta^{-1}(L')$, with $g$ varying in $\G$. As the $\G$-action on $\Omega$ is measure-preserving, the subsets $\theta^{-1}(gL')$ all have the same measure, and therefore this measure is positive. In particular $\Omega'=\theta^{-1}(L')$ has positive measure, as desired.
		% \end{proof}
		
		In the sequel, Proposition~\ref{prop:ME-coupling-restriction} will be applied in the form of the following corollary (and often with $L=\Aut(K)$).
  %in which case the assumption we make on orbits is precisely $\Htrans$).
		
		\begin{cor}\label{cor:ME-coupling-restriction}
			Let $K$ be a polyhedral complex with countably many cells. Let $L$ be a Polish group acting on $K$ through a measurable homomorphism $L\to\Aut(K)$. Let $\G$ be a countable subgroup of $L$, and assume that the actions of $\G$ and $L$ on $VK$ have the same orbits.
			
			Let $\sfH$ be a countable group that is measure equivalent to $\G$, let $(\Omega,\mu)$ be a measure equivalence coupling between $\G$ and $\sfH$, and assume we are given a homomorphism $\iota:\sfH\to L$ and a measurable almost $(\G\times \sfH)$-equivariant map $\theta:\Omega\to L$.
			
			Then for every $v\in VK$, the groups $\G_v=\Stab_{\G}(v)$ and $\sfH_v=\iota^{-1}(\Stab_{L}(v))$ are measure equivalent. More precisely $\Omega_v=\theta^{-1}(\Stab_{L}(v))$ is a measure equivalence coupling between $\G_v$ and $\sfH_v$, and every Borel fundamental domain for the action of $\G_v$ (resp.\ $\sfH_v$) on $\Omega_v$ is contained in a Borel fundamental domain for the action of $\G$ (resp.\ $\sfH$) on $\Omega$.
		\end{cor}
  
		\begin{proof}
			This follows from Proposition~\ref{prop:ME-coupling-restriction}, applied with $L'=\Stab_{L}(v)$. The fact that $L=\G\cdot L'$ follows from our assumption that $\G$ and $L$ have the same orbits on $VK$.
		\end{proof}

\begin{rk}\label{rk:K}
 In view of Proposition~\ref{prop:ME-coupling-restriction}, more generally, for every subgroup $\mathsf{K}\subseteq\G$, every Borel fundamental domain for the action of $\mathsf{K}\cap\G_v$ on $\Omega_v$ is contained in a Borel fundamental domain for the action of $\mathsf{K}$ on $\Omega$.
\end{rk}

Given a countable set $\mathbb{D}$, we denote by $\Bij(\mathbb{D})$ the group of all bijections of $\mathbb{D}$; we equip it with the topology of pointwise convergence, for which it is a Polish group. Recall from the introduction that a group $\sfH$ has \emph{bounded finite subgroups} if there is a bound on the cardinality of a finite subgroup of $\sfH$.

\begin{prop}\label{prop:orbits}
%\Ccomm{slightly rephrase the beginning}
Let  $L^0$ be a Polish group with a countable subgroup $\G$. Assume that $\G$ acts transitively on a countable set $\mathbb{D}$ with finite stabilizers through a measurable homomorphism $L^0\to\Bij(\mathbb{D})$.

Let $\sfH$ be a countable group with bounded finite subgroups that is measure equivalent to $\G$, and let $\Omega$ be a measure equivalence coupling between $\G$ and $\sfH$. Assume that we are given a homomorphism $\iota:\sfH\to L^0$ and a measurable almost $(\G\times\sfH)$-equivariant map $\theta:\Omega\to L^0$.

Then $\sfH$ acts on $\mathbb{D}$ (through $\iota$) with only finitely many orbits and finite stabilizers.
\end{prop}

\begin{proof}
We denote by $\kappa$ the common cardinality of all $\G$-stabilizers on $\mathbb{D}$. Given two elements $s,u\in\mathbb{D}$, we let $L^0_{s\to u}$ be the Borel subset of $L^0$ consisting of all elements that send $s$ to $u$, and we let $\Omega_{s\to u}=\theta^{-1}(L^0_{s\to u})$. Notice that $\Omega_{s\to u}$ is invariant under $\G_{u}$ -- and on the other hand $\mu(g\Omega_{s\to u}\cap\Omega_{s\to u})=0$ for every $g\in\G\setminus\G_u$. In addition, the $\G$-translates of $\Omega_{s\to u}$ cover $\Omega$ because the $\G$-action on $\mathbb{D}$ is transitive. Thus, any Borel fundamental domain for the action of $\G_u$ on $\Omega_{s\to u}$ is also a fundamental domain for the action of $\G$ on $\Omega$. The measure of any such fundamental domain is equal to 
			%and contains a Borel fundamental domain $\Omega^1_{s\to u}$ for the action of $\G$ on $\Omega$ of measure 
			$\mu(\Omega_{s\to u})/|\G_{u}|=\mu(\Omega_{s\to u})/\kappa$. In particular, as $s,u$ vary in the $\G$-orbit of $v$, the sets $\Omega_{s\to u}$ all have the same (positive) measure, which we denote by $m$.  
			
			Corollary~\ref{cor:ME-coupling-restriction} (applied with $L=L^0$ and $K=\mathbb{D}$, so that $\Aut(K)=\Bij(\mathbb{D})$) ensures that the $\sfH$-stabilizer of any element $u\in\mathbb{D}$ is measure equivalent to $\G_{u}$, whence finite. Let $k$ be a bound on the cardinality of a finite subgroup of $\sfH$. For every $v\in \G\cdot u$, the set $\Omega_{u\to v}$ is $\sfH_u$-invariant; we let $\Omega'_{u\to v}\subseteq\Omega_{u\to v}$ be a fundamental domain for the action of the finite group $\sfH_u$. Then $\Omega'_{u\to v}$ has measure at least $m/k$.
			
			One has $\Omega=\coprod_{u\in \mathbb{D}}\Omega_{u\to v}$. Therefore, if $T\subseteq \mathbb{D}$ is a set of representatives of the $\sfH$-orbits, then $\coprod_{u\in T}\Omega'_{u\to v}$ is a fundamental domain for the $\sfH$-action on $\Omega$. Since $\sfH$ has a finite measure fundamental domain, and the measure of the above fundamental domain is at least equal to $m|T|/k$, it follows that $|T|<+\infty$. This concludes our proof.
\end{proof}

		\begin{cor}\label{cor:ME-coupling-orbits}
			Let $\G$ be a countable group, acting faithfully by automorphisms on a polyhedral complex $K$ with countably many cells,  in such a way that $\G$ and $\Aut(K)$ have the same orbits of vertices.  Let $\sfH$ be a countable group with bounded finite subgroups that is measure equivalent to $\G$, let $(\Omega,\mu)$ be a measure equivalence coupling between $\G$ and $\sfH$, and assume we are given a homomorphism $\iota:\sfH\to\Aut(K)$ and a measurable almost $(\G\times \sfH)$-equivariant map $\theta:\Omega\to\Aut(K)$.
			
			\begin{enumerate}
				\item If $v\in VK$ is a vertex with finite $\G$-stabilizer,  
    then the $\G$-orbit of $v$ is the union of finitely many $\sfH$-orbits. 
				\item If $v\in VK$ is a vertex, and if $V\subseteq VK$ is a set that is invariant under $\Stab_{\Aut(K)}(v)$, consisting of vertices having finite $\G$-stabilizer, and on which $\G_v$ acts transitively, then the action of $\sfH_v=\iota^{-1}(\Stab_{\Aut(K)}(v))$ on $V$ has finitely many orbits.
			\end{enumerate}
		\end{cor}

\begin{proof}
    The first assertion follows from Proposition~\ref{prop:orbits}, applied with $L^0=\Aut(K)$ and with $\mathbb{D}=\G\cdot v$, on which $L^0$ acts because $\G\cdot v=\Aut(K)\cdot v$.

    The second assertion follows from Proposition~\ref{prop:orbits}, applied with $L^0=\Stab_{\Aut(K)}(v)$, with $\G_v$ in place of $\G$ and $\sfH_v$ in place of $\sfH$, and with $\mathbb{D}=V$. The required measure equivalence coupling $\Omega_v$ between $\G_v$ and $\sfH_v$, coming with maps $\iota_v:\sfH_v\to L^0$ and $\theta_v:\Omega_v\to L^0$, is provided by Corollary~\ref{cor:ME-coupling-restriction}, applied with $L=\Aut(K)$.
\end{proof}

\begin{rk}
    Proposition~\ref{prop:orbits} and Corollary~\ref{cor:ME-coupling-orbits} fail if one does not assume that $\mathsf{H}$ has bounded finite subgroups. Indeed, it is possible to find a countable group $\mathsf{G}$ acting properly and cocompactly on a polyhedral complex $K$, satisfying the assumptions of Corollary~\ref{cor:ME-coupling-orbits}, and a non-uniform lattice $\mathsf{H}$ in $\Aut(K)$, acting with infinitely many orbits of vertices. In this case $\mathsf{G}$ and $\mathsf{H}$ are measure equivalent.
\end{rk}

\subsection{Measure equivalence couplings and index}

The following lemma will be used in Section~\ref{sec:factor-action} of the paper.

\begin{lemma}
			\label{lemma:finite index}
			Let $\mathsf G$ and $\mathsf H$ be two countable groups, and let $\Sigma$ be a measure equivalence coupling between $\mathsf G$ and $\mathsf H$. Let $\mathsf H'\subset \mathsf H$ be a subgroup. Assume that there exists a positive measure $(\mathsf G \times\mathsf H')$-invariant Borel subset $\Sigma'\subset \Sigma$ such that for any $h\in \mathsf H\setminus \mathsf H'$ we have $\mu(h\Sigma'\cap\Sigma')=0$.
			
			Then $\mathsf H'$ is of finite index in $\mathsf H$.
		\end{lemma}
		
		\begin{proof}
			Let $E$ and $E'$ be respective Borel fundamental domains for the actions of $\mathsf G$ and $\mathsf H'$ on $\Sigma'$, chosen such that $U=E\cap E'$ has positive measure (these exist because $\mathsf G$ and $\mathsf{H}$ admit Borel fundamental domains on $\Sigma$). 
			
			We claim that $E$ is contained in a Borel fundamental domain $X_{\mathsf G}$ for the $\mathsf G$-action on $\Sigma$, and $E'$ is contained in a Borel fundamental domain $ X_{\mathsf H}$ for the $\mathsf H$-action on $\Sigma$. For the first part of the claim, we take any Borel fundamental domain $X'_{\mathsf G}$ for the $\mathsf G$-action on $\Sigma$ and take $X_{\mathsf G}=(X'_{\mathsf G}\setminus \Sigma')\cup E$. For the second part of the claim, it suffices to prove that for any $h\in \mathsf H$, the intersection $hE'\cap E'$ has measure zero. If $h\in \mathsf H'$, then this follows from that $E'$ is a fundamental domain for the $\mathsf H'$-action on $\Sigma'$. If $h\notin \mathsf H'$, then our assumption implies that $\mu(h\Sigma'\cap \Sigma')=0$, and the claim follows.
			
			As recalled in Section~\ref{sec:intro-me}, there is a natural measure-preserving action $\mathsf H\actson X_{\mathsf G}$, obtained from the identification $X_{\mathsf G}\approx \mathsf G\backslash \Sigma$. We let $\calg_1$ be the associated measured groupoid over $X_{\mathsf G}$; it is equipped with a natural cocycle $\calg_1\to \mathsf H$. Similarly, we have a natural measure-preserving action $\mathsf G\actson X_{\mathsf H}$, and we let $\calg_2$ be the associated measured groupoid over $X_{\mathsf H}$. It is equipped with a natural cocycle $\calg_2\to \mathsf G$. Let now $U_1=X_{\mathsf G}\cap X_{\mathsf H}$: this is a positive measure Borel subset (because it contains $U=E\cap E'$). We have $(\calg_1)_{|U_1}=(\calg_2)_{|U_1}$. We denote by $\calg$ this measured groupoid over $U_1$; up to restricting $U_1$ to a conull Borel subset, $\calg$ is naturally equipped with two cocycles $\rho_{\mathsf G}:\calg\to \mathsf G$ and $\rho_{\mathsf H}:\calg\to \mathsf H$.   
			
			We now claim that $\calg_{|U}=\rho_{\mathsf H}^{-1}(\mathsf H')_{|U}$. The lemma follows from this claim by using e.g.\ \cite[Lemma~B.3]{HH-Higman}.
			
	To prove the claim, it suffices to show that $\calg_{|U}\subseteq \rho_{\mathsf H}^{-1}(\mathsf H')_{|U}$. Let $\gamma\in\calg_{|U}$, let $u=s(\gamma)$ be the source of $\gamma$. As $\Sigma'$ is $\mathsf{G}$-invariant, we have $gu\in \Sigma'$. Since $E'$ is a fundamental domain for the action of $\mathsf H'$ on $\Sigma'$, there exists a unique $h\in \mathsf H'$ such that $hgu\in E'$. So $h$ is also the unique element of $\mathsf H$ satisfying $hgu\in X_{\mathsf H}$. Therefore $\rho_{\mathsf H}(\gamma)=h$, showing that $\gamma\in\rho_{\mathsf H}^{-1}(\mathsf H')$, as desired.
		\end{proof}

\section{Proximal dynamics and strong ICC property for $\Aut(\B)$}\label{sec:proximal}

The goal of this section is to prove the following proposition. We refer to Definition~\ref{de:icc} for the definition of the strong ICC property.

\begin{prop}\label{prop:icc}
    Let $G$ be a right-angled Artin group with trivial center, and let $\B$ be its right-angled building. Then the inclusion $G\subseteq\Aut(\B)$ is strongly ICC.
\end{prop}

The assumption that $G$ has trivial center amounts to requiring that in the defining graph of $\Gamma$, no vertex is joined to all other vertices by an edge.

Our plan is to embed $\Aut(\B)$ measurably into a bigger group, namely the homeomorphism group $\Homeo(R(\B))$ of the \emph{regular part} of the Roller compactification of $\B$ (defined in Section~\ref{subsec:dynamics}), and show that the inclusion $G\subseteq\Homeo(R(\B))$ is strongly ICC using properties of the dynamics of the action of $G$ on the compact topological space $R(\B)$, which implies the proposition. In Section~\ref{subsec:ICC}, we give a general dynamical criterion for proving strongly ICC, and Section~\ref{subsec:dynamics}, we justify the measureable embedding of $\Aut(\B)$ into $\Homeo(R(\B))$, and show the action of $G$ on $R(\B)$ satisfies the dynamical criterion.

\subsection{A strong ICC lemma, after Bader--Furman--Sauer}
\label{subsec:ICC}
Recall that an action of a countable group $G$ by homeomorphisms on a compact metrizable space $K$ is \emph{strongly proximal} in the sense of Furstenberg \cite{Furs} if the $G$-orbit of every $\nu\in\Prob(K)$ contains a Dirac mass in its weak-$*$ closure. The action is \emph{minimal} if every $G$-orbit is dense in $K$. Given a compact metrizable space $K$, the group $\Homeo(K)$ is equipped with the topology of uniform convergence, which turns it into a Polish group. The following lemma is a small variation over an argument of Bader--Furman--Sauer \cite[Lemma~2.4]{BFS}.

\begin{lemma}\label{lemma:bfs}
    Let $G_1,\dots,G_k$ be countable groups, let $K_1,\dots,K_k$ be compact metrizable spaces, and assume that for every $i\in\{1,\dots,k\}$, the group $G_i$ acts  faithfully, minimally and strongly proximally on $K_i$.

    Then the inclusion $G_1\times\dots\times G_k\subseteq\Homeo(K_1\times\dots\times K_k)$ is strongly ICC.
\end{lemma}

\begin{proof}
We write $G=G_1\times\dots\times G_k$, and $K=K_1\times\dots\times K_k$. Let $\mu$ be a Borel probability measure on $\Homeo(K)$ which is invariant under the conjugation by any element of $G$. Let \[\Prob_\mu(K)=\{\nu\in\Prob(K)\mid \mu\ast\nu=\nu\}\] be the space of Borel $\mu$-stationary probability measures on $K$, which is a nonempty compact subset of $\Prob(K)$, equipped with the weak-$*$ topology. Note that $\Prob_\mu(K)$ is $G$-invariant, using that $\mu$ is conjugation-invariant.

We claim that for every $x=(x_1,\dots,x_k)\in K$, the Dirac mass $\delta_x$ belongs to $\Prob_\mu(K)$. This claim will imply that for every $x\in K$, we have $\mu\{f\in\Homeo(K)\mid f(x)=x\}=1$, so $\mu=\delta_{\mathrm{id}}$, and the lemma will follow.

We are thus left with proving the above claim. Let $\nu\in\Prob_\mu(K)$. For every $i\in\{1,\dots,k\}$, let $\nu_i\in\Prob(K_i)$ be the pushforward of $\nu$ under the projection map $K\to K_i$. Since the $G_i$-action on $K_i$ is minimal and strongly proximal, we can find a sequence $(g_{i,n})_{n\in\mathbb{N}}\in G_i^{\mathbb{N}}$ such that $(g_{i,n})_\ast\nu_i$ converges to $\delta_{x_i}$, as $n$ goes to $+\infty$, in the weak-$*$ topology. 

For every $n\in\mathbb{N}$, let $g_n=(g_{1,n},\dots,g_{k,n})$. We will now prove that the probability measures $(g_n)_\ast\nu$ converge to  $\delta_{x}$, as $n$ goes to $+\infty$, in the weak-$*$ topology. This will conclude our proof.

Let $U\subseteq K$ be an open set that contains $x$. Then there exist open neighborhoods $U_i$ of $x_i$, for every $i\in\{1,\dots,k\}$, such that $U_1\times\dots\times U_k\subseteq U$. For every $i\in\{1,\dots,n\}$, we have $(g_{i,n})_\ast\nu_i(U_i)\to 1$ by the Portmanteau Theorem (see \cite[Theorem~17.20]{Kec}). This means that \[\nu(K_1\times\dots \times K_{i-1}\times g_{i,n}^{-1}(U_i)\times K_{i+1}\times\dots\times K_k)\to 1\] as $n$ goes to $+\infty$. Intersecting these sets as $i$ varies in $\{1,\dots,k\}$, we obtain that $\nu(g_n^{-1}(U_1\times\dots\times U_k))\to 1$ as $n$ goes to $+\infty$. In particular $(g_n)_\ast\nu(U)\to 1$. Since $U$ was an arbitrary open set containing $x$, by the Portmanteau Theorem, this is enough to prove that $(g_n)_\ast\nu$ converges to $\delta_x$, as desired.   
\end{proof}

\subsection{Dynamics on the regular boundary}
\label{subsec:dynamics}
 We refer to \cite{Sag} for background on hyperplanes and halfspaces in $\mathrm{CAT}(0)$ cube complexes. 
Let $X$ be a $\mathrm{CAT}(0)$ cube complex. The Roller compactification of $X$ is defined in \cite{Rol,BCGNW} as follows. 
 Let $\mathrm{HS}$ be the set of all halfspaces in $X$ (i.e.\ complementary components of hyperplanes). Let $\varphi:VX\to \{0,1\}^{\mathrm{HS}}$ be the map such that $\varphi(v)(h)=1$ if $v\in h$, and $\varphi(v)(h)=0$ otherwise. The \emph{Roller compactification}, denoted $\overline X^R$, is the closure of $\varphi(VX)$ in $\{0,1\}^{\mathrm{HS}}$, in the topology of pointwise convergence. Thus, every point $\xi\in\overline{X}^R$ is a function from $\mathrm{HS}$ to $\{0,1\}$; we denote by $\langle \xi,h\rangle$ the value $\xi(h)$, and say that $h$ \emph{points towards} $\xi$ if $\langle\xi,h\rangle=1$. 
The \emph{Roller boundary} is $\partial_RX=\overline{X}^R\setminus VX$.

If $X$ has countably many cubes (and therefore countably many half-spaces), the compactification $\overline{X}^R$ is metrized as follows: fix an enumeration $\mathrm{HS}=\{h_n\}_{n\in\mathbb{N}}$, and for $\xi_1\neq\xi_2$, let $d(\xi_1,\xi_2)=2^{-n}$, where $n$ is the smallest integer such that $\langle\xi_1,h_n\rangle\neq\langle\xi_2, h_n\rangle$.

Two hyperplanes $\mathfrak h_1$ and $\mathfrak h_2$ in $X$ are \emph{strongly separated} \cite[Definition~2.1]{BC} if no hyperplane $\mathfrak h$ has a non-empty intersection with both $\mathfrak h_1$ and $\mathfrak h_2$ (this implies in particular that $\mathfrak h_1$ and $\mathfrak h_2$ are disjoint). Following Fern\'os \cite[Definition~7.3 and Proposition~7.4]{Fer}, we say that an element $\xi\in\overline{X}^R$ is \emph{regular} (also called \emph{strongly separated} in \cite[Definition~11]{KS}) if there exists an infinite nested sequence of halfspaces $h_1\supseteq h_2\supseteq\dots$ pointing towards $\xi$, whose boundary hyperplanes are pairwise strongly separated. Following \cite[Definition~5.9]{FLM}, if $X$ is irreducible, i.e.\ does not split as a product of two non-trivial convex subcomplexes, we define the \emph{regular boundary} $R(X)\subseteq\overline{X}^R$ (denoted by $S(X)$ in \cite{KS}) as the closure of the set of regular elements, which is compact. More generally, when $X=X_1\times\dots\times X_n$, with each $X_i$ irreducible, the Roller compactification splits naturally as $\overline{X}^R=\overline{X}_1^R\times\dots\times \overline{X}_n^R$, and we let $R(X)=R(X_1)\times\dots\times R(X_n)$, a compact subspace of $\overline{X}^R$.

Following \cite{CS}, we say that an action of a group $G$ on a $\mathrm{CAT}(0)$ cube complex $X$ is \emph{essential} if no $G$-orbit remains in a bounded neighborhood of a halfspace of $X$.  The action is \emph{non-elementary} if it has no   global fixed point, and no finite orbit in the visual boundary $\partial_\infty X$.

\begin{lemma}\label{lemma:basics-on-B}
Let $G=G_\Gamma$ be a non-cyclic irreducible right-angled Artin group, and let $\B$ be its right-angled building. Then $\B$ is irreducible and non-Euclidean, and $G$ acts essentially and  non-elementarily on $\B$.
%with no global fixed point in $\partial_\infty \B$.\Ccomm{change $X$ to $\B$.}
\end{lemma}

\begin{proof}
It was proved in \cite[Theorem~1.3 and Section~6.2]{CH} that $G$ acts on $\B$ with two independent strongly contracting isometries, which implies that $\B$ is irreducible, non-Euclidean, and that $G$ acts non-elemenarily on $\B$.

For essentiality, notice that every hyperplane crosses an edge $e$ joining a rank $0$ vertex to a rank $1$ vertex. It is enough to prove that $e$ is contained in a periodic bi-infinite geodesic  (for the $\mathrm{CAT}(0)$ metric on $\B$), as such geodesic is not contained in a bounded neighbourhood of any halfspace bounded by the hyperplane dual to $e$, and any periodic geodesic is contained in a bounded neighbourhood of any $G$-orbit.

	Let $\Gamma^c$ be the complement graph of $\Gamma$, i.e.\ $V\Gamma^c=V\Gamma$, and two vertices of $\Gamma^c$ are adjacent if and only if they are non-adjacent in $\Gamma$. As $\Gamma$ is not a join, $\Gamma^c$ is connected. Let $a_1$ be the label of $e$, and let $\gamma$ be a loop in $\Gamma^c$ based at $a_1$ made of a sequence $e_1,\dots,e_n$ of edges in $\Gamma^c$. We assume $\gamma$ has at least one edge and it is \emph{immersed}, i.e.\ for every $i\in\{1,\dots,n\}$ (considered modulo $n$), the edge $e_{i+1}$ is not equal to the edge $e_i$ traversed in the opposite direction. Let $a_1,\dots,a_n,a_{n+1}=a_1$ be consecutive vertices encountered along $\gamma$, and let $w$ be the product of these generators $a_1a_2\cdots a_{n-1}a_n$, an element of $G$. 
	
		\begin{figure}[h]
	\centering
	\includegraphics[scale=0.85]{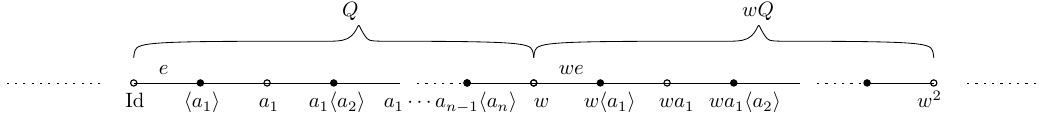}
	\caption{An axis of $w$.}
	\label{fig:geodesic}
\end{figure}	
			
We refer to Figure~\ref{fig:geodesic} for the rest of the proof. Recall that each vertex of $\B$ is represented by a left coset of a standard flat in $G$. We consider the following sequence of vertices in $\B$, alternating between rank 0 and rank 1 vertices:
			\[\{\mathrm{id}\},\ \langle a_1\rangle,\  a_1,\ a_1\langle a_2\rangle,\  a_1a_2,\ a_1a_2\langle a_3\rangle,\ \ldots,\ a_1a_2\cdots a_n.\]
			Consecutive members in this sequence are adjacent vertices in $\B$, so the above gives an edge path $Q$. We claim that $Q$ is a geodesic segment: indeed, the angle (in the sense of e.g.\ \cite[Chapter~II.3]{BH}) between two adjacent edges at a rank 1 vertex is $\pi$ (using that the link of any rank $1$ vertex has a bipartition into rank $0$ and higher-rank vertices).			And the angle between two adjacent edges at a rank 0 vertex is $\pi$ because $a_i$ and $a_{i+1}$ are not adjacent in $\Gamma$.  This is enough to ensure that $Q$ is a geodesic segment as local geodesics in a $\mathrm{CAT}(0)$ space are global geodesics \cite[Proposition~II.1.4(2)]{BH}. Likewise, as $a_1$ and $a_n$ are not adjacent in $\Gamma$, the concatenation $\cup_{i\in\mathbb Z}w^i Q$ is a geodesic line $\ell_w$ in $\B$, which is an axis for $w$. It has a translate passing through $e$, which is the desired bi-infinite geodesic.
\end{proof}

The following corollary is then an immediate application of \cite[Proposition~1]{KS}.
\begin{cor}\label{cor:kar-sageev}
Let $G$ be an irreducible right-angled Artin group, and let $\mathbb{B}$ be its right-angled building. Then the $G$-action on $R(\mathbb{B})$ is minimal and strongly proximal.
\qed
\end{cor}

\begin{lemma}
\label{lem:inj}
Let $G$ be a group acting on a $\mathrm{CAT}(0)$ cube complex $X$ which is irreducible and non-Euclidean. We assume that
\begin{enumerate}
    \item the action of $G$ on $X$ is essential and non-elementary;
    \item if an element of $G$ fixes a cube setwise, then it fixes the cube pointwise.
\end{enumerate}
Then the homomorphism $\Aut(X)\to\Homeo(R(X))$ is injective.
\end{lemma}

\begin{proof}
Let $f\in\Aut(X)$ act trivially on $R(X)$. By \cite[Corollary~7.7]{Fer} or \cite[Lemma~8]{KS},  $R(X)$ is non-empty, so let $\xi\in R(X)$ be a regular point.  By \cite[Corollary~6.2]{FLM} (which relies on work of Caprace--Lytchak \cite[Theorem~1.1]{CL}), there is an $\Aut(X)$-equivariant map $\pi:\partial_RX\to\partial_\infty X$. Since the action of $G$ on $X$ is non-elementary, the $G$-orbit of $\pi(\xi)$ in $\partial_\infty X$ is infinite, which implies that the $G$-orbit of $\xi$ is infinite. In particular $R(X)$ contains at least $3$ regular points, and by assumption they are all fixed by $f$. We can therefore use a barycenter argument, provided by \cite[Lemma~13]{KS} or \cite[Lemma~5.14]{FLM}, and deduce that $f$ fixes a point $x\in X$. Let $Z$ be the fix point set of $f$. Assumption 2 ensures that $Z$ is a subcomplex. As $X$ is $\mathrm{CAT}(0)$, $Z$ is convex. Thus $Z$ is a convex subcomplex of $X$. By $G$-invariance of $R(X)$, we know that $Z$ contains the $G$-orbit of $x$. Assumption 1 implies that each hyperplane of $X$ separates at least two points in $Z$. Thus $Z=X$ by \cite[Lemma~13.8]{HW}.
\end{proof}

\begin{lemma}\label{lemma:aut-homeo}
Let $G$ be a right-angled Artin group with trivial center, and let $\B$ be its right-angled building. Then the homomorphism $\Aut(\B)\to\Homeo(R(\B))$ is measurable and injective.
\end{lemma}
\begin{proof}
Injectivity in the case where $G$ is irreducible follows from  Lemma~\ref{lem:inj}, using Lemma~\ref{lemma:basics-on-B} to check the first assumption (essentiality and non-elementarity), and Lemma~\ref{lemma:rank} to check the second (setwise and pointwise stabilizers of cubes coincide).

 We now prove injectivity in general, without assuming irreducibility. By \cite[Proposition~2.6]{CS}, every $f\in\Aut(X)$ preserves the product structure $\B=\B_1\times\dots\times\B_k$ (but could potentially permute the factors). Since $|R(\B_i)|\ge 2$ for every $i\in\{1,\dots,k\}$, and since $f$ acts trivially on $R(\B)$, the automorphism $f$ does not permute the factors (as shown by comparing the images of two points of $R(\B)$ that differ on only one coordinate). By using the above for each factor independently, we see that $f=\mathrm{id}$.

We finally prove the measurability of the natural map $\Aut(\B)\to\Homeo(R(\B))$. Given $f\in\Aut(\B)$, we will denote by $f_\infty$ its image by this map. The Polish group $\Homeo(R(\B))$ is metrized with the uniform metric, i.e.\ $d(f,g)=\sup_{\xi\in R(\B)}d(f(\xi),g(\xi))$. It is enough to prove that for every $n\in\mathbb{N}$, the set of all automorphisms $f\in\Aut(\B)$ such that $d(f_\infty,\mathrm{id})<2^{-n}$ is a Borel subset of $\Aut(\B)$. Say that two half-spaces $h,h'$ are \emph{$R$-indistinguishable} if for every $\xi\in R(\B)$, one has $\langle \xi,h\rangle=\langle\xi,h'\rangle$. Then $d(f_\infty,\mathrm{id})<2^{-n}$ if and only if for every $k\le n$, the half-spaces $h_k$ and $f^{-1}(h_k)$ are $R$-indistinguishable, where we recall our enumeration $\mathrm{HS}=\{h_k\}_{k\in\mathbb{N}}$. This is a Borel condition, as it can be expressed by saying that for every $k\le n$, there exists a half-space $h$ that is $R$-indistinguishable from $h_k$, such that for every vertex $v\in h$, one has $f(v)\in h_k$.
\end{proof}

We are now ready to complete our proof of Proposition~\ref{prop:icc}.

\begin{proof}
    Write $G=G_1\times\dots\times G_k$, where no $G_i$ splits as a direct product. For every $i\in\{1,\dots,k\}$, let $\B_i$ be the right-angled building of $G_i$. By Corollary~\ref{cor:kar-sageev}, for every $i\in\{1,\dots,k\}$, the action of $G_i$ on the compact metrizable space $R(\B_i)$ is minimal and strongly proximal, and it is faithful by Lemma~\ref{lemma:aut-homeo}. Therefore, by Lemma~\ref{lemma:bfs}, the inclusion $G\subseteq\Homeo(R(\B))$ is strongly ICC. Now, if $\mu$ is a conjugation-invariant probability measure on $\Aut(\B)$, Lemma~\ref{lemma:aut-homeo} enables us to pushforward $\mu$ to a conjugation-invariant probability measure on $\Homeo(R(\B))$. It follows that $\mu$ is the Dirac mass at $\mathrm{id}$, as desired.
\end{proof}
  
		\section{Action on the right-angled building with amenable stabilizers}\label{sec:action-with-amenable-stab}
		
		The goal of the present section is to prove the following statement.
		
		\begin{prop}
			\label{prop:stabilizer-1}
			Let $G$ be a right-angled Artin group with $|\Out(G)|<+\infty$, let $\B$ be its right-angled building, and let $H$ be a countable group. Assume that $G$ and $H$ are measure equivalent.
			
			Then $H$ acts on $\B$ with amenable vertex stabilizers. If in addition $H$ has bounded finite subgroups, then this action can be chosen to be cocompact. 
		\end{prop}

		Proposition~\ref{prop:stabilizer-1} will be proved by applying the general statements established in Section~\ref{sec:me} to the action of an appropriate finite-index extension $\hat{G}$ of $G$ on $\Aut(\B)$. 
		
		\subsection{A finite-index extension of $G$ with the same transitivity as $\Aut(\Bu)$}\label{sec:finite-index-extension}
		
		Throughout the section, we let $G=G_\Gamma$ be a right-angled Artin group. The (finite) automorphism group $\Aut(\Gamma)$ naturally acts on $G$ by (outer) automorphisms. We let $\hat{G}=G\rtimes\Aut(\Gamma)$.
		
		The action of $\Aut(\Gamma)$ on $G$ sends standard flat to standard flat. Therefore, the $G$-action on its right-angled building $\Bu$ extends to an action of $\hat{G}$ by cubical automorphisms, where an element $(g,\theta)\in \hat{G}$ sends a vertex representing a coset $hG_\Lambda$, where $\Lambda\subseteq\Gamma$ is a complete subgraph, to the vertex representing $g\theta(hG_\Lambda)$.
		
		The importance of the group $\hat{G}$ for us is that it acts on $\Bu$ with the same transitivity as the full automorphism group $\Aut(\Bu)$, as shown by the following lemma. This property will be crucial in order to apply Corollary~\ref{cor:ME-coupling-orbits}. 
		
		\begin{lemma}\label{lemma:transitivity-edges}
			For every edge $e\in E\Bu$, the orbits of $e$ under $\hat{G}$ and under $\Aut(\Bu)$ coincide. 
		\end{lemma}
		
		\begin{proof}
			It suffices to show that if $he=e'$ for some $h\in \Aut(\B)$, then there exists $g\in \hat G$ such that $ge=e'$. Recall that vertices of $\B$ are in one-to-one correspondence with standard flats in $G$, and elements in $\Aut(\mathbb B)$ can be viewed as flat-preserving bijections of $G$. Edges of $\B$ are in one-to-one correspondence with codimension one inclusions of standard flats $F_1\subseteq F_2$. Let $F_1\subseteq F_2$ and $F'_1\subseteq F'_2$ be the inclusions of standard flats associated to the edges $e$ and $e'$, respectively. Since every automorphism of $\B$ preserves the ranks of vertices (Lemma~\ref{lemma:rank}), we have  $h(F_1,F_2)=(F'_1,F'_2)$.
			
			Let $x\in F'_1$ be a vertex. Then there is a unique standard flat $\tilde F_1$ containing $x$ such that $F_1$ and $\tilde F_1$ have the same type. Likewise there exists a unique standard flat $\tilde F_2$ containing $\tilde F_1$ which has the same type as $F_2$. Then $(F_1,F_2)$ and $(\tilde F_1,\tilde{F}_2)$ differ by a left translation, i.e.\ there exists $g_1\in G$ such that $g_1 (F_1,F_2)=(\tilde F_1,\tilde F_2)$. Thus $g_1h^{-1}(F'_1,F'_2)=(\tilde F_1,\tilde F_2)$. 
			
			Let $x'=g_1h^{-1}(x)\in \tilde F_1$. As $x$ and $x'$ both belong to $\tilde F_1$, there exists $g_2\in G$ such that $g_2(x')=x$ and $g_2 (\tilde F_1,\tilde F_2)=(\tilde F_1,\tilde F_2)$. Thus $g_2g_1h^{-1}(F'_1,F'_2)=(\tilde F_1,\tilde F_2)$. Also $g_2g_1h^{-1}(x)=x$ belongs to $F'_1\cap \tilde F_1$. Let $v_x$ be the vertex in $\B$ associated to $x$, and let $\tilde e$ be the edge associated to $(\tilde F_1,\tilde F_2)$. Let $K_x$ be the union of all cubes in $\B$ containing $v_x$. Note that $\tilde e,e'\subset K_x$, and $g_2g_1h^{-1}$ stabilizes $K_x$, sending $e'$ to $\tilde e$. Note that $\lk(v_x,K_x)$ is isomorphic to the \emph{flag completion} of the defining graph $\Gamma$ of $G$, i.e.\ vertices of $\lk(v_x,K_x)$  are in one-to-one correspondence with vertices of $\Gamma$, and a collection of vertices in $\lk(v_x,K_x)$ span a simplex whenever the associated vertices in $\Gamma$ span a complete subgraph.
 So the map $(g_2g_1h^{-1})_{|K_x}:K_x\to K_x$, and its conjugate $(x^{-1}g_2g_1h^{-1}x)_{|K_\mathrm{id}}:K_{\mathrm{id}}\to K_{\mathrm{id}}$, are induced by an automorphism of $\Ga$. This means that $x^{-1}g_2g_1h^{-1}x$ has the same action on $K_{\mathrm{id}}$ as an element of $\Aut(\Gamma)\subset \hat G$.
    Thus there exists $g_3\in \hat G$ such that $g_3(e')=\tilde e$, and hence $g^{-1}_3g_1(e)=e'$, as desired.
		\end{proof}
		
		\begin{cor}\label{cor:transitivity}
			For every vertex $v\in V\B$, the orbits of $v$ under $\hat{G}$ and under $\Aut(\B)$ coincide. 
		\end{cor}

		\begin{proof}
			Let $v'\in V\B$, and assume that there exists $h\in\Aut(\B)$ such that $hv=v'$. Let $e$ be an edge that contains $v$, and let $e'=he$. By Lemma~\ref{lemma:transitivity-edges}, there exists $g\in\hat{G}$ such that $e'=ge$. Since every automorphism of $\B$ preserves the ranks of vertices (Lemma~\ref{lemma:rank}), we deduce that $v'=gv$, and the corollary follows.
		\end{proof}
		
		We will also need to know that $\hat{G}$ and $\Aut(\B)\simeq\Aut(\Gamma^e)$ have the same orbits of vertices when acting on the extension graph, as shown by the following lemma.
		
		\begin{lemma}\label{lemma:transitivity-extension-graph}
		Assume that $|\Out(G)|<+\infty$. Then for every $\sfv\in V\Gamma^e$, the orbits of $\sfv$ under $\hat{G}$ and $\Aut(\Gamma^e)$ coincide.
		\end{lemma}
		
		\begin{proof}
			Let $h\in\Aut(\Gamma^e)$ and let $\sfw=h\sfv$. We aim to prove that there exists $g\in\hat G$ such that $\sfw=g\sfv$. 
			
			When viewed as an automorphism of $\B$ through the isomorphism $\Aut(\Gamma^e)\to\Aut(\B)$ provided by Lemma~\ref{lemma:iso}, the element $h$ sends a $\sfv$-line $\ell$ (associated to a rank $1$ vertex $v\in V\B$) to some $\sfw$-line $\ell'$ (associated to a rank $1$ vertex $v'\in V\B$). By Corollary~\ref{cor:transitivity}, there exists $g\in\hat G$ with $gv=v'$. Since $g$ sends the $\sfv$-line $\ell$ to the $\sfw$-line $\ell'$, we have $g\sfv=\sfw$, which concludes our proof. 
		\end{proof}

		\subsection{Reduction of self couplings to the right-angled building}
		
 The following lemma establishes the crucial reduction property of self-couplings from Theorem~\ref{theo:taut}, for the action of $\hat{G}$ on $\B$.

		\begin{lemma}\label{lemma:self-coupling}
			Let $G$ be a right-angled Artin group with $|\Out(G)|<+\infty$, and let $\Sigma$ be a self measure equivalence coupling of $\hat G$.
			
			Then there exists a measurable almost $(\hat G\times \hat G)$-equivariant map $\Sigma\to\Aut(\B)$.
		\end{lemma}
		
		Our proof of Lemma~\ref{lemma:self-coupling} is essentially a translation of the main results of our earlier paper \cite{HH} from the language of measured groupoids to the language of self-couplings, using the dictionary between measure equivalence and stable orbit equivalence developed by Furman \cite{Fur-oe} and recalled in Section~\ref{sec:intro-me}, and some arguments from the work of Kida \cite{Kid-me}. 
		
		\begin{proof}
			Since $G$ has finite index in $\hat{G}$, the space $\Sigma$ is also a self measure equivalence coupling of $G$.
			
			Let $X_\ell,X_r\subseteq \Sigma$ be respective fundamental domains for the actions of $G_\ell=G\times\{1\}$ and $G_r=\{1\}\times G$ on $\Sigma$. In view of \cite[Lemma~2.27]{Kid-survey}, we can (and will) choose $X_\ell,X_r$ so that $(G\times G)\cdot (X_\ell\cap X_r)=\Sigma$ up to null sets. Let $U=X_\ell\cap X_r$.
	
  As recalled in Section~\ref{sec:intro-me}, there are natural measure-preserving actions $G_\ell\actson X_r$ and $G_r\actson X_\ell$, obtained through the identifications $X_r\approx G_r\backslash\Sigma$ and $X_\ell\approx G_\ell\backslash \Sigma$. Their orbits coincide on $U$, so the two corresponding measured groupoids $G_\ell\ltimes X_r$ and $G_r\ltimes X_\ell$ restrict to isomorphic measured groupoids on $U$. We denote by $\calg$ this common measured groupoid over $U$, which is naturally equipped (up to restricting to a conull Borel subset of $U$) with two cocycles $\rho_\ell:\calg\to G_\ell$ and $\rho_r:\calg\to G_r$. 
			
	 The first two paragraphs of the proof of \cite[Theorem~3.19]{HH} yield a Borel map $\theta:V\Gamma^e\times U\to V\Gamma^e$ such that for every $\sfv\in V\Gamma^e$, there exists a partition $U^*=\dunion_{i\in I}U_i$ of a conull Borel subset $U^*\subseteq U$ into at most countably many Borel subsets satisfying the following properties:
   \begin{enumerate}
       \item for every $i\in I$, the restriction $\theta|_{\{\sfv\}\times U_i}$ takes constant value, denoted $\sfw_i\in V\Gamma^e$;
       \item for every $i\in I$, we have $\rho_\ell^{-1}(\Stab_{G_\ell}(\sfv))_{|U_i}=\rho_r^{-1}(\Stab_{G_r}(\sfw_i))_{|U_i}$. 
   \end{enumerate}
   Moreover, it is shown in the proof of \cite[Theorem~3.19]{HH} that for almost every $u\in U$, the map $\theta(\cdot,u)$ gives an element in $\Aut(\Gamma^e)$, and this gives a measurable map $\bar\theta: U\to\Aut(\Gamma^e)$ (up to replacing $U$ by a conull Borel subset). 
   
   We claim that the map $\bar\theta$ satisfies the following equivariance: up to restricting $U$ to a conull Borel subset, for every $g\in\calg$, one has $\bar\theta(r(g))=\rho_r(g)\bar\theta(s(g))\rho_\ell(g)^{-1}$. The argument comes from \cite[Lemma~5.5]{Kid-me}, and is the following. It is enough to prove this equivariance on a Borel subset $B\subseteq\calg$ where the values of $\rho_\ell$ and $\rho_r$ are constant, and which induces a Borel isomorphism between $s(B)$ and $r(B)$ -- indeed $\calg$ is covered by countably many such Borel subsets. We denote by $h_\ell,h_r\in G$ the respective values of $\rho_\ell,\rho_r$ on $B$. Let $\sfv\in V\Gamma^e$. Up to a countable Borel partition of $B$, we can further assume that $\bar\theta(\cdot)(\sfv)$ is constant on $s(B)$ (with value denoted by $\sfw$), and that $\bar\theta(\cdot)(h_\ell\sfv)$ is constant on $r(B)$ (with value denoted by $\sfw'$). Then $\rho_\ell^{-1}(\Stab_G(\sfv))_{|s(B)}=\rho_r^{-1}(\Stab_G(\sfw))_{|s(B)}$ by the definition of $\bar\theta$.  Thus $\rho_{\ell}^{-1}(\Stab_G(h_\ell\sfv))_{|r(B)}$ is both equal to $\rho_r^{-1}(\Stab_G(h_r\sfw))_{|r(B)}$ (by the choice of $B$) and to $\rho_r^{-1}(\Stab_G(\sfw'))_{|r(B)}$ (by the definition of $\bar\theta$).
   Hence $h_r\sfw=\sfw'$, using the uniqueness statement \cite[Lemma~3.14]{HH}. In other words, for every $g\in B$, we have proved that $\bar\theta(r(g))(\rho_\ell(g)\sfv)=\rho_r(g)(\bar\theta(s(g))(\sfv))$. As $V\Gamma^e$ is countable, this is exactly the desired equivariance.
   
Under the natural isomorphism between $\Aut(\Gamma^e)$ and $\Aut(\B)$ recalled in Section~\ref{sec:extension-graph}, we view $\bar\theta$ as a map $U\to\Aut(\B)$, which satisfies the same equivariance (see Remark~\ref{rk:iso}).
   
   Recall that by our choice of $U$, there exists a conull Borel subset $\Sigma^*\subseteq\Sigma$ such that $(G\times G)\cdot U=\Sigma^*$. We now claim, following \cite[Theorem~5.6]{Kid-me}, that the assignment $(g_1,g_2)x\mapsto g_1\bar\theta(x)^{-1}g_2^{-1}$, with $x\in U$, is a well-defined $(G\times G)$-equivariant Borel map $\Sigma^*\to\Aut(\B)$. The only point we need to check is that if $(g_1,g_2)x=y$ with $x,y\in U$, then $g_1\bar\theta(x)^{-1}g_2^{-1}=\bar\theta(y)^{-1}$. This is precisely the contents of the equivariance proved at the level of groupoids, so our claim is proved.
   
   Finally, using that $G$ is normal in $\hat{G}$, Lemma~\ref{lemma:extension} and  Proposition~\ref{prop:icc} ensure that $\bar{\theta}$ is in fact $(\hat{G}\times\hat{G})$-equivariant, which completes our proof.
			\end{proof}

		\subsection{Proof of Proposition~\ref{prop:stabilizer-1}}
		
		Proposition~\ref{prop:stabilizer-1} follows from the combination of Lemmas~\ref{lemma:action-on-B},~\ref{lemma:restricted-coupling} and Corollary~\ref{cor:cocompact} below.
		
		\begin{lemma}\label{lemma:action-on-B}
			Let $G$ be a right-angled Artin group with $|\Out(G)|<+\infty$, let $H$ be a countable group, and let $\Omega$ be a measure equivalence coupling between $\hat{G}$ and $H$. 
			
			Then there exist a group homomorphism $\iota:H\to\Aut(\B)$ with finite kernel and a measurable almost $(\hat G\times H)$-equivariant map $\theta:\Omega\to\Aut(\B)$, i.e.\ such that for all  $(g,h)\in \hat G\times H$ and a.e.\ $\omega\in\Omega$, one has $\theta((g,h)\omega)=g\theta(\omega)\iota(h)^{-1}$.
		\end{lemma}
		
		\begin{proof}
			This is a consequence of Theorem~\ref{theo:taut}, applied with $\G=\hat{G}$ and  $L=\Aut(\B)$, using that the $\hat{G}$-action on $\B$ is faithful. Indeed the inclusion $\hat G\subseteq\Aut(\B)$  is strongly ICC by Proposition~\ref{prop:icc}, and self-couplings of $\hat{G}$ map to $\Aut(\B)$ by Lemma~\ref{lemma:self-coupling}.
		\end{proof}

		Let $(\Omega,\mu)$ be a measure equivalence coupling between $\hat{G}$ and $H$. Lemma~\ref{lemma:action-on-B} gives a measurable group homomorphism $\iota:H\to\Aut(\B)$ and a measurable almost $(\hat{G}\times H)$-equivariant map $\theta:\Omega\to\Aut(\B)$. In particular $\iota$ provides an action of $H$ on $\B$.
		
		Let $v\in V\B$, let $\Omega_v=\theta^{-1}(\Stab_{\Aut(\B)}(v))$, let $H_v=\iota^{-1}(\Stab_{\Aut(\B)}(v))$, and let $G_v$ (resp.\ $\hat{G}_v$) be the stabilizer of $v$ in $G$ (resp.\ $\hat{G}$). The equivariance of $\theta$ ensures that $\Omega_v$ is invariant under the action of $\hat{G}_v\times H_v$. 
		
		\begin{lemma}\label{lemma:restricted-coupling}
			The space $\Omega_v$ is a measure equivalence coupling between $G_v$ and $H_v$, in particular $H_v$ is amenable. 
   
   In addition, every fundamental domain for the action of $\hat G_v$ (resp.\ $G_v$, resp.\ $H_v$) on $\Omega_v$ is contained in a fundamental domain for the action of $\hat G$ (resp. $G$, resp.\ $H$) on $\Omega$.
		\end{lemma}

		\begin{proof}
			The fact that $\Omega_v$ is a measure equivalence coupling between $G_v$ and $H_v$ follows from Corollary~\ref{cor:ME-coupling-restriction}, applied (with $L=\Aut(\B)$) to the action of $\G=\hat{G}$ on $\B$. Indeed $\hat{G}$ and $\Aut(\B)$ have the same orbits of vertices by Corollary~\ref{cor:transitivity}. 
   
   The amenability of $H_v$ follows from the fact that a countable group which is measure equivalent to an amenable one, is itself amenable, see e.g.\ \cite[Corollary~1.3]{Fur-me}. 
   
   The statement about fundamental domains for the actions of $\hat{G}_v$ and $H_v$ also follows from Corollary~\ref{cor:ME-coupling-restriction}. Finally, the statement about fundamental domains for the action of $G_v$ follows from Remark~\ref{rk:K}, applied with $\mathsf{K}=G$. 
		\end{proof}

		Given a vertex $v\in V\B$, we denote by $V^0(\B)_{\le v}$ the set of all rank $0$ vertices of $\B$ that are smaller than $v$ (for the partial order on $V\B$ introduced in Section~\ref{sec:extension-graph}).
		
		\begin{lemma}\label{lemma:cocompact}
			Assume that $H$ has bounded finite subgroups. Then the $H$-action on the set of rank 0 vertices of $\B$ has finite stabilizers and finitely many orbits. In addition, for every vertex $v\in V\B$, the $H_v$-action on $V^0(\B)_{\le v}$ has finitely many orbits.
		\end{lemma}
		
		\begin{proof}
			The first part follows from the first conclusion of Corollary~\ref{cor:ME-coupling-orbits}, using that $H$ has bounded finite subgroups and that rank $0$ vertices have finite $\hat{G}$-stabilizers and all belong to the same $\hat{G}$-orbit.
   %and that the $\hat{G}$-action on $\B$ satisfies $\Htrans$. 
   The second part follows from the second conclusion of Corollary~\ref{cor:ME-coupling-orbits}, applied to $V=V^0(\B)_{\le v}$: indeed $V$ is invariant under $\Stab_{\Aut(\B)}(v)$ because $\Aut(\B)$ preserves ranks of vertices (Lemma~\ref{lemma:rank}), and the action of $G_v$ on $V$ is transitive.
		\end{proof}

\begin{cor}
\label{cor:cocompact}
 Assume that $H$ has bounded finite subgroups. Then the $H$-action on $\B$ is cocompact.
\end{cor}

\begin{proof}
By Lemma~\ref{lemma:cocompact}, $H$ acts on the set of rank 0 vertices of $\B$ with finitely many orbits. Recall that the action of $\Aut(\B)$ on $\B$ preserves ranks of vertices (Lemma~\ref{lemma:rank}), hence respects the order of vertices on $\B$. As each vertex of rank at least $1$ is lower bounded by at least one rank 0 vertex, and there is a bound $C$ such that each rank $0$ vertex is smaller than at most $C$ vertices of higher rank in $\B$, it follows that the action of $H$ on $\B$ has finitely many orbits of vertices. More generally, given any vertex $s$ of $\B$, there are only finitely many $k$-cells that contain $s$ as a minimal rank vertex. Thus the action of $H$ on $\B$ has only finitely many orbits of cells, and is therefore cocompact.
\end{proof}

		\section{Exploiting integrability}\label{sec:integrability}
		
		In this section, we exploit the integrability condition on the measure equivalence coupling between $G$ and $H$ in order to prove that the stabilizers of rank $1$ vertices for the $H$-action on the right-angled building of $G$ are virtually cyclic.

Recall from the introduction that an \emph{$(L^1,L^0)$-measure equivalence coupling from $H$ to $G$} is a measure equivalence coupling $(\Omega,\mu)$ between $H$ and $G$ for which there exists a fundamental domain $X_G$ for the $G$-action on $\Omega$ such that for each $h\in H$, \[\int_{X_G}|c(h,x)|_G\; d\mu(x)<+\infty,\] where $c:H\times X_G\to G$ is the associated measure equivalence cocycle, and $|\cdot |_G$ is a word length on $G$ associated to some finite generating set.
		
		\begin{prop}
			\label{prop:stabilizer-L1}
			Let $G$ be a right-angled Artin group with $|\Out(G)|<+\infty$, let $\B$ be its right-angled building, and let $H$ be a countable group with bounded finite subgroups. Assume that there exists an $(L^1,L^0)$-measure equivalence coupling from $H$ to $G$.
			
			Then $H$ acts on $\B$ with virtually infinite cyclic stabilizers of rank $1$ vertices. 
		\end{prop}

The plan of the proof is to show first that for each vertex $v\in V\B$ there is an $(L^1,L^0)$-measure equivalence coupling from $H_v$ (the $H$-stabilizer of $v$) to $G_v$ (Lemma~\ref{lemma:L1}), which then gives an $L^1$-integrable embedding from $H_v$ to $G_v$ (Corollary~\ref{cor:l1-embedding}). When $v$ is rank 1, we have $G_v\cong \mathbb Z$. Then we show in Theorem~\ref{theo:fg} that $H_v$ is virtually $\mathbb Z$: if we knew that $H_v$ were finitely generated, then this would follow from work of Bowen (in an appendix of an article by Austin \cite[Theorem~B.10]{Aus}) saying that growth is invariant under $L^1$-measure equivalence; the main point of our work is to extend this to the case where $H_v$ is possible infinitely generated.
  
In the whole section, when $G$ is a finitely generated group, with a finite generating set $S$, we will write $|g|_S$ to denote the word length of an element $g\in G$ with respect to the generating set $S$. When $S$ is clear from the context (or irrelevant to the statement), we will sometimes simply write $|g|_G$.

		 \subsection{Integrable coupling between vertex stabilizers}\label{sec:integrable-stab}
		
		 We start with the following easy observation.
		
		 \begin{lemma}\label{lemma:integrable-coupling-finite-index}
		 	Let $G,H$ be countable groups, with $G$ finitely generated, and let $\hat{G}$ be a finite-index extension of $G$. Let $\hat{\Omega}$ be an $(L^1,L^0)$-measure equivalence coupling from $H$ to $\hat{G}$.
			
		 	Then $\hat\Omega$ is an $(L^1,L^0)$-measure equivalence coupling from $H$ to $G$.
		 \end{lemma}
		
		 \begin{proof}
		 	By definition, there exists a fundamental domain $X_{\hat G}$ for the action of $\hat{G}$ on $\hat\Omega$ such that the cocycle $\hat c:H\times X_{\hat G}\to \hat{G}$ is $L^1$-integrable.
			
		 	Let $S=\{g_1,\dots,g_k\}$ be a (finite) set of representatives of the right cosets of $G$ in $\hat{G}$. Then $X_G=g_1X_{\hat{G}}\cup\dots\cup g_kX_{\hat{G}}$ is a fundamental domain for the $G$-action on $\hat\Omega$.
			
			We claim that the associated cocycle $c:H\times X_G\to G$ is $L^1$-integrable. Indeed, let $h\in H$. For $x\in X_G$, there exists a unique $\hat x\in X_{\hat G}$ and $j\in\{1,\dots,k\}$ such that $x=g_j\hat{x}$. Also by definition of $X_{\hat G}$, there exists a unique element $\hat g\in \hat G$ such that $\hat g hx\in X_{\hat G}$. Using the fact that the actions of $\hat G$ and $H$ on $\hat\Omega$ commute, we see that $(\hat g g_j)h\hat x\in X_{\hat G}$, showing that $\hat{c}(h,\hat{x})=\hat g g_j$. On the other hand, the set $S^{-1}=\{g_1^{-1},\dots,g_k^{-1}\}$ is a set of representatives of the left cosets of $G$ in $\hat{G}$, so there exist $i\in\{1,\dots,k\}$ and $g\in G$ such that $\hat{g}=g_i^{-1}g$. We then have $ghx\in g_i X_{\hat G}$, showing that $c(h,x)=g$. It follows that $c(h,x)\in S\hat{c}(h,\hat x)S^{-1}$. As this is true for every $h\in H$ and $x\in X_G$ and $S$ is a finite set, the lemma follows.
		 \end{proof}
		
	Let now $G,H$ be as in the statement of Proposition~\ref{prop:stabilizer-L1}. Let $\Omega$ be an $(L^1,L^0)$-measure equivalence coupling from $H$ to $G$, and let $X_G$ be a Borel fundamental domain for the $G$-action on $\Omega$ such that the measure equivalence cocycle $c:H\times X_G\to G$ is $L^1$-integrable.

 Let $\hat{G}$ be the finite-index extension of $G$ introduced in Section~\ref{sec:finite-index-extension}. Let $\hat{\Omega}$ be the induced coupling between $\hat{G}$ and $H$, namely $\hat{\Omega}=G\backslash (\hat{G}\times\Omega)$ -- here $\hat{G}\times G$ acts on $\hat{G}$ via $(\hat{g},g)\cdot h=\hat{g}hg^{-1}$, and $\hat{G}$ acts trivially on $\Omega$ while $H$ acts trivially on $\hat{G}$, and we are taking the quotient by the diagonal action of $G$ on $\hat{G}\times\Omega$. Notice that $X_{\hat G}=\{\mathrm{id}\}\times X_G$, identified to a subset of $\hat{\Omega}$, is a Borel fundamental domain for the action of $\hat{G}$ on $\hat{\Omega}$. The associated measure equivalence cocycle $H\times X_{\hat{G}}\to\hat{G}$ takes its values in $G$ and coincides with $c$. In particular $\hat{\Omega}$ is an $(L^1,L^0)$-measure equivalence coupling from $H$ to $\hat G$. We can therefore apply Lemma~\ref{lemma:integrable-coupling-finite-index} and obtain a Borel fundamental domain $\hat{X}_G$ for the $G$-action on $\hat{\Omega}$, such that the measure equivalence cocycle $\hat{c}:H\times \hat{X}_G\to G$ is $L^1$-integrable.
 
 Lemma~\ref{lemma:action-on-B} gives us a homomorphism $\iota:H\to\Aut(\B)$ and an almost equivariant map $\theta:\hat\Omega\to\Aut(\B)$. In particular we have an action of $H$ on $\B$. Given $v\in V\B$, we denote by $G_v$ and $H_v$ its stabilizers for the actions of $G$ and $H$, respectively. 
		
		\begin{lemma}\label{lemma:L1}
			For every $v\in V\B$, there exists an $(L^1,L^0)$-measure equivalence coupling from $H_v$ to $G_v$.
		\end{lemma}

Our proof of Lemma~\ref{lemma:L1} is inspired by an argument coming from work of Escalier and the first-named author \cite[Theorem~12.1]{EH}.
  
		\begin{proof}
 As above, let $\hat{X}_G$ be a fundamental domain for the action of $G$ on $\hat{\Omega}$, such that the cocycle $\hat{c}:H\times \hat{X}_G\to G$ is $L^1$-integrable. 
			
			By Lemma~\ref{lemma:restricted-coupling}, the space $\hat\Omega_v=\theta^{-1}(\Stab_{\Aut(\B)}(v))$ is a measure equivalence coupling between $G_v$ and $H_v$. Let $Y_{G_v},Y_{H_v}$ be respective Borel fundamental domains for these actions on $\hat\Omega_v$. Lemma~\ref{lemma:restricted-coupling} also shows that $Y_{G_v},Y_{H_v}$ extend to Borel fundamental domains $Y_G,Y_H$ for the actions of $G$ and $H$ on $\hat\Omega$.
			
			Let $c_v:H_v\times Y_{G_v}\to G_v$ be the measure equivalence cocycle. Then $c_v$ extends to a measure equivalence cocycle $c':H\times Y_G\to G$. The cocycle $c'$ is cohomologous to $\hat{c}$, i.e.\ there exists $\theta:Y_G\to G$ such that $c'(h,x)=\theta(h\cdot x)^{-1}\hat{c}(h,\tilde{x})\theta(x)$, where $\tilde{x}$ is the unique element of $\hat{X}_G$ in the same $G$-orbit as $x$.
			
			As usual $G$ is equipped with its standard generating set. Let $\psi:G\to G$ be defined by letting $\psi(g)$ be the unique element with smallest word length such that $\psi(g) G_v = g G_v$ (uniqueness comes from the normal form in the right-angled Artin group, in the sense of graph products \cite{Gre}).
			
			Let $\theta'=\psi\circ\theta$, and let $c'':H\times Y_G\to G$ be defined by  $c''(h,x)=\theta'(h\cdot x)^{-1}\hat{c}(h,\tilde{x})\theta'(x)$. By definition of $\psi$, for every $x\in X$, there exists $g_x\in G_v$ such that $\theta'(x)=\theta(x)g_x$. Using this and the fact that $c'(H_v\times Y_{G_v})\subseteq G_v$, we deduce that $c''(H_v\times Y_{G_v})\subseteq G_v$. In addition, a normal form argument  shows that $|c''(h,x)|_{G_v}\le |\hat{c}(h,\tilde{x})|_{G}$ for every $h\in H_v$ and $x\in Y_{G_v}$ (indeed, one writes $\theta'(hx)c''(h,x)\theta'(x)^{-1}=\hat{c}(h,\tilde{x})$, and observes that the subword $c''(h,x)$ on the left cannot be shortened by cancellations with $\theta'(x)$ and $\theta'(hx)$ in view of the choice of $\theta'$).
			
			Therefore the cocycle $c'':H_v\times Y_{G_v}\to G_v$ is $L^1$-integrable. It is $G_v$-cohomologous to $c'$, and therefore it is a measure equivalence cocycle from $H_v$ to $G_v$, for the coupling $\Omega_v$. This completes our proof.
		\end{proof}

\subsection{Integrable embeddings}

Let $G,H$ be countable groups, with $G$ finitely generated. Let $(X,\mu)$ be a standard probability space equipped with a measure-preserving action of $H$ by Borel automorphisms. We say that a cocycle $c:H\times X\to G$ is \emph{$L^1$-integrable} if for every $h\in H$, one has \[\int_{X}|c(h,x)|_G\; d\mu(x)<+\infty.\] In the sequel of this section, we will work with the notion of integrable embeddings, as defined by Bowen in the appendix of \cite{Aus}.

\begin{de}[{Integrable embedding (Bowen \cite[Definition~B.8]{Aus})}]
    Let $G,H$ be countable groups, with $G$ finitely generated. Let $(X,\mu)$ be a standard probability space equipped with a measure-preserving action of $H$ by Borel automorphisms. 

    A cocycle $c:H\times X\to G$ is an \emph{$L^1$-integrable embedding} if for every $\varepsilon>0$, there exist an $L^1$-integrable cocycle $c':H\times X\to G$ which is cohomologous to $c$, a Borel subset $X_0\subseteq X$ with $\mu(X_0)\ge 1-\varepsilon$, and a constant $C=C(\varepsilon)$, such that for every $h\in H$ and almost every $x\in X_0$, one has $|\{g\in G \mid gx\in X_0 \text{~and~} c'(g,x)=h\}|<C.$ 
\end{de}

The following elementary lemma, analogous to Bowen's \cite[Theorem~B.9]{Aus}, enables to check the second property from the above definition for measure equivalence cocycles. 
		
		\begin{lemma}\label{lemma:k-to-one}
			Let $\Omega$ be a measure equivalence coupling between two countable groups $G$ and $H$, let $X_G\subseteq\Omega$ be a fundamental domain for the $H$-action, and let $c:H\times X_G\to G$ be the associated measure equivalence cocycle. Let $\varepsilon>0$.
			
			Then there exist $C>0$ and a measurable subset $X_\varepsilon$ of $X_G$ with $\mu(X_\varepsilon)\ge\mu(X_G)-\varepsilon$ such that for every $x\in X_\varepsilon$ and every $g\in G$, one has \[|\{h\in H \mid h\cdot x\in X_\varepsilon \text{~and~} c(h,x)=g\}|<C.\]
		\end{lemma}
		
		\begin{proof}
			Let $X_H$ be a fundamental domain for the $H$-action on $\Omega$. Then $\Omega$ is covered by countably many pairwise disjoint $H$-translates of $X_H$. For every $h\in H$, let $X_{G,h}=X_G\cap hX_H$. Then the subsets $X_{G,h}$ form a countable partition of $X_G$. Choose a finite subset $F\subseteq H$ such that $\mu(\cup_{h\in F}X_{G,h})>\mu(X_G)-\varepsilon$. Let $C=|F|$, and let $X_\varepsilon=\cup_{g\in F}X_{G,h}$.
			
			Let now $g\in G$ and $x\in X_\varepsilon$. We claim that if two distinct elements $h_1,h_2\in H$ satisfy $c(h_1,x)=c(h_2,x)=g$, then $h_1x$ and $h_2x$ belong to different subsets of the partition $g^{-1}X_G=\sqcup g^{-1}X_{G,h}$. Indeed, notice first that as $c(h_1,x)=c(h_2,x)=g$, we have $h_1x,h_2x\in g^{-1}X_G$. Arguing by contradiction that the claim fails, there exists $h\in H$ such that $h_1x,h_2x\in g^{-1}(X_G\cap hX_H)$. In particular $gh_1x,gh_2x\in hX_H$. As the actions of $G$ and $H$ on $\Omega$ commute, it follows that $h_1$ and $h_2$ both send $gx$ to a common fundamental domain (namely $hX_H$) of the $H$-action on $\Omega$. This is a contradiction, which proves the claim. 

   When considering the $H$-action on $X_G$, the above claim translates as follows: if two distinct elements  $h_1,h_2\in H$ are such that $c(h_1,x)=c(h_2,x)=g$, then $h_1\cdot x$ and $h_2\cdot x$ belong to different subsets of the partition $X_G=\sqcup_{h\in H}X_{G,h}$. Since only $C$ of these subsets are contained in $X_\varepsilon$, the lemma follows.
		\end{proof}

An immediate corollary of Lemma~\ref{lemma:k-to-one} is the following, see also \cite[Theorem~B.9]{Aus}.

\begin{cor}\label{cor:l1-embedding}
Let $G$ and $H$ be countable groups, with $G$ finitely generated. If there is an $(L^1,L^0)$-measure equivalence coupling from $H$ to $G$, then there is an $L^1$-integrable embedding from $H$ to $G$.
\qed
\end{cor}  

		\subsection{Stabilizers of rank $1$ vertices are virtually cyclic}
		
	In this section, we prove Proposition~\ref{prop:stabilizer-L1}. We already know that there exists an $(L^1,L^0)$-measure equivalence coupling from $H_v$ to $G_v$ (Lemma~\ref{lemma:L1}). If we additionally knew that $H_v$ were finitely generated, then we could use a theorem of Bowen \cite[Theorem~B.10]{Aus} to deduce that the growth of $H_v$ is at most equal to the growth of $G_v$, whence at most linear. Being infinite, $H_v$ is therefore virtually isomorphic to $\mathbb{Z}$. However, we do not know \emph{a priori} that $H_v$ is finitely generated. In this section, we extend Bowen's theorem in order to deal with this issue, see Theorem~\ref{theo:fg}. 
	
Given a finitely generated group $G$, with a finite generating set $S$, we let $B_{G,S}(e,n)=\{g\in G\mid |g|_S\le n\}$. The \emph{growth function} $\gr_{G,S}:\mathbb{N}\to\mathbb{N}$ is defined as \[\gr_{G,S}(n)=|B_{G,S}(e,n)|.\] Given two functions $f,g:\mathbb{N}\to\mathbb{N}$, write $f\preceq g$ if there exists $A$ such that $f(n)\le Ag(An+A)+A$. Write $f\sim g$ if $f\preceq g$ and $g\preceq f$: this is an equivalence relation on maps $\mathbb{N}\to\mathbb{N}$. The growth functions associated to two finite generating sets of the same group are always equivalent; we denote by $\gr_G$ their equivalence class. 
	When $S$ is clear from context, we will also simply write $B_G(e,n)$ instead of $B_{G,S}(e,n)$.

		\begin{lemma}\label{lemma:growth}
			Let $L$ be a countable group, let $G$ be a finitely generated group, and assume that there exists an $(L^1,L^0)$-measure equivalence coupling from $L$ to $G$.
			
			Then for every finitely generated subgroup $L'$ of $L$, one has $\gr_{L'}\preceq \gr_{G}$. 
		\end{lemma}
		
		\begin{proof}
			Lemma~\ref{lemma:k-to-one} ensures that any $L^1$-integrable measure equivalence cocycle $L\times X\to G$ is an $L^1$-integrable embedding. In particular, the same holds true for the restriction $L'\times X\to G$. The corollary is therefore a consequence of Bowen's theorem regarding the behavior of growth with respect to $L^1$-integrable embeddings \cite[Theorem~B.10]{Aus}.
		\end{proof}
		
		Given a probability measure-preserving action $G\actson (X,\mu)$, a Borel subset $Z\subset X$ and $x\in X$, the \emph{return time set} of $x$ with respect to $Z$ is $R_Z(x)=\{g\in G\mid gx\in Z\}$. 
		\begin{lemma}
			\label{lemma:Bowen}
			Let $L$ and $J$ be finitely generated groups, with $L$ infinite, equipped with word metrics $|\cdot|_L$ and $|\cdot|_J$ (associated to finite generating sets).
			Let $L\actson (X,\mu)$ be a measure-preserving action on a standard probability space, and $c:L\times X\to J$ an $L^1$-integrable cocycle.  Let $M'>0$ and $M>0$ be such that for every $g\in L$ satisfying $|g|_L\ge M'$, one has $\displaystyle\int_X|c(g,x)|_J\; d\mu(x)\le M |g|_L$. Let $X_0\subseteq X$ be a measurable subset with $\mu(X_0)\ge 0.9$. 
			
			Then there exists $n_0>0$ such that for every $n\ge n_0$, there exists a measurable subset $Y_n\subset X_0$ with $\mu(Y_n)\ge 0.3$, such that for every $x\in Y_n$, one has 
			\[\frac{|\{g\in B_{L}(e,n):g\in R_{X_0}(x)\ \mathrm{and}\ |c(g,x)|_J\le 120M|g|_{L}\}|}{|B_{L}(e,n)|}\ge 0.1.\]
		\end{lemma}
		
		The existence of $M'$ and $M$ follows from the cocycle identity and subadditivity, see e.g.\ \cite[Section A.2]{BFS}. 
	%Although we state the lemma in this way so we can refer to some special value of $M$ later in the proof of Theorem~\ref{theo:fg}.
		
		\begin{proof}
			We follow Bowen's proof of \cite[Theorem~B.10]{Aus} very closely. For $g\in L$, we write $\kappa(g)=\displaystyle\int_X|c(g,x)|_J\; d\mu(x)$.
			As $L$ is infinite and there are only finitely many elements in $L$ such that $\kappa(g)> M|g|_L$, there exists $n_0$ such that for every $n\ge n_0$, one has
			\[\frac{1}{|B_L(e,n)|}\sum_{g\in B_L(e,n)}\frac{\kappa(g)}{|g|_L}\le 2M.\]
			We now fix $n\ge n_0$, and find $Y_n$ as in the lemma. As 
			\[\int_X\left(\frac{1}{|B_L(e,n)|}\sum_{g\in B_L(e,n)}\frac{|c(g,x)|_J}{|g|_L}\right) d\mu(x)\le 2M,\]
			we deduce from the Markov inequality that there exists a measurable subset $X_1\subset X$ with $\mu(X_1)\ge 0.9$ such that \[\frac{1}{|B_L(e,n)|}\sum_{g\in B_L(e,n)}\frac{|c(g,x)|_J}{|g|_L}\le 20 M\] for any $x\in X_1$. For every $x\in X_1$, let 
			\[L_x=\{g\in B_L(e,n) \mid |c(g,x)|_J\le 120M |g|_L\},\]
			then $|L_x|\ge (5/6)|B_L(e,n)|$. Let $X_2=X_0\cap X_1$. Then $\mu(X_2)\ge 0.8$. By \cite[Lemma~B.11]{Aus}, 
			\[\int_{X_2}\frac{|R_{X_2}(x)\cap B_L(e,n)|}{|B_L(e,n)|}d\mu(x)\ge 2\mu(X_2)-\mu(X)\ge 0.6.\]
			Thus
			\begin{align*}
				&\int_{X_2}\frac{|R_{X_0}(x)\cap L_x|}{|B_L(e,n)|}d\mu(x)\ge \int_{X_2}\frac{|R_{X_2}(x)\cap L_x|}{|B_L(e,n)|}d\mu(x)\\
				&\ge \int_{X_2}\frac{|R_{X_2}(x)\cap B_L(e,n)|+|L_x|-|B_L(e,n)|}{|B_L(e,n)|}d\mu(x)\ge 0.6+\mu(X_2)(5/6-1)\ge 0.4.
			\end{align*}
			Then the measurable subset
			\[X_3:=\left\{x\in X_2 \Bigm| \frac{|R_{X_0}(x)\cap L_x|}{|B_L(e,n)|}\ge 0.1\right\}\]
			satisfies $\mu(X_3)\ge 0.3$, otherwise
			\begin{align*}
				&\int_{X_2}\frac{|R_{X_0}(x)\cap L_x|}{|B_L(e,n)|}d\mu(x)= \int_{X_3}\frac{|R_{X_0}(x)\cap L_x|}{|B_L(e,n)|}d\mu(x)+\int_{X_2\setminus X_3}\frac{|R_{X_0}(x)\cap L_x|}{|B_L(e,n)|}d\mu(x)\\
				&\le \mu(X_3)+0.1\mu(X_2\setminus X_3)\le \mu(X_3)+0.1\mu(X_2)<0.4,
			\end{align*}
			which is a contradiction. Now the lemma follows with $Y_n=X_3$.
		\end{proof}
		
		Recall that a group $L$ is \emph{locally virtually cyclic} if every finitely generated subgroup of $L$ is virtually cyclic (possibly finite). We are grateful to Yves Cornulier for providing us with the following lemma and its proof, improving a lemma we were using in an earlier version of this paper.
		
	\begin{lemma}[Cornulier]\label{lemma:cornulier}
	    Let $L$ be a locally virtually cyclic countable group with bounded finite subgroups. Then $L$ has a characteristic subgroup of finite-index which is torsion-free and locally cyclic.
	\end{lemma}
		
		\begin{proof}
		 We will assume that $L$ is infinite, as otherwise the conclusion is obvious. We will take advantage of a result of Wall \cite[Lemma~4.1]{Wal}, ensuring that every group which is virtually isomorphic to $\mathbb{Z}$ surjects onto either $\mathbb{Z}$ or onto the infinite dihedral group $D_\infty$ with finite kernel. This justifies the case disjunction made in the next two paragraphs.
	 
		 We first assume that every infinite finitely generated subgroup of $L$ surjects onto $\mathbb{Z}$ with finite kernel. In this case the set $F$ of finite-order elements of $L$ is a subgroup, which by assumption is finite (and it is characteristic). The quotient $L/F$ is torsion-free, so every finitely generated subgroup of $L/F$ is isomorphic to $\mathbb{Z}$. Let $M\subseteq L$ be the centralizer of $F$. Then $M$ has finite index in $L$, it is characteristic, and it is central-by-(locally cyclic) hence abelian. Let $n$ be the order of $F\cap M$. Then the subgroup $M^n$ consisting of $n^{\text{th}}$ powers of elements of $M$ intersects $F$ trivially, and has finite index, so is the desired locally cyclic finite-index characteristic subgroup of $M$.  
		 
		 We now assume that $L$ contains a subgroup $H$ that maps to $D_\infty$ with finite kernel. Let $\theta_H:H\to\mathbb{Z}/2\mathbb{Z}$ be the homomorphism given by the action of $H$ on its two ends. Every finitely generated subgroup $G$ of $L$ containing $H$ must then also surjects onto $D_\infty$ with finite kernel, and the homomorphism $\theta_G$ extends $\theta_H$. Altogether these give a homomorphism $L\to\mathbb{Z}/2\mathbb{Z}$, whose kernel is a characteristic finite-index subgroup to which the above paragraph applies.
		\end{proof}
		
		We record the following corollary, which is the form in which Lemma~\ref{lemma:cornulier} will be applied in the sequel.
		
		\begin{cor}\label{cor:locally-virtually-cyclic}
			Let $L$ be a locally virtually cyclic countable group with bounded finite subgroups. Then either $L$ is virtually cyclic, or else there exists an infinite cyclic subgroup $L_1\subseteq L$ such that for every $k\in\mathbb{N}$, there exists an infinite cyclic subgroup $L_k\subseteq L$ that contains $L_1$ as a subgroup of index at least $k$. 
		\end{cor}

\begin{proof}
Let $L^0$ be a torsion-free locally cyclic finite-index subgroup of $L$: this exists by Lemma~\ref{lemma:cornulier}. If $L^0$ is finitely generated, then $L$ is virtually cyclic. Otherwise, we can find an infinite strictly increasing sequence $L_1\subseteq L_2\subseteq\dots$ of finitely generated subgroups of $L^0$, whence all isomorphic to $\mathbb{Z}$. In particular $L_k$ contains $L_1$ as a subgroup of index at least $k$ (in fact at least $2^{k-1}$), and we are done. 
\end{proof}

		\begin{theo}
			\label{theo:fg}
Let $L$ be a countable group with bounded finite subgroups, and assume that there exists an $L^1$-integrable embedding from $L$ to $\mathbb{Z}$.			
			
			Then $L$ is virtually cyclic (possibly finite).
		\end{theo}
		
		\begin{proof}
			We write $J=\mathbb{Z}$. We choose once and for all a generator for $J$, and consider the associated word length on $J$ for the rest of the argument. For $C>0$, we say that a map between two sets $f:E_1\to E_2$ is \emph{at most $C$-to-1} if the cardinality of each point inverse of $f$ is at most $C$.
			 
 Without loss of generality, we can assume that $L$ is infinite. By Lemma~\ref{lemma:growth}, every finitely generated subgroup of $L$ has growth at most linear, and therefore is virtually cyclic (possibly finite), see e.g.\ \cite{Jus}. Assuming towards a contradiction that $L$ is not virtually cyclic, let $L_1\subseteq L$ be an infinite cyclic subgroup of $L$ given by  Corollary~\ref{cor:locally-virtually-cyclic}. We will prove that there exists $K>0$ such that for every other infinite cyclic subgroup $L_2$ containing $L_1$, one has $[L_2:L_1]\le K$. This will contradict Corollary~\ref{cor:locally-virtually-cyclic} and conclude our proof.

   By assumption, there exists a measure-preserving action of $L$ on a standard probability space $(X,\mu)$, a cocycle $c:L\times X\to J$ which is $L^1$-integrable, a Borel subset $X_0\subseteq X$ with $\mu(X_0)\ge 0.9$, and a constant $C>0$, such that for every $x\in X_0$, the restriction $c(\cdot,x)_{|R_{X_0}(x)}$ is at most $C$-to-$1$.  
			
			Let $L_2$ be an infinite cyclic subgroup of $L$ that contains $L_1$, and let $k=[L_2:L_1]$. For every $i\in\{1,2\}$, we choose a generator $g_i$ of $L_i$, and we endow $L_i$ with the word metric with respect to this choice of generator. For $g\in L_2$, let \[\kappa(g)=\int_X|c(g,x)|_J\; d\mu(x).\] Let $M_1=\kappa(g_1)$ and $D=\max\{\kappa(g_2^i)\}_{0\le i\le k-1}$. Let $g\in L_2$, and write it as $g=\ell g_1+ig_2$ with $\ell\in\mathbb{Z}$ and $0\le i\le k-1$. Using the cocycle relation, we see that $\kappa(g)\le |\ell| M_1+D$. On the other hand $|g|_{L_2}=k|\ell|\pm i$. Therefore, there exists $M'\in\mathbb{N}$ such that $\kappa(g)\le \frac{2M_1}{k}|g|_{L_2}$ whenever $|g|_{L_2}\ge M'$.
			
			Recall that $\mu(X_0)\ge 0.9$. We can therefore apply Lemma~\ref{lemma:Bowen} with $M=\frac{2M_1}{k}$, and with $L_2$ in place of $L$, and deduce that for $n$ large enough, there exists a Borel subset $Y_n\subseteq X_0$ with $\mu(Y_n)\ge 0.3$ such that for every $x\in Y_n$, one has \[\frac{|\{g\in B_{L_2}(e,n):g\in R_{X_0}(x)\ \mathrm{and}\ |c(g,x)|_J\le 240\frac{M_1}{k}|g|_{L_2}\}|}{|B_{L_2}(e,n)|}\ge 0.1.\]
			Let $B_{L_2,x}^{\mathrm{good}}(e,n)$ be the set defined as in the numerator of the above expression. 
			By definition of $X_0$, for $x\in Y_n$, the map $c(\cdot,x)$ restricted to $B_{L_2,x}^{\mathrm{good}}(e,n)$ is at most $C$-to-1. On the other hand, \[c(B_{L_2,x}^{\mathrm{good}}(e,n)\times\{x\})\subseteq B_{J}\left(e,\frac{240M_1n}{k}\right),\] 
			thus \[|B_{L_2,x}^{\mathrm{good}}(e,n)|\le \frac{480CM_1n}{k}+C.\] 
			For $n$ large enough, it follows that
			\[\frac{481CM_1n/k}{|B_{L_2}(e,n)|}\ge 0.1.\]
			As $|B_{L_2}(e,n)|=2n+1$, we deduce that $\frac{481CM_1n}{2kn}\ge 0.1$. Thus $k\le 2405CM_1$.
		\end{proof}

		\begin{rk}\label{rk:odometer}
			Theorem~\ref{theo:fg} does not hold if $L$ is not assumed to have bounded finite subgroups. The simplest example is to take $L=\bigoplus_{\mathbb{N}}\mathbb{Z}/2\mathbb{Z}$. Then $L$ and $\mathbb{Z}$ have orbit equivalent actions on $X=\{0,1\}^{\mathbb{N}}$ (equipped with the product measure of the uniform measure on $\{0,1\}$). Here $\mathbb{Z}$ acts as the odometer, and $L$ acts by coordinatewise addition; two elements are in the same orbit if and only if they have the same tail. Letting $c:L\times X\to \mathbb{Z}$ be the orbit equivalence cocycle, it is easy to see that $|c(s,\cdot)|$ is bounded for every $s\in L$. 		
		\end{rk}
		
\begin{proof}[Proof of Proposition~\ref{prop:stabilizer-L1}]
		Let $v\in V\B$ be a rank $1$ vertex. By Lemma~\ref{lemma:L1}, there is an $(L^1,L^0)$-measure equivalence coupling from $H_v$ to $G_v$. By Corollary~\ref{cor:l1-embedding}, this yields in particular an $L^1$-integrable embedding from $H_v$ to $G_v$. Since $G_v$ is isomorphic to $\mathbb{Z}$, Theorem~\ref{theo:fg} implies that $H_v$ is virtually cyclic. And $H_v$ is infinite because it is measure equivalent to the infinite group $G_v$, so $H_v$ is virtually isomorphic to $\mathbb{Z}$.
		\end{proof}

		\section{Controlling the factor actions}\label{sec:factor-action}
	
 Throughout the section, we let $G$ be a right-angled Artin group with $|\Out(G)|<\infty$, and $H$ be a countable group with bounded finite subgroups, such that there exists an $(L^1,L^0)$-measure equivalence coupling $\Omega$ from $H$ to $G$.  Let $\hat\Omega$ be the measure equivalence coupling between $H$ and $\hat{G}$ defined by letting $\hat{\Omega}=(\hat{G}\times\Omega)/G$ as in Section~\ref{sec:integrable-stab}. Let $\iota:H\to\Aut(\Gamma^e)$ and $\theta:\hat\Omega\to\Aut(\Gamma^e)$ be the maps given by Lemma~\ref{lemma:action-on-B} through the canonical isomorphism $\Aut(\B)\simeq\Aut(\Gamma^e)$ recalled in Section~\ref{sec:extension-graph}.  In particular $H$ acts on $\Gamma^e$, and therefore $H$ acts on $G$ by flat-preserving bijections, see Section~\ref{sec:autos}.

For $\sfv\in V\Gamma^e$, we let $\hat\Omega_\sfv=\theta^{-1}(\Stab_{\Aut(\Gamma^e)}(\sfv))$ and $H_\sfv=\iota^{-1}(\Stab_{\Aut(\Gamma^e)}(\sfv))$. As usual we denote by $G_\sfv$ and $\hat{G}_\sfv$ the stabilizers of $\sfv$ for the actions of $G$ and $\hat{G}$, respectively. Since the actions of $\hat{G}$ and of $\Aut(\Gamma^e)$ on $\Gamma^e$ have the same orbits of vertices (Lemma~\ref{lemma:transitivity-extension-graph}), it follows from Corollary~\ref{cor:ME-coupling-restriction} that $\hat\Omega_\sfv$ is a measure equivalence coupling between the stabilizers $\hat{G}_\sfv$ and $H_\sfv$. 

 Let $\mathcal{L}_\sfv$ be the set of all $\sfv$-lines in $G$.  Recall from Section~\ref{sec:qi-criterion} that the union $P_\sfv$ of all $\sfv$-lines in $G$ is a left coset of form $gG_{\st(v)}$ for some $v\in V\Gamma$.  Recall also that the action $H_\sfv\actson G$ is by flat-preserving bijections and that $H_\sfv$ preserves $P_\sfv$ and maps $\sfv$-lines to $\sfv$-lines. In particular, there is an induced action $H_\sfv\actson \mathcal L_\sfv$. More generally the action of $\Aut_\sfv(\B)$ on $G$ by flat-preserving bijections preserves $P_\sfv$ and sends $\sfv$-line to $\sfv$-line. We denote by $\theta_\sfv:\hat\Omega_\sfv\to\Bij(P_\sfv)$ and $\iota_\sfv:H_\sfv\to\Bij(P_\sfv)$ the induced maps.	

 The goal of this section is to understand the action $H_\sfv\actson P_\sfv$ to the extent that we can connect with the assumptions of Theorem~\ref{theo:QI}. We will first prove a few group theoretical properties of $H_\sfv$, namely $H_\sfv$ is finitely generated (Section~\ref{subsec:fg}), and $H_\sfv$ contains a commensurated $\mathbb Z$ subgroup (Section~\ref{sec:commensurability}), which is actually virtually normal (Section~\ref{subsec:normal}), before we prove the main lemma (Lemma~\ref{lemma:uniform-qi}) on the action  $H_\sfv\actson P_\sfv$ in Section~\ref{subsec:conjugate}.

\subsection{Finite generation of $H_\sfv$}
\label{subsec:fg}
The main goal of this subsection is to prove that $H_\sfv$ is finitely generated (Lemma~\ref{lemma:hv-fg} below). Recall that finite generation has already been established for $H$-stabilizers of vertices of rank at most
1 in $\B$ (Proposition~\ref{prop:stabilizer-L1}), and this will be used to prove the finite generation of $H_\sfv$.

Let $\B(\sfv)$ be the union of all cubes in $\B$ whose vertices correspond to standard flats that are contained in $P_\sfv$. Then $\B(\sfv)$ is isomorphic to the right-angled building associated with $G_{\st(v)}$, in particular it is simply connected. Moreover, $\B(\sfv)$ is $H_\sfv$-invariant.

\begin{lemma}\label{lemma:cocompact-on-bv}
			The action of $H_\sfv$ on $P_\sfv$ has finitely many orbits of vertices. In particular,
   \begin{enumerate}
   \item the action of $H_\sfv$ on $\call_\sfv$ has finitely many orbits;
   \item for every $\sfv$-line $\ell\in\call_\sfv$, the action of $\Stab_{H_\sfv}(\ell)$ on $\ell$ has finitely many orbits of vertices;
   \item the action of $H_\sfv$ on $\B(\sfv)$ is cocompact.
   \end{enumerate}
		\end{lemma}

		\begin{proof}
The action of $G_\sfv$ on $P_\sfv$ is transitive. We can thus apply Corollary~\ref{cor:ME-coupling-orbits} with $\G=G_\sfv$ and $\sfH=H_\sfv$, with $K=P_\sfv$, and with the coupling $\hat\Omega_\sfv$ and the maps $\iota_\sfv$ and $\theta_\sfv$. Since vertices in $P_\sfv$ have trivial $G_\sfv$-stabilizer, the first part of Corollary~\ref{cor:ME-coupling-orbits} shows that the action of $H_\sfv$ on $P_\sfv$ has finitely many orbits.

  The first two consequences follow because $H_\sfv$ sends $\sfv$-lines to $\sfv$-lines. For Assertion~3, note that $H_\sfv$ acts on the set of rank 0 vertices of $\B(\sfv)$, which is exactly $P_\sfv$, with finitely many orbits. So Assertion~3 follows by exactly the same argument as in Corollary~\ref{cor:cocompact}.
\end{proof}		
			
		\begin{lemma}
			\label{lemma:fg}
			The $H$-stabilizer of any vertex in $\B$ is finitely generated.
		\end{lemma}
		
		\begin{proof}
			Recall that the $H$-action on $\B$ can be equivalently viewed as an $H$-action on $G$ by flat-preserving bijections. Using this viewpoint, for every vertex $v\in V\B$ with associated standard flat $F$, the stabilizer $H_v$ coincides with the stabilizer $H_F$ of $F$ for this $H$-action on $G$. We will use this viewpoint throughout the proof.
			
			We first claim that for any standard flat $F'\subset F$, the group $H_{F'}$ has a finite index subgroup $H^0_{F'}$ which is contained in $H_F$, and acts on $F'$ with finitely many orbits. Indeed, let $v_{F'}$ be the vertex of $\B$ associated to $F'$. Lemma~\ref{lemma:cocompact} ensures that $H_{F'}$ acts on the set of rank 0 vertices of $\B$ that are smaller than $v_{F'}$ with finitely many orbits. Therefore $H_{F'}$ acts on $F'$ with finitely many orbits. Any element in $H_{F'}$ sends $F$ to another standard flat containing $F'$. As there are only finitely many standard flats containing $F'$, we can find a finite index subgroup $H^0_{F'}\subseteq H_{F'}$ that stabilizes $F$. The action $H^0_{F'}\actson F'$ still has finitely many orbits, thus the claim is proved.

			We now prove the lemma by induction on the rank of vertices of $\B$. Stabilizers in $H$ of rank $0$ vertices are finite (Lemma~\ref{lemma:cocompact}), and the case of rank 1 vertices is given by Proposition~\ref{prop:stabilizer-L1}. Let now $v\in V\B$ be a vertex of rank $n\ge 2$, corresponding to a standard flat $F$, and assume by induction that the lemma is proven for all vertices of rank at most $n-1$. Take a standard line $\ell\subset F$. By the above claim, there is a finite index subgroup $H^0_{\ell}\subseteq H_{\ell}$ that stabilizes both $F$ and $\ell$, and whose action on $\ell$ has finitely many orbits. Let $\mathcal C$ be the set of all standard flats in $F$ of dimension $\dim F-1$ that intersect $\ell$ in exactly one point. As the $H$-action on $F$ is flat-preserving, the group $H_\ell^0$ permutes elements in $\mathcal C$, and the permutation action $H_\ell^0\actson\mathcal C$ has finitely many orbits. Let $\{F_1,\dots,F_k\}$ be a set consisting of exactly one representative in each orbit of this permutation action. By the above claim, for each $i\in\{1,\dots,k\}$, there is a finite index subgroup $H^0_{F_i}\subseteq H_{F_i}$ that stabilizes both $F$ and $F_i$, and acts on $F_i$ with finitely many orbits. Let $K_i$ be a finite subset of $F_i$ such that $H^0_{F_i} K_i=F_i$. Let $K=\cup_{i=1}^kK_i$ and let $H'$ be the subgroup of $H_{F}$ generated by $H^0_{\ell}$ and $\{H^0_{F_i}\}_{i=1}^k$. As $H^0_{\ell}$ and $H^0_{F_i}$ are finitely generated by induction, it follows that $H'$ is finitely generated. Moreover, our construction implies that $H'K=F$. Thus for any $h\in H_{F}$, there exists $h'\in H'$ such that $h'h K\cap K\neq\emptyset$.
			On the other hand, as the action $H_{F}\actson F$ has finite stabilizers (Lemma~\ref{lemma:cocompact}) and $K$ is finite, there are only finitely many elements $h\in H_{F}$ such that $hK\cap K\neq\emptyset$. Thus $H'$ has finite index in $H_{F}$, and hence $H_{F}$ is finitely generated.
		\end{proof}

		\begin{cor}\label{cor:fg}
			The $H$-stabilizer of any cube in $\B$ is finitely generated.
		\end{cor}
		
		\begin{proof}
			Let $C$ be a cube of $\B$, and let $v$ be the (unique) vertex of $\B$ of minimal rank. Then $v$ is the vertex of minimal rank in only finitely many cubes of $\B$. As the $H$-action on $\B$ preserves ranks of vertices (Lemma~\ref{lemma:rank}), it follows that $\Stab_H(C)$ is a finite-index subgroup of $\Stab_H(v)$. The corollary thus follows from Lemma~\ref{lemma:fg}.
		\end{proof}
		
		Recall that the $H$-action on $\B$ induces an $H$-action on the extension graph of $G$. We will also need the following fact.
		
		\begin{lemma}\label{lemma:hv-fg}
			For every $\sfv\in V\Gamma^e$, the stabilizer $H_\sfv$ is finitely generated.
		\end{lemma}
		
		\begin{proof}
First we claim that $H_\sfv$ is of finite index in $\Stab_H(P_\sfv)$. Indeed, recall that $P_\sfv=gG_{\st(v)}$ for some $v\in V\Gamma$. Let $\st(v)=\{v\}\circ \Gamma_1\circ\cdots \circ \Gamma_k$ be the join decomposition of $\st(v)$. This gives $P_\sfv=gG_{\st(v)}\cong g\langle v\rangle\times g G_{\Gamma_1}\times \cdots\times gG_{\Gamma_k}$. As $\Stab_H(P_\sfv)$ sends standard flats to standard flats in $P_\sfv$, its action on $P_\sfv$ respects this product decomposition (it can possibly permute the factors). Then $\Stab_H(P_\sfv)$ has a finite index subgroup sending $\sfv$-lines to $\sfv$-lines, which is $H_\sfv$.

	If $F\subseteq P_\sfv$ is a standard flat (corresponding to a vertex of $\B(\sfv)$), then $F$ is contained in only finitely many regions of the form $P_\sfw$ with $\sfw\in V\Gamma^e$.  Then $\Stab_H(F)$ has a finite index subgroup preserving $P_\sfv$. By the claim in the previous paragraph, it follows that $\Stab_{H_\sfv}(F)$ is a finite-index subgroup of $\Stab_H(F)$, and is thus finitely generated by Lemma~\ref{lemma:fg}. More generally, by the same argument (and using Corollary~\ref{cor:fg}), for any cube $C$ in $\B(\sfv)$, the $H_\sfv$-stabilizer of $C$ has finite index in its $H$-stabilizer, hence is finitely generated.
			
			Now $H_\sfv$ acts by cubical automorphisms on the simply connected complex $\B(\sfv)$ cocompactly (Lemma~\ref{lemma:cocompact-on-bv}(3)), with finitely generated cell stabilizers, so it is finitely generated by \cite[Theorem~1]{Brown}.
		\end{proof}

  \subsection{Commensurability of stabilizers}\label{sec:commensurability}

 Let $Z_\sfv=g\langle v\rangle$. The splitting $P_{\sfv}=gG_{\st(v)}=g(\langle v\rangle\times G_{\lk(v)})$ gives a projection map $\pi_1:P_\sfv\to Z_\sfv$. On the other hand, as $\mathcal L_\sfv$ is the collection of $\sfv$-lines, there is a map $\pi_2:P_\sfv\to \mathcal L_\sfv$ sending each point to the $\sfv$-line containing this point. Moreover, $\mathcal L_\sfv$ can be naturally identified with $gG_{\lk(v)}$. Note that the map $(\pi_1,\pi_2)$ gives an identification between $P_\sfv$ and $Z_\sfv\times \mathcal L_\sfv$. Moreover, the action of $G_\sfv$ on $P_\sfv$ respects this product decomposition, hence gives two factor actions:
	\begin{enumerate}
		\item $G_\sfv\actson Z_\sfv$, for which the stabilizer of each element is $gG_{\lk(v)}g^{-1}$;
		\item  $G_{\sfv}\actson \mathcal L_\sfv$, for which the stabilizer of each element is $g\langle v\rangle g^{-1}$.
	\end{enumerate} 
	
	We let $Z_1=Z_{\sfv}$ and $Z_2=\mathcal{L}_\sfv$. We denote by $\Aut_\sfv(\B)$ the stabilizer of $\sfv$ in $\Aut(\B)$ (under the canonical isomorphism between $\Aut(\B)$ and $\Aut(\Gamma^e)$ recalled in Section~\ref{sec:autos}). Note that
	\begin{enumerate}
		\item $\Aut_\sfv(\B)$ preserves $P_\sfv$;
		\item $\Aut_\sfv(\B)$ sends standard flats in $P_\sfv$ to standard flats in $P_\sfv$;
		\item $\Aut_\sfv(\B)$ sends $\sfv$-lines to $\sfv$-lines.
	\end{enumerate}	
	Thus  $\Aut_\sfv(\B)$ preserves the splitting $P_\sfv=Z_1\times Z_2$.
	For every $i\in\{1,2\}$, we have factor actions of $\Aut_{\sfv}(\B)=\Stab_{\Aut(\Gamma^e)}(\sfv)$, $G_\sfv$ and $H_\sfv$ on $Z_i$. Given $x,y\in Z_i$, we denote by $\Aut_{\sfv,x}(\B)$, $G_{\sfv,x}$ and $H_{\sfv,x}$ the stabilizers of $x$ in $\Aut_\sfv(\B)$, $G_{\sfv}$ and $H_{\sfv}$ respectively, and by $G_{\sfv,x,y}$ and $H_{\sfv,x,y}$ the common stabilizers of $x$ and $y$.
	
	\begin{prop}\label{prop:commensurable}
		For every $i\in\{1,2\}$, and any $x,y\in Z_i$, the groups $H_{\mathsf{v},x}$ and $H_{\mathsf{v},y}$ are commensurable in $H_{\mathsf{v}}$.
	\end{prop}
	
	\begin{proof}
		By symmetry, it is enough to prove that $H_{\sfv,x,y}$ has finite index in $H_{\sfv,x}$. Our proof will use two facts regarding the action of $G_{\sfv}$ on $Z_i$ (proved below), namely:
		\begin{itemize}
			\item (Fact 1) The actions of $G_\sfv$ and of $\Aut_\sfv(\B)$ have the same orbits on $Z_i$ (they are both transitive).
			\item (Fact 2) For any $x,z\in Z_i$, we have $G_{\sfv,x}=G_{\sfv,z}$.
		\end{itemize}
		For Fact 1, note that the action $G_\sfv\actson P_\sfv$ is transitive and respects the splitting $P_\sfv\cong Z_\sfv\times \mathcal L_\sfv$, so the action $G_\sfv\actson Z_i$ is transitive for $i=1,2$. Hence the same holds for the bigger group $\Aut_\sfv(\B)$. Fact 2 follows from the discussion of stabilizers of factor actions before the proposition.
		
  As before, let $\hat\Omega$ be a measure equivalence coupling between $\hat G$ and $H$, and let $\theta:\hat\Omega\to \Aut(\B)\cong \Aut(\Gamma^e)$ be a measurable $(\hat G\times H)$-equivariant map. As recalled in the introductory paragraph of Section~\ref{sec:factor-action}, the space 
		$\hat\Omega_{\mathsf{v}}=\theta^{-1}(\Aut_\sfv(\B))$ is a measure equivalence coupling between $\hat{G}_{\mathsf{v}}$ and $H_{\mathsf{v}}$. Hence $\hat\Omega_{\sfv}$ is also a measure equivalence coupling between $G_{\sfv}$ and $H_{\sfv}$.
		
		Let $\hat\Omega_{\mathsf{v},x}=\theta^{-1}(\Aut_{\mathsf{v},x}(\B))$. Since the orbits of $x$ under $\Aut_{\mathsf{v}}(\B)$ and $G_{\mathsf{v}}$ coincide (Fact~1 above), it follows from  Corollary~\ref{cor:ME-coupling-restriction} (applied with $K=Z_i$, with $L=\Aut_\sfv(\B)$, with $\G=G_\sfv$, and with $\hat\Omega_\sfv$ in place of $\Omega$) that $\hat\Omega_{\mathsf{v},x}$ is a measure equivalence coupling between $G_{\mathsf{v},x}$ and $H_{\mathsf{v},x}$.
		
		 We want to apply Lemma~\ref{lemma:finite index} with $\mathsf G=G_{\sfv,x}, \mathsf H=H_{\sfv,x}, \Sigma=\hat\Omega_{\sfv,x}$ and $\mathsf H'=H_{\sfv,x,y}$. It remains to find $\Sigma'$ with desired properties. For the following discussion, we refer to the statement of Lemma~\ref{lemma:action-on-B} for our convention of how $\mathsf G$ and $\mathsf H$ act on $\Aut_{\mathsf{v},x}(\B)$.
		
		Given $z\in Z_i$, we let $\Aut_{\mathsf{v},x,y\to z}(\B)$ be the Borel subset of $\Aut_{\mathsf{v},x}(\B)$ consisting of all automorphisms that send $y$ to $z$, and let $\hat\Omega_{\mathsf{v},x,y\to z}=\theta^{-1}(\Aut_{\mathsf{v},x,y\to z}(\B))$. Then $\hat\Omega_{\mathsf{v},x}=\dunion_{z\in Z_i}\hat\Omega_{\mathsf{v},x,y\to z}$.
		Therefore, we can (and will) choose $z\in Z_i$ such that $\mu(\hat\Omega_{\mathsf{v},x,y\to z})>0$. We take $\Sigma'=\hat\Omega_{\mathsf{v},x,y\to z}$. It is invariant under $G_{\mathsf{v},x,z}$ and $\mathsf H'=H_{\mathsf{v},x,y}$, and the former group is equal to $\mathsf G= G_{\sfv,x}$ by Fact 2. Now take $h\in \mathsf H\setminus \mathsf H'$. Then $h(x)=x$ and $h(y)\neq y$. Then $h\Sigma'=\hat\Omega_{\mathsf{v},x,h(y)\to z}$, which is disjoint from $\hat\Omega_{\mathsf{v},x,y\to z}$, as desired.
	\end{proof}

			\subsection{From commensurated to normal}
	\label{subsec:normal}
	Let $\sfv\in V\Gamma^e$. Proposition~\ref{prop:commensurable}, applied to $Z_2=\mathcal{L}_\sfv$, shows that the $H_\sfv$-stabilizers of any two $\sfv$-lines are commensurable  (and they are virtually infinite cyclic by Proposition~\ref{prop:stabilizer-L1}). In this section, we will improve this by showing that $H_\sfv$ contains a normal subgroup that preserves all $\sfv$-lines: this is Proposition~\ref{prop:normal} below. We start with a lemma.
	
	\begin{lemma}
		\label{lemma:index1}
		Let $\sfv\in V\Gamma^e$. Let $\ell_1,\ell_2$ be two $\sfv$-lines in the same $H_\sfv$-orbit.
		Let $A$ be any finite-index infinite cyclic subgroup of $\stab_{H_\sfv}(\ell_1)\cap\stab_{H_\sfv}(\ell_2)$. 
		
		Then $[\stab_{H_\sfv}(\ell_1):A]=[\stab_{H_\sfv}(\ell_2):A]$.
	\end{lemma}
	
	\begin{proof}
		Recall that every $\sfv$-line $\ell$ determines a rank $1$ vertex $v_\ell$ in $\B$, and vertices of $\ell$ correspond to rank $0$ vertices adjacent to $v_\ell$. Therefore, for every $\sfv$-line $\ell$, 
		%associated to an element of $V^1(\B)\cap\B(\sfv)$, 
		the action of $\Stab_{H_\sfv}(\ell)$ on $\ell$ has finite stabilizers (Lemma~\ref{lemma:cocompact}) and finitely many orbits (Lemma~\ref{lemma:cocompact-on-bv}). As $\Stab_{H_\sfv}(\ell_1)$ and $\Stab_{H_\sfv}(\ell_2)$ are commensurable (Proposition~\ref{prop:commensurable} applied to $Z_2=\mathcal{L}_\sfv$), and $A$ is torsion-free and has finite index in $\Stab_{H_\sfv}(\ell_1)\cap\Stab_{H_\sfv}(\ell_2)$, it follows that the actions $A\actson \ell_1$ and $A\actson \ell_2$ are free with finitely many orbits.
		
		As $A\subseteq H_\sfv$, the $A$-action on $P_\sfv=Z_\sfv\times \mathcal L_\sfv$ preserves the product structure. Therefore the two actions $A\actson \ell_1$ and $A\actson \ell_2$ have the same number of orbits. Let $h\in H_\sfv$ be such that $h(\ell_1)=\ell_2$. Then $h$ induces a bijection $h_\ast:\calo_1^j\mapsto\calo_2^j$ between the set of orbits $\{\mathcal{O}^1_1,\dots,\mathcal{O}^k_1\}$ of $\stab_{H_\sfv}(\ell_1)\actson \ell_1$, and the set of orbits $\{\mathcal{O}^1_2,\dots,\mathcal{O}^k_2\}$ of $\stab_{H_\sfv}(\ell_2)\actson \ell_2$. In addition $h_\ast$ preserves the (finite) cardinalities of stabilizers, i.e.\ the $\Stab_{H_\sfv}(\ell_1)$-stabilizer of any point in $\mathcal{O}_1^j$ has the same cardinality (denoted by $k_j$) as the $\Stab_{H_\sfv}(\ell_2)$-stabilizer of any point in $\mathcal{O}_2^j$. Now, for every $i\in\{1,2\}$, the number of $A$-orbits on $\ell_i$ is equal to $[\Stab_{H_\sfv}(\ell_i):A]\sum_{j=1}^k \frac{1}{k_j}$. Therefore $[\Stab_{H_\sfv}(\ell_1):A]=[\Stab_{H_\sfv}(\ell_2):A]$. 
	\end{proof}

	\begin{prop}
		\label{prop:normal}
		For every $\sfv\in V\Gamma^e$, there exists an infinite cyclic normal subgroup ${N_\sfv\unlhd H_\sfv}$ which preserves every $\sfv$-line.
		%the subgroup $N_\sfv$ fixes $s$ and has finite index in $H_s$.
	\end{prop}
	
	\begin{proof} 
		As the action of $H_\sfv$ on $\call_\sfv$ has finitely many orbits (Lemma~\ref{lemma:cocompact-on-bv}), we can (and will) choose a finite subset $\call_0\subseteq\call_\sfv$ such that $\call_\sfv= H_\sfv\call_0$.
		Recall from Lemma~\ref{lemma:hv-fg} that $H_\sfv$ is finitely generated. 
		Take a finite generating set $S$ of $H_\sfv$ containing the trivial element, and let $\call_1=\cup_{s\in S}s^{-1}\call_0$.

		For every $\ell\in\call_1$,  recall that $\Stab_{H_\sfv}(\ell)$ is virtually infinite cyclic by Proposition~\ref{prop:stabilizer-L1}.  Let $k_\ell\in\mathbb{N}$ be the smallest integer such that the intersection of all subgroups of $\Stab_{H_\sfv}(\ell)$ of index $k_\ell$ is infinite cyclic, and let $Z_\ell$ be this intersection. Notice that $Z_\ell$ is a characteristic subgroup of $\Stab_{H_\sfv}(\ell)$.
		We observe that if $\ell\in\call_0$ and $s\in S$, then $s^{-1}\ell\in\call_1$ and $Z_{s^{-1}\ell}=s^{-1}Z_\ell s$: indeed $h\mapsto s^{-1}hs$ determines an isomorphism between $\Stab_{H_\sfv}(\ell)$ and $\Stab_{H_\sfv}(s^{-1}\ell)$, so $k_\ell=k_{s^{-1}\ell}$ and the isomorphism $h\mapsto s^{-1}hs$ sends $Z_\ell$ to $Z_{s^{-1}\ell}$ in view of the definition of these subgroups.
		
		Let $N_\sfv=\cap_{\ell\in\call_1}Z_\ell$. We claim that $sN_\sfv s^{-1}=N_\sfv$ for every $s\in S$. Indeed, let $s\in S$, and take $\ell\in \call_0$. Then $N_\sfv$ has finite index in $\stab_{H_\sfv}(\ell)$. As both $\ell$ and $s^{-1}\ell$ belong to $\call_1$, Lemma~\ref{lemma:index1} shows that $[\stab_{H_\sfv}(s^{-1} \ell):N_\sfv]=[\stab_{H_\sfv}(\ell):N_\sfv]$. As $N_\sfv\subseteq \stab_{H_\sfv}(s^{-1}\ell)$, we have $sN_\sfv s^{-1}\subseteq \stab_{H_\sfv}(\ell)$ and $[\stab_{H_\sfv}(s^{-1} \ell):N_\sfv]=[\stab_{H_\sfv}(\ell):sN_\sfv s^{-1}]$. Thus $sN_\sfv s^{-1}$ and $N_\sfv$ are subgroups of $\stab_{H_\sfv}(\ell)$ of the same finite index. In addition, $N_\sfv$ is contained in $Z_\ell$ and in $Z_{s^{-1}\ell}$, so $sN_\sfv s^{-1}\subseteq Z_\ell$. Therefore $[Z_\ell:N_\sfv]=[Z_\ell:sN_\sfv s^{-1}]$. As $Z_\ell$ is infinite cyclic, it follows that $s N_\sfv s^{-1}=N_\sfv$ for every $s\in S$. As $S$ generates $H_\sfv$, it follows that $N_\sfv$ is a normal subgroup of $H_\sfv$. As $N_\sfv$ preserves each line in $\call_0$, and $H_\sfv\call_0=\call_\sfv$, it follows that $N_\sfv$ preserves every $\sfv$-line.
	\end{proof}

			\subsection{Conjugating the factor action to one by uniform quasi-isometries}\label{subsec:conjugate}
			 
			We now consider the factor action  $\alpha_\sfv:H_\sfv\to\Bij(Z_\sfv)$. Let $N_\sfv\unlhd H_\sfv$ be a normal infinite cyclic subgroup that preserves every $\sfv$-line, given by Proposition~\ref{prop:normal}.
			
			As $N_{\sfv}$ is torsion-free, by Lemma~\ref{lemma:cocompact}, the action of $N_\sfv$ on each $\sfv$-line is free, and by Lemma~\ref{lemma:cocompact-on-bv} it has finitely many orbits. 
						Let $\{\calo_1,\dots,\calo_n\}$ be the set of orbits for the action of $N_\sfv$ on $Z_\sfv$. Since $N_\sfv$ is normal in $H_\sfv$, the group $H_\sfv$ acts by permutations of the set $\{\calo_1,\dots,\calo_n\}$. Let $H_\sfv^0\subseteq H_\sfv$ be a finite index subgroup which preserves every orbit $\calo_i$.
			
			Let $a$ be a generator of $N_\sfv$. We can (and will) identify each $\calo_i$ with $\mathbb{Z}$ in such a way that $a$ acts by translation by $1$ on each $\calo_i$. Since $N_\sfv$ is normal in $H_\sfv$, the group $H^0_\sfv$ acts on each $\calo_i$ by isometries: indeed, every element $h\in H^0_\sfv$ either commutes with $a$, and then satisfies $h(x+1)=h(x)+1$ for $x\in \calo_i\cong \mathbb Z$, whence acts by translations; or else $hah^{-1}=a^{-1}$, in which case $h(x+1)=h(x)-1$ and $h$ acts by an orientation-reversing isometry.
			
			Let $H_\sfv^1\subseteq H_\sfv^0$ be the finite-index subgroup consisting of all elements that act by positive isometries (i.e.\ translations) on each $\calo_i$. Let $\tau_i:H_\sfv^1\to\mathbb{Z}$ be the translation length homomorphism on $\calo_i$.
			
			\begin{lemma}\label{lemma:translation-lengths}
				For any $i,j\in\{1,\dots,n\}$ and any $h\in H_\sfv^1$, one has $\tau_i(h)=\tau_j(h)$.
			\end{lemma}
			
			\begin{proof}
				Arguing towards a contradiction, let $h\in H_\sfv^1$ be such that $\tau_i(h)\neq\tau_j(h)$. Let $x\in \calo_i$ and $y\in\calo_j$. Then all powers of $ha^{-\tau_i(h)}$ belong to $H_{\sfv,x}\setminus H_{\sfv,y}$. This contradicts the fact that $H_{\sfv,x}$ and $H_{\sfv,y}$ are commensurable (Proposition~\ref{prop:commensurable} applied to $Z_1=Z_\sfv$).  
			\end{proof}
			
			\begin{lemma}\label{lemma:uniform-qi}
				The action of $H_\sfv$ on $Z_\sfv$ is conjugate to an action on $\mathbb Z$ by uniform quasi-isometries.
			\end{lemma}
			
			\begin{proof}
				Let $\sigma:H_\sfv\to\mathfrak{S}(\{1,\dots,n\})$ be the homomorphism given by the permutation of the orbits $\calo_i$. Recall that we have fixed a bijection between every orbit $\calo_i$ and $\mathbb{Z}$, such that the action of $a$ on each $\calo_i$ is by translation by $1$.
				
				Take $h\in H_\sfv$. We claim that for each $i\in\{1,\dots,n\}$, the map $h_{|\calo_i}:\calo_i\cong \mathbb Z\to \calo_{\sigma(h)(i)}\cong \mathbb Z$ is an isometry. Indeed, either $hah^{-1}=a$, in which case $h(x+1)=h(x)+1$ for any $x\in\calo_1\cup\dots\cup\calo_n$. Or else $hah^{-1}=a^{-1}$, in which case $h(x+1)=h(x)-1$ for any $x\in \calo_1\cup\dots\cup\calo_n$. This proves the claim. Moreover, either $h_{|\calo_i}$ is a translation for all $i$, or $h_{|\calo_i}$ is a reflection for all $i$.
				
				We now choose a bijection between $Z_\sfv$ and $\mathbb{Z}$ such that, under this bijection, each $\calo_i$ is identified to the subset $\{mn+i\}_{m\in\mathbb{Z}}$, and the action of $a$ is by translation by $n$. It follows from the previous paragraph that for every $h\in H_\sfv$, there exist integers $c_1(h),\dots,c_n(h)$ such that either $h(mn+i)=(m+c_i(h))n+\sigma(h)(i)$ for any $m\in\mathbb Z$ and any $i\in\{1,\dots,n\}$, or else $h(mn+i)=(c_i(h)-m)n+\sigma(h)(i)$ for any $x\in\mathbb Z$ and any $i\in\{1,\dots,n\}$. Thus each $h\in H_\sfv$ acts on $\mathbb Z$ by a quasi-isometry. We are now left with proving that the quasi-isometry constants are in fact uniform.
				
				Lemma~\ref{lemma:translation-lengths} implies that the action of $H_\sfv^1$ on $Z_\sfv\cong\mathbb{Z}$ is by translations (in particular by uniform quasi-isometries). Recall that $H_\sfv^1$ has finite index in $H_\sfv$. Let $F$ be a finite set of representatives of the left cosets of $H_\sfv^1$ in $H_\sfv$. Any $h\in H$ can then be decomposed as $h=fh'$ for some $f\in F$ and some $h'\in H_\sfv^1$. Thus the quasi-isometry constant of $h$ is bounded by a constant that only depends on the quasi-isometry constants of elements of $F$. This concludes our proof.
			\end{proof}

			\section{Proof of the main theorem}\label{sec:proof}
			
			We are now in position to complete the proof of our main theorem.
			
			\begin{proof}[Proof of Theorem~\ref{theointro:main}]
				The conclusion is obvious if $G=\{1\}$. The case where $G$ is isomorphic to $\mathbb{Z}$ is a consequence of Theorem~\ref{theo:fg}. From now on we assume that $G$ is not cyclic.
				
				Let $\B$ be the right-angled building of $G$. Lemma~\ref{lemma:action-on-B} yields an action of $H$ on $\B$ by cubical automorphisms (equivalently, an action of $H$ on $G$ by flat-preserving bijections). Lemma~\ref{lemma:cocompact} ensures that this flat-preserving action has finite point stabilizers, and finitely many orbits of vertices. Finally, Lemma~\ref{lemma:uniform-qi} ensures that for every $\sfv\in V\Gamma^e$, the factor action $H_\sfv\actson Z_\sfv$ is conjugate to an action on $\mathbb{Z}$ by uniform quasi-isometries. Therefore, Theorem~\ref{theo:QI} applies and shows that $H$ is finitely generated and quasi-isometric to $G$.
			\end{proof}
			
		We now show the importance of the assumption on bounded finite subgroups in our theorem, by giving two examples of infinitely generated groups $H$ with unbounded finite subgroups such that there exists an $(L^\infty,L^0)$-measure equivalence coupling $\Omega$ from $H$ to a right-angled Artin group $G$ (possibly with $|\Out(G)|<+\infty$). The $(L^\infty,L^0)$-condition means that there exists a Borel fundamental domain $X_G$ for the $G$-action on $\Omega$ such that, denoting by $c:H\times X_G\to G$ the associated cocycle, for every $h\in H$, $c(h,\cdot)$ takes essentially only finitely many values.
			
			\begin{prop}
				Let $K$ be a locally finite countable polyhedral complex, let $G$ be a finitely generated group acting properly discontinuously and cocompactly on $K$, and let $H$ be a lattice in $\Aut(K)$. 
				
				Then there exists an $(L^\infty,L^0)$-measure equivalence coupling from $H$ to $G$.
			\end{prop}
			
			In particular, let $G$ be a non-abelian right-angled Artin group, and let $H$ be an infinitely generated (non-uniform) lattice in the automorphism group of the universal cover of the Salvetti complex of $G$ (examples were constructed in \cite[Section~4.2]{HH}, and these have unbounded finite subgroups). Then there exists an $(L^\infty,L^0)$-measure equivalence coupling from $H$ to $G$.
			
			\begin{proof}
				We view $\Aut(K)$, equipped with its Haar measure, as a measure equivalence coupling between $G$ and $H$, where the action of $G\times H$ is via $(g,h)\cdot f=gfh^{-1}$. As $G$ acts cocompactly on $K$, we can find a finite set $V_0$ of representatives of the $G$-orbits of vertices in $K$. Fix $v_0\in V_0$. Then $X=\{f\in\Aut(K)\mid f(v_0)\in V_0\}$ is a measurable fundamental domain for the $G$-action on $\Aut(K)$. We let $c:H\times X\to G$ be the associated measure equivalence cocycle.  We aim to prove that for every $h\in H$, $c(h,\cdot)$ takes only finitely many values.
				
				Fix $h\in H$. Let $d$ be the metric on the $1$-skeleton of $K$ obtained by assigning length $1$ to every edge, and considering the induced path metric.  For every $f\in X$, we have \[d(v_0,f(h^{-1} v_0))\le d(v_0,f(v_0))+d(v_0,h^{-1} v_0)\le\mathrm{diam}(V_0)+d(v_0,h^{-1} v_0).\] Since the $G$-action on $K$ is properly discontinuous, we can find a finite set $B_h\subseteq G$ such that for every $f\in X$, one has $f(h^{-1}v_0)\in B_h^{-1}V_0$.  In particular, there exists $g(f)\in B_h$ such that $g(f)fh^{-1}\in X$. This shows that $c(h,f)\in B_h$ and concludes our proof. 
			\end{proof}

An \emph{$(L^\infty,L^0)$-orbit equivalence coupling} from $H$ to $G$ is a measure equivalence coupling $\Omega$ for which there exists a common Borel fundamental domain $X$ for the actions of $G$ and $H$ on $\Omega$, such that denoting by $c:H\times X\to G$ the associated cocycle, for every $h\in H$, $c(h,\cdot)$ essentially takes only finitely many values. The existence of an $(L^\infty,L^0)$-orbit equivalence coupling is equivalent to requiring that $G$ and $H$ admit free, measure-preserving actions on a standard probability space $X$ with the same orbits, so that for the natural orbit equivalence cocycle $c:H\times X\to G$, for every $h\in H$, $c(h,\cdot)$ essentially only takes finitely many values. 
   
			\begin{prop}
				Let $\Gamma$ be a finite simplicial graph. Let $H$ be any graph product over $\Gamma$ with vertex groups isomorphic to $\bigoplus_{\mathbb{N}}\mathbb{Z}/2\mathbb{Z}$. 
				
				Then there exists an $(L^\infty,L^0)$-orbit equivalence coupling from $H$ to $G$.
			\end{prop}
			
			\begin{proof}
				There exists an $(L^\infty,L^0)$-orbit equivalence coupling from $\bigoplus_{\mathbb{N}}\mathbb{Z}/2\mathbb{Z}$ to $\mathbb{Z}$, see Remark~\ref{rk:odometer}. It thus follows from \cite[Proposition~4.2]{HH} that $G$ and $H$ are orbit equivalent. The proposition follows because the argument in \cite[Proposition~4.2]{HH} preserves the integrability of the coupling (by the computation from \cite[Section~11]{EH}).
			\end{proof}
	
			\section{Lattice embeddings of right-angled Artin groups}\label{sec:lattice-envelope}

In this section, we first prove a theorem that gives restrictions on the possible lattice embeddings of a countable group $H$ with bounded finite subgroups which is measure equivalent to a right-angled Artin group $G$ with $|\Out(G)|<+\infty$ (Theorem~\ref{theo:cocompact}). And we then describe all possible lattice embeddings under an integrability condition on the coupling between $G$ and $H$ (Theorem~\ref{theo:lattice-embedding}). For the second statement, we will first need to introduce the language of quasi-actions, which will be also useful in the next section. 

\subsection{Lattice embeddings and measure equivalence} 

   \begin{theo}\label{theo:cocompact}
       Let $G$ be a non-cyclic right-angled Artin group with $|\Out(G)|<+\infty$, and let $H$ be a countable group with bounded finite subgroups that is measure equivalent to $G$. Let $\mathfrak{H}$ be a locally compact second countable group, and let $\tau:H\to\mathfrak{H}$ be a lattice embedding.

Then $\mathfrak{H}$ maps continuously with compact kernel to a totally disconnected locally compact group, and $\tau$ is cocompact.
   \end{theo}

In the proof, we will make use of the extension of the notion of measure equivalence to unimodular locally compact second countable groups given by \cite[Definition~1.1]{BFS}.

   \begin{proof}
        The group $\mathfrak{H}$ is measure equivalent to $G$. Let $(\Omega,\mu)$ be a measure equivalence coupling between $G$ and $\mathfrak{H}$. Recall that the inclusion $G\subseteq\Aut(\B)$ is strongly ICC (Proposition~\ref{prop:icc}). By Lemma~\ref{lemma:self-coupling}, any self measure equivalence coupling of $G$ is taut relative to the inclusion $G\subseteq\Aut(\B)$ in the sense of \cite[Definition~1.10]{BFS} (the uniqueness of the tautening map in this definition is ensured by \cite[Lemma~A.8(1)]{BFS}). Notice also that the locally compact second countable group $\mathfrak{H}$ is unimodular because it contains $H$ as a lattice. Therefore we can apply \cite[Theorem~2.6]{BFS} and deduce that there exists a continuous homomorphism $\iota:\mathfrak{H}\to\Aut(\B)$ with compact kernel $K$ and closed image.

The intersection $H\cap K$ is finite, whence a lattice in the compact group $K$. Let $\pi:\mathfrak{H}\to\mathfrak{H}/K$ be the projection map. By \cite[Theorem~1.13]{Rag}   (as restated in \cite[Lemma~3.5]{BFS2}), the $\pi$-image of $H$ in $\mathfrak{H}/K$ is again a lattice. Since $\Aut(\B)$ is totally disconnected, so is its closed subgroup $\mathfrak{H}/K$. Since every lattice embedding of a countable group with bounded finite subgroups in a totally disconnected locally compact group is cocompact (see \cite[Corollary~4.11]{BCGM}), it follows that the image of $H$ in $\mathfrak{H}/K$ is cocompact. Let $K'\subseteq\mathfrak{H}/K$ be a compact set whose $H$-translates cover $\mathfrak{H}/K$. The map $\pi$ is closed, continuous, and it has compact fibers because $K$ is compact, see \cite[Theorem~1.5.7]{AT}. It follows that $\pi^{-1}(K')$ is compact, see \cite[Theorem~3.7.2]{Eng}, and therefore $H$ is cocompact in $\mathfrak{H}$.
\end{proof}

		\subsection{Quasi-actions} 

In this section and the next, we will need the notion of a quasi-action of a group. Given $L\ge 1$ and $A\ge 0$, an \emph{$(L,A)$-quasi-action} of a group $\mathfrak{H}$ on a metric space $(Z,d)$ is a map $\rho:\mathfrak{H}\times Z\to Z$ such that 
\begin{itemize}
    \item for every $h\in \mathfrak{H}$, the map $\rho(h,\cdot):Z\to Z$ is an $(L,A)$-quasi-isometry,
    \item for every $h_1,h_2\in \mathfrak{H}$ and every $z\in Z$, one has $d(\rho(h_1,\rho(h_2,z)),\rho(h_1h_2,z))<A$, and
    \item for every $z\in Z$, one has $d(\rho(e,z),z)<A$.
\end{itemize}
A \emph{quasi-action} of $\mathfrak{H}$ on $(Z,d)$ is an $(L,A)$-quasi-action for some $L\ge 1$ and $A\ge 0$.

If $\mathfrak{H}$ is a topological group, the quasi-action $\rho$ is \emph{proper} if for any $z\in Z$ and $R>0$, the set $\{h\in \mathfrak{H}\mid \rho(h,z)\in B_R(z)\}$ has compact closure. Notice that when $\mathfrak{H}$ is discrete, this amounts to requiring that this set is finite. The quasi-action $\rho$ is \emph{cobounded} if there exist $z\in Z$ and $L>0$ such that $Z$ is contained in the $L$-neighborhood of $\rho(\mathfrak{H}\times \{z\})$. 

Two quasi-actions $\rho,\rho'$ of $\mathfrak{H}$ on the same metric space $Z$ are \emph{equivalent} if \[\sup_{h\in \mathfrak{H}}\sup_{z\in Z}d(\rho(h,z),\rho'(h,z))<+\infty.\] More generally, two quasi-actions $\rho,\rho'$ of $\mathfrak{H}$ on two metric spaces $Z,Z'$ are \emph{quasi-conjugate} if there exists a quasi-isometry $\varphi:Z\to Z'$ such that \[\sup_{h\in \mathfrak{H}}\sup_{z\in Z}d(\varphi\circ\rho(h,z),\rho'(h,\varphi(z)))<+\infty.\]

The study of quasi-actions on right-angled Artin groups relies on the following theorem of the second-named author.

\begin{theo}[{\cite[Theorem~4.18]{Hua}}]\label{theo:qi}
 Let $G$ be a non-cyclic right-angled Artin group with $|\Out(G)|<+\infty$. For every $L\ge 1,A\ge 0$, there exists $D=D(A,L)\ge 0$ such that the following holds. For every $(L,A)$-quasi-isometry $q:G\to G$, there exists a unique flat-preserving bijection $q':G\to G$ such that $d(q(x),q'(x))\le D$ for every $x\in G$. 

 In addition, for every $x\in G$, denoting by $\mathcal{F}_x$ the set of all maximal standard flats that contain $x$, one has \[\{q'(x)\}=\bigcap_{F\in\mathcal{F}_x}q_\ast(F),\] where $q_\ast(F)$ is the unique maximal standard flat at finite Hausdorff distance of $q(F)$.
\end{theo}

We refer to \cite[Lemma 4.12 and Equation 4.13]{Hua} for the ``in addition'' statement of the above theorem. We record the following corollary of Theorem~\ref{theo:qi}.

\begin{cor}
\label{cor:equivalent}
Let $G$ be a non-cyclic right-angled Artin group with $|\Out(G)|<+\infty$, and let $\mathfrak{H}$ be a locally compact second countable group. Then any quasi-action $\rho:\mathfrak{H}\times G\to G$ is equivalent to a unique action $\alpha:\mathfrak{H}\times G\to G$ by flat-preserving bijections.  Moreover, if $\rho$ is a measurable map, then $\alpha$ is a continuous map.
 \end{cor}

\begin{proof}
  The existence and uniqueness of $\alpha$ follow from Theorem~\ref{theo:qi}. Its measurability follows from the measurability of $\rho$ and the fact that $\alpha(h,g)=g'$ if and only if there exists $M\in\mathbb{N}$ such that for every maximal standard flat $F$ containing $g$, there exists a maximal standard flat $F'$ containing $g'$ such that for every $x\in F$, there exists $x'\in F'$ such that $d(\rho(h,x),x')\le M$. Equipped with the compact-open topology, the group $\Bij_{\mathrm{FP}}(G)$ is secound-countable. So the measurable map $\mathfrak{H}\to\Bij_{\mathrm{FP}}(G)$ induced by $\alpha$ is in fact automatically continuous \cite[Theorem~B.3]{Zim}. Therefore $\alpha$ is continuous.
\end{proof}

\begin{prop}\label{prop:q-action}
Let $G$ be a non-cyclic right-angled Artin group with $|\Out(G)|<+\infty$. Let $H$ be a finitely generated group that is quasi-isometric to $G$, and let $\mathfrak{H}$ be a compactly generated locally compact second countable group which is quasi-isometric to $G$.

%that contains $H$ as a cocompact lattice.

Then $\mathfrak{H}$ has a proper, cobounded, continuous  action on $G$ by flat-preserving bijective quasi-isometries with uniform constants, and a proper, cocompact, continuous action on a locally finite $\mathrm{CAT}(0)$ cube complex $Y$ which is quasi-isometric to $G$.
\end{prop}

\begin{proof}
Let $X$ be a Cayley-Abels graph of $\mathfrak H$, equipped with the path metric with edge length $1$, and let $\varphi:X\to G$ be a quasi-isometry, with its quasi-inverse given by $\psi:G\to X$. Then the proper and cocompact isometric action $\rho:\mathfrak H\times X\to X$ gives a proper and cobounded quasi-action $\rho':\mathfrak H\times G\to G$ with $\rho'(\cdot,g)=\varphi(\rho(\cdot,\psi(g)))$.

%By \cite[Lemma~28]{MSW}, the group $\mathfrak{H}$ has a proper, cobounded, measurable quasi-action $\rho$ on $H$. Notice that, while properness and measurability of the quasi-action are not stated in \cite[Lemma~28]{MSW}, they follow from the construction. Indeed, to construct $\rho$, one starts with an enumeration $H=\{h_n\}_{n\in\mathbb{N}}$ and chooses a symmetric compact subset $K\subseteq\mathfrak{H}$ containing the identity, such that $HK=\mathfrak{H}$. For $g\in\mathfrak{H}$ and $h\in H$, one lets $\rho(g,h)=h'$, where $h'$ is the first element of the enumeration of $H$ such that $g h\in h' K$. In particular $\mathfrak{H}$ has a proper, cobounded, measurable quasi-action on $G$, through a quasi-isometry between $H$ and $G$.

By Corollary~\ref{cor:equivalent}, there exists a unique continuous flat-preserving action $\alpha:\mathfrak{H}\to\Bij_{\mathrm{FP}}(G)$ such that there exists $D\ge 0$ such that for every $h\in\mathfrak{H}$ and every $g\in G$, one has $d(\rho'(h,g),\alpha(h)(g))\le D$. In particular the action $\alpha$ is by quasi-isometries, and it is proper and cobounded. This proves the first part of the proposition.

The second part of the proposition then follows from \cite[Theorem~6.2]{HK}, and $Y$ arises as a blow-up building as in Section~\ref{sec:qi-criterion}. As in the proof of \cite[Theorem~6.2]{HK}, the properness of the $\mathfrak{H}$-action on $Y$ follows from the properness of  its action on $G$, and the fact that there is an equivariant quasi-isometry from $G$ to the set of rank $0$ vertices of $Y$ (see the paragraphs after the proof of the claim in \cite[p.~587]{HK}). The coboundedness (in fact cocompactness as $Y$ is locally finite) of the action follows from the same argument. The continuity of the action follows from \cite[Lemma 19.29]{Cor-quasi}.
%And the measurability of the $\mathfrak{H}$-action on $Y$ also follows from the measurability of its action on $G$ by the same argument (the argument gives the measurability at the level of rank $0$ vertices, but the action on rank $0$ vertices determines the whole action). Finally, the continuity of the homomorphism $\mathfrak{H}\to\Aut(Y)$ follows again from an automatic continuity statement due to Mackey, see \cite[Theorem~B.3]{Zim}. 
\end{proof}

\subsection{Lattice embeddings under an integrability assumption}

\begin{theo}\label{theo:lattice-embedding}
Let $G$ be a non-cyclic right-angled Artin group with $|\Out(G)|<+\infty$, and let $H$ be a countable group with bounded finite subgroups. Assume that there exists an $(L^1,L^0)$-measure equivalence coupling from $H$ to $G$. Let $\mathfrak{H}$ be a locally compact second countable group in which $H$ embeds as a lattice.

Then there exists a uniformly locally finite $\mathrm{CAT}(0)$ cube complex $Y$ quasi-isometric to $G$, and a continuous homomorphism $\mathfrak{H}\to\Aut(Y)$ with compact kernel and cocompact image. Such $Y$ can be chosen to be a blow-up building in the sense of Section~\ref{subsec:blowup building}.
\end{theo}

\begin{proof} 
 Since there is an $(L^1,L^0)$-measure equivalence coupling from $H$ to $G$, Theorem~\ref{theointro:main} implies that $H$ is finitely generated and quasi-isometric to $G$. Since $H$ is a cocompact lattice in $\mathfrak{H}$ by Theorem~\ref{theo:cocompact}, it follows from Proposition~\ref{prop:q-action} that there exists a continuous proper cobounded action by cubical automorphisms $\alpha:\mathfrak{H}\to\Aut(Y)$ on a $\mathrm{CAT}(0)$ cube complex $Y$ that is quasi-isometric to $G$. The kernel of $\alpha$ is then a closed subgroup of $\mathfrak{H}$, in fact compact by properness of the action. And the cocompactness of the image of $\mathfrak{H}$ in $\Aut(Y)$ follows from the cocompactness of the $\mathfrak{H}$-action on $Y$.  
\end{proof}

   \section{Lack of virtual locally compact model for products}\label{sec:non-rigidity}

In this section, we prove Theorem~\ref{theointro:nonrigidity} from the introduction. More precisely, we start with a right-angled Artin group $G_\Lambda$ with $|\Out(G_\Lambda)|<+\infty$, which splits as a direct product $G_\Lambda=G_{\Gamma_1}\times G_{\Gamma_2}$. We construct a sequence of finite-index subgroups $G_{\Lambda_n}=G_{\Gamma_{1,n}}\times G_{\Gamma_{2,n}}$, and cocompact torsion-free lattices $U_n$ in the automorphism group $\Aut(X_{2n,1}\times X_{2n,2})$ of the universal cover of the Salvetti complex of $\Lambda_{2n}$, such that the groups $U_n$ do not embed as lattices in a common locally compact group $\mathfrak{G}$, even up to passing to finite-index subgroups (see Theorem~\ref{theo:final}).

There are several building blocks in our construction. We start from cocompact lattices in the automorphism group of a product of two trees, coming from the celebrated construction of Burger--Mozes \cite{BM,BM2} -- this will be reviewed in Section~\ref{sec:burger-mozes} below. We will then use two constructions, presented in Section~\ref{sec:extension}, which will allow us to extend the action on a product of trees, to a cocompact action on the product $X_{2n,1}\times X_{2n,2}$. The construction of the groups $U_n$ is completed in Section~\ref{sec:construction}. Checking that the groups $U_n$ do not admit any common lattice embedding, even virtually, will require a very fine analysis on factor actions (which were already key in the proof of the main theorem of this paper). This analysis is carried in Section~\ref{sec:factor-action-more}, and the proof of Theorem~\ref{theointro:nonrigidity} is completed in Section~\ref{sec:end}.

\subsection{Two arguments for extending actions to a bigger complex}\label{sec:extension}

We make the following definition.

\begin{de}[Even action on a tree]\label{de:even}
An action of a group $H$ on a tree $T$ is \emph{even} if it preserves the colours of vertices in any bipartite colouring of the vertex set $VT$.
\end{de}

For each $n\in\mathbb{N}$, let $T_n$ be the regular tree of valence $n$, with each edge of length $2$. %An action of a group $H$ on $T_n$ by automorphisms is \emph{even} if it preserves the colours of vertices in any bipartite colouring of $VT_n$.

Let $J_n$ be the Cayley graph of a rank $n$ free group with respect to a free generating set $X$, with the usual orientation and labeling of edges by elements of $X$, where edges have length $1$. Notice that $J_n$ is a regular tree of valence $2n$. Let $\tilde{X}$ be a finite set in bijection with $X$. We label edges of $J_n\times J_n$ by elements of $X\cup \tilde{X}$: edges with a nondegenerate projection to the first (resp.\ second) factor are labeled by elements of $X$ (resp.\ $\tilde{X}$). A \emph{standard line} in $J_n$ is a line made of edges of the same label. A \emph{standard flat} in $J_n\times J_n$ is a product of two standard lines, one in each factor.

An action $\alpha$ of a group $H$ on a product $T\times T'$ of two trees is \emph{factor-preserving} if there exist actions $\beta:H\to\Aut(T)$ and $\beta':H\to\Aut(T')$ such that $\alpha(h)=(\beta(h),\beta'(h))$ for every $h\in H$ -- in particular $\alpha$ does not swap the two factors. The actions $\beta$ and $\beta'$ are called the \emph{factor actions} of $\alpha$.

\begin{lemma}
	\label{lem:extend1}
	Let $\alpha:H\actson T_n\times T_n$ be a free, cocompact, factor-preserving action of a group $H$, whose factor actions are even.

Then there exist an isometric embedding $\theta:T_n\times T_n\to J_{2n}\times J_{2n}$, a group $H'$ with a factor-preserving, free and cocompact action $\alpha':H'\to\Aut(J_{2n}\times J_{2n})$ which preserves the orientation of edges and sends standard flats to standard flats, and an injective homomorphism $\varphi:H\to H'$, such that $\theta$ is equivariant with respect to $\alpha,\alpha'$ and $\varphi$. 

In addition $\theta$ can be chosen so that for any $x\in V(T_n\times T_n)$, and any two distinct edges $e_1,e_2\in E(T_n\times T_n)$ containing $x$, the labels of $\theta(e_1)$ and $\theta(e_2)$ are different.
\end{lemma}

\begin{proof}
	Let $S=\{s_1,\dots,s_n\}$ be a finite set of cardinality $n$.  We label each edge of $T_n$ by a letter in $S$, in such a way that two edges incident on the same vertex never have the same label. % $\{s_i\}_{i=1}^n$. 
 We fix a bipartition of $T_n$ into \emph{black} and \emph{white} vertices.
	Let $T'_n$ be the barycentric subdivision of $T_n$. Vertices of $T'_n$ are of three types: black vertices (in $V(T_n)$), white vertices (in $V(T_n)$), and \emph{gray} vertices, corresponding to the midpoints of edges of $T_n$. We orient each edge of $T'_n$ so that it points towards its gray vertex. We fix a set $\{a_1,\dots,a_n,a'_1,\dots,a'_n\}$ of labels. 
	Any edge of $T'_n$ from a white (resp.\ black) vertex of $T_n$ to the midpoint of an $s_i$-labeled edge of $T_n$, is labeled by $a_i$ (resp.\ $a'_i$). 
	
	Let $W_{2n}$ be a wedge sum of $2n$ oriented circles, labeled by $a_1,\ldots,a_n,a'_1,\ldots,a'_n$. There is a unique map $\pi:T'_n\to W_{2n}$ which preserves labels and orientations of edges. But $\pi$ is not a covering map. Now we enlarge $T'_n$ to a larger space $K_n$, and extend $\pi$ to a map $\pi:K_n\to W_{2n}$ which is a covering map. This relies on Haglund and Wise's \emph{canonical completion} procedure \cite[Section~6]{HW}, which we explain in this special case.

	We enlarge $T'_n$ as follows.
	For each edge $e$ of $T'_n$ oriented from a vertex $x$ to another vertex $y$, we add an edge $e'$ to $T'_n$ such that $e'$ and $e$ have the same label, but $e'$ is oriented from $y$ to $x$. Next for each white (resp.\ black) vertex $x$ of $T'_n$, we attach $n$ oriented loops based $x$,  labeled by $\{a'_1,\ldots,a'_n\}$ (resp.\ $\{a_1,\ldots,a_n\}$). Finally for each gray vertex $x$ of $T'_n$ which is the midpoint of an $s_i$-labeled edge of $T_n$, we attached $2(n-1)$ oriented loops based at $x$, labeled by $\{a_1,\ldots,\hat a_i,\ldots,a_n,a'_1,\ldots,\hat a'_i,\ldots,a'_n\}$ -- here $\hat a_i,\hat a'_i$ means that we remove these two elements from the set of labels. Now one readily verifies that the map $\pi:T'_n\to W_{2n}$ extends to a label and orientation preserving covering map $\pi:K_n\to W_{2n}$. 
	
	A \emph{standard circle} in $K_n$ is an embedded copy of $\mathbb S^1$ made of edges with the same label. Each standard circle in $K_n$ has either one or two edges.
	
	Now consider the action $H\actson T_n\times T_n$, which gives two factor actions $H\actson T_n$ preserving the bipartite vertex coloring of $T_n$. Any factor action induces an action $H\actson T'_n$ which preserves orientation of edges and colors of vertices of $T'_n$ (though it might not preserve labels of edges of $T'_n$).  And since the number of edge-loops at vertices of a given colour is constant, the action $H\actson T'_n$ extends to an action $H\actson K_n$ which preserves the orientation of edges and sends standard circles to standard circles.
  	
	Now the free and cocompact action $H\actson T_n\times T_n$ extends to a free and cocompact action $H\actson K_n\times K_n$. A \emph{standard torus} in $K_n\times K_n$ is a product of two standard circles, one  from each factor. Then the action $H\actson K_n\times K_n$ preserves the edge orientation and sends every standard torus to a standard torus. Moreover, the subspace $T'_{n}\times T'_n$ is invariant under this action.
	Let $J_{2n}\times J_{2n}$ be the universal cover of $K_n\times K_n$ with the induced label and orientation of edges (where the edges coming from the second factor are labeled by $\tilde{a}_i,\tilde{a}'_i$). Then standard flats in $J_{2n}\times J_{2n}$ are lifts of standard tori in $K_n\times K_n$ -- for this it is important to notice that components of $K_n$ consisting of edges with the same label are reduced to circles.
	
Let $H'\subseteq \Aut(J_{2n}\times J_{2n})$ be the subgroup consisting of all automorphisms that lift automorphisms in $H$. Then the $H'$-action on $J_{2n}\times J_{2n}$ is free and cocompact, preserves the orientation of edges, and sends standard flats to standard flats. In addition, the inclusion map $T'_n\times T'_n\to K_n\times K_n$ lifts to an isometric embedding $\theta:T'_n\times T'_n\to J_{2n}\times J_{2n}$. As $T'_n\times T'_n$ is invariant under the $H$-action on $K_n\times K_n$, such a lift gives an injective homomorphism $\varphi:H\to H'$ such that $\theta$ is equivariant with respect to $\varphi$. Finally, the additional part of the lemma follows from our construction.  
\end{proof}

We need another extension criterion, which is based on an idea of Hughes in \cite[Section~7.2]{Hug}. Our next lemma is a variation over a statement from work of Mj and the second-named author \cite[Section 4.2]{HM}.

Recall that $X_\Gamma$ is the locally finite cube complex canonically associated to $G_\Gamma$, in other words $X_\Gamma$ is the universal cover of the Salvetti complex of $G_\Gamma$. Given a group $H$ acting by flat-preserving bijections on $G_\Gamma$, we have a \emph{type cocycle} $c: H\times VX_\Gamma\to \Aut(\Gamma)$ where $c(h,x)$ with $h\in H$ and $x\in VX_\Gamma$ is defined as follows.
Given a vertex $v\in V\Gamma$, let $\ell$ be the standard line of type $v$ containing $x$. Then $c(h,x)(v)$ is defined to be the type of $h(\ell)$. One readily vertices that $c(h,x):V\Gamma\to V\Gamma$ preserves adjacency of vertices, hence extends to an automorphism of $\Gamma$, and that $c(h,x)$ is indeed a cocycle. 

\begin{lemma}
	\label{lem:extend2}
	Let $\Gamma_1$ be an induced subgraph of $\Gamma_2$.
	Suppose $H$ is a group acting freely and cocompactly on $X_{\Gamma_1}$, preserving the orientation of edges, and sending standard flats to standard flats. Let  $c_1:H\times VX_{\Gamma_1}\to \Aut(\Gamma_1)$ be the associated type cocycle. Suppose that there exists a cocycle $c_2:H\times VX_{\Gamma_1}\to \Aut(\Gamma_2)$ such that $c_2(h,x)_{\mid \Gamma_1}=c_1(h,x)$ for any $h\in H$ and $x\in VX_{\Gamma_1}$.

	Then there exist a group $H'$ acting freely and compactly on $X_{\Gamma_2}$ sending standard flats to standard flats, an injective group homomorphism $\phi:H\to H'$, and a $\phi$-equivariant embedding $j:X_{\Gamma_1}\to X_{\Gamma_2}$ preserving labels and orientations of edges. 
	Moreover, 
	each element of $H'$ sends standard lines labeled by $v\in V\Gamma_2$ to standard lines whose labels belong to the orbit of $v$ under the action of elements in $\{c_2(h,x)\}_{h\in H,x\in VX_{\Gamma_1}}$.
\end{lemma}

\begin{proof}
The inclusion of $\Gamma_1$ as an induced subgraph of $\Gamma_2$ yields an isometric embedding $S_{\Gamma_1}\hookrightarrow S_{\Gamma_2}$ between the Salvetti complexes. By pre-composing this with the covering map $X_{\Gamma_1}\to S_{\Gamma_1}$, we obtain  a local isometric embedding $X_{\Gamma_1}\to S_{\Gamma_2}$. While $X_{\Gamma_1}\to S_{\Gamma_2}$ is not a covering map, we can ``complete'' it to a covering map as follows, using (as in the previous proof) a special case of a construction by Haglund--Wise \cite[Section~6]{HW}. Consider the homomorphism $G_{\Gamma_2}\to G_{\Gamma_1}$ fixing each generator in $\Gamma_1$ and sending all generators in $\Gamma_2\setminus\Gamma_1$ to identity. Let $K$ be the kernel of this homomorphism and $Z$ be the cover of $S_{\Gamma_2}$ corresponding to $K$. Then there is an embedding $X_{\Gamma_1}\to Z$. Under such an embedding, $X_{\Gamma_1}$ and $Z$ have the same vertex set. Moreover, we can obtain the 1-skeleton of $Z$ from the 1-skeleton of $X_{\Gamma_1}$ by attaching a collection of edge loops to each of the vertices of $X_{\Gamma_1}$, one edge loop for each vertex outside $\Gamma_2\setminus \Gamma_1$. The complex $Z$ is called the \emph{canonical completion} of $X_{\Gamma_1}$ with respect to the local  isometric embedding $X_{\Gamma_1}\to S_{\Gamma_2}$. 

Now we define an action of $H$ on the 1-skeleton $Z^{(1)}$ of $Z$ extending the existing action of $H$ on $X^{(1)}_{\Gamma_1}$ as follows. For $h\in H$, we let $h$ send an edge loop based at $x\in X_{\Gamma_1}$ labeled by $v\in V\Gamma_2\setminus V\Gamma_1$ to an edge loop based at $h(x)$ labeled by $c_2(h,x)(v)$, moreover, we require that $h$ respects the orientation of edges. It follows from the construction of the action $H\actson Z^{(1)}$ that it sends a pair of edges with commuting labels based at the same vertex to another pair of edges with commuting labels. Thus $H\actson Z^{(1)}$ extends to an action of $H$ on the 2-skeleton of $Z$. One readily checks that the action extends to higher skeleta as well. This gives an edge orientation preserving action $H\actson Z$ extending the existing action $H\actson X_{\Gamma_1}$. Note that $H\actson Z$ is also free and cocompact. 

Let $H'$ be the subgroup of $\Aut(X_{\Gamma_2})$ consisting of all lifts of automorphisms of $Z$ coming from the action $H\actson Z$. Then $H'$ fits into an exact sequence $1\to \pi_1(Z)\to H'\to H\to 1$. As $H\actson Z$ is free and cocompact, the same holds for the action $H'\actson X_{\Gamma_2}$. The moreover statement of the lemma follows from the construction of $H'$. Let $j:X_{\Gamma_1}\to X_{\Gamma_2}$ be a lift of the embedding $X_{\Gamma_1}\to Z$ with respect to the covering map $X_{\Gamma_2}\to Z$. Now we define an injective homomorphism $\phi:H\to H'$ as follows. Given $h\in H$, let $\alpha_h$ be the automorphism of $Z$ coming from the action $H\actson Z$, and let $\beta_h=(\alpha_h)_{|X_{\Gamma_1}}$, an automorphism of  $X_{\Gamma_1}$. Let $\beta'_h$ be the map $j\circ \beta_h\circ j^{-1}$ defined on $j(X_{\Gamma_1})$. We define $\phi(h):X_{\Gamma_2}\to X_{\Gamma_2}$ to be the unique lift of $\alpha_h:Z\to Z$ with respect to the covering $X_{\Gamma_2}\to Z$ such that $\phi(h)$ is an extension of $\beta'_h$. Then $\phi$ is indeed a group homomorphism, and $j$ is equivariant with respect to $\phi$.
\end{proof}

\subsection{On local actions}\label{sec:burger-mozes}

Given a factor-preserving action of a group $H$ on $T_n\times T_n$, and a vertex $x$ in one of the factor trees, the \emph{local action} of $H$ at $x$ is the action of the $H$-stabilizer of $x$ (with respect to the factor action) on the collection of edges of the factor tree containing $x$. The \emph{local group} at $x$ is the subgroup of the permutation group of edges containing $x$ coming from the local action.

The following can be deduced from work of Lazarovich--Levcovitz--Margolis \cite{LLM}, relying on earlier works by Burger--Mozes \cite{BM2,BM} and Radu \cite{Rad} (recall Definition~\ref{de:even} for the notion of even factor actions).

\begin{lemma}
	\label{lemma:simple}
	For each $n_0>0$, there exist $n\ge n_0$ and a simple  torsion-free group $H$ which acts on $T_n\times T_n$ freely, cocompactly, in a factor-preserving way, with even factor actions, such that for any vertex $x$ in one of the tree factors, the local action of $H$ at $x$ has an orbit of size at least $n/4$.
\end{lemma}

\begin{proof}
	By \cite[Theorem~4.2 and Lemma~4.3]{LLM}, there is a group $H$ acting on $T_n\times T_n$ as a uniform lattice such that all the local groups are the full symmetry groups on $n$ letters. Moreover, $H$ has an index $4$ subgroup $H^+$ which is simple and preserves the bipartition of each factor tree (see the beginning of \cite[Section 2]{LLM} for the definition of $H^+$). In addition $H^+$ is torsion-free by \cite[Lemma~3.1]{Rad}. Thus the group $H^+$ satisfies the requirement of the lemma.
\end{proof}

\subsection{The family of examples}\label{sec:construction}

Let $\Gamma_1$ and $\Gamma_2$ be finite simplicial non-complete
graphs -- at this point we are not making any extra assumption, but in the next section our discussion will be applied to graphs $\Gamma_1,\Gamma_2$ such that $|\Out(G_{\Gamma_i})|<+\infty$ for every $i\in\{1,2\}$. 

For every $i\in\{1,2\}$, we fix a choice of two vertices $v_i,w_i\in V\Gamma_i$, with $w_i\notin \st(v_i)$. For every $n\ge 1$, let $\Gamma_{n,i}$ be the finite simplicial graph obtained by gluing $n$ copies of $\Gamma_i$ along $\st(v_i)$. We will denote by $v_{n,i}$ the image of $v_i$ in $\Gamma_{n,i}$. 
Let $\Gamma^w_{n,i}$ be the totally disconnected subgraph of $\Gamma_{n,i}$ whose vertices are the $n$ copies of $w_i$ in $\Gamma_{n,i}$, denoted by $w_{n,i}[1],\ldots,w_{n,i}[n]$. 

Let $\Lambda=\Gamma_1\circ \Gamma_2$ and $\Lambda_n=\Gamma_{n,1}\circ\Gamma_{n,2}$. 
Let $\Lambda^w_n=\Gamma^w_{n,1}\circ\Gamma^w_{n,2}$, a subgraph of $\Lambda_n$.

Notice that for every $i\in\{1,2\}$ and every $n\ge 1$, the group $G_{\Gamma_{n,i}}$ is isomorphic to the kernel of the homomorphism $\varphi_{n,i}: G_{\Gamma_i}\to \mathbb Z/n\mathbb Z$ that sends $v_i$ to $1$ and all other generators to $0$. This gives an injective homomorphism $q_{n,i}:G_{\Gamma_{n,i}}\to G_{\Gamma_i}$ with finite index image.  More precisely, let $\Gamma_{n,i}[1],\dots,\Gamma_{n,i}[n]$ be the $n$ copies of $\Gamma_i$ in $\Gamma_{n,i}$, where $\Gamma_{n,i}[j]$ is the copy that contains $w_{n,i}[j]$. Then (up to reordering the $\Gamma_{n,i}[j]$, which we do once and for all), for any vertex $u\in\Gamma_{n,i}[j]$ with $u\neq v_{n,i}$, we have $q_{n,i}(u)=v_i^{j-1}\bar uv_i^{-j+1}$, where $\bar u$ denotes the vertex of $\Gamma_i$ that corresponds to $u$ (when writing the above equality, we identify vertices of $\Gamma_i$ and $\Gamma_{n,i}$ with the corresponding elements of $G_{\Gamma_i}$ and $G_{\Gamma_{n,i}}$). And $q_{n,i}(v_{n,i})=v_i^n$. In particular $q_{n,i}(w_{n,i}[j])=v_i^{j-1}w_iv_i^{-j+1}$. 
	
	For every $i\in\{1,2\}$ and every $n\ge 1$, the $q_{n,i}$-image of any standard line of $G_{\Gamma_{n,i}}$ is Hausdorff close to a standard line of $G_{\Gamma_i}$. In addition $q_{n,i}$ sends parallel standard lines to parallel standard lines, up to finite Hausdorff distance. Thus $q_{n,i}$ induces a map from the vertex set of the extension graph $\Gamma_{n,i}^e$ to the vertex set of $\Gamma_i^e$, which extends to a graph isomorphism $(q_{n,i})_*:\Gamma_{n,i}^e\to \Gamma_i^e$. Likewise we have a map $q_n:G_{\Lambda_n}\to G_\Lambda$.

For every $i\in\{1,2\}$, let $X_{n,i}$ be the universal cover of the Salvetti complex of $G_{\Gamma_{n,i}}$. We orient edges of $X_{n,i}$ and label them by vertices of $\Gamma_{n,i}$. As $\{w_{n,i}[1],\ldots,w_{n,i}[n]\}$ generates a free subgroup of $G_{\Gamma_{n,i}}$, we have an embedding $j_{n,i}:J_n\to X_{n,i}$ which preserves the orientation and the labeling of edges, which gives $j_n:J_n\times J_n\to X_{n,1}\times X_{n,2}$. We will often refer to the first product factor as the \emph{horizontal} factor, and the second factor as \emph{vertical}.

\begin{cor}
	\label{cor:embedding properties}
	There exists an infinite subset $\mathcal C\subset \mathbb Z_{\ge 0}$ such that the following is true. For each $n\in \mathcal C$,  there exists a group $V_n$ acting on $T_n\times T_n$ satisfying the requirements of Lemma~\ref{lemma:simple}, a group $U_n$ acting on $X_{2n,1}\times X_{2n,2}$ freely and cocompactly sending standard flats to standard flats, an isometric embedding $\theta_n:T_n\times T_n\to X_{2n,1}\times X_{2n,2}$, and an injective group homomorphism $\phi_n:V_n\to U_n$ such that
	\begin{enumerate}
		\item $\theta_n$ sends each edge of $T_n\times T_n$ (of length 2) to a concatenation of two edges in $X_{2n,1}\times X_{2n,2}$;
		\item $\theta_n$ is equivariant with respect to $\phi_n$;
		\item for each $x\in T_n$, there exists $y\in X_{2n,1}$ such that $\theta(\{x\}\times T_n)\subset \{y\}\times X_{2n,2}$;
		\item for any vertex $x\in T_n\times T_n$ and any two different edges $e_1$ and $e_2$ containing $x$, the edges $\theta_n(e_1)$ and $\theta_n(e_2)$ have different labels, both contained in $\Lambda^w_{2n}$;
		\item for every $i\in\{1,2\}$, each element of $U_n$ sends standard lines labeled by $v_{2n,i}$ to standard lines labeled by $v_{2n,i}$.
	\end{enumerate}
\end{cor}

\begin{proof}
Let $V_n$ be a group given by Lemma~\ref{lemma:simple}. By Lemma~\ref{lem:extend1}, there exist a group $V'_n$ acting geometrically on $J_{2n}\times J_{2n}$, preserving the orientation of edges, and sending standard flats to standard flats, an injective homomorphism $\varphi_n:V_n\to V'_n$, and a $\varphi_n$-equivariant isometric embedding $T_n\times T_n\to J_{2n}\times J_{2n}$. Since the action of $V'_n$ is factor-preserving, the type cocycle of $V'_n\actson J_{2n}\times J_{2n}$ is defined and takes its values in automorphisms of $\Lambda^w_{2n}$ that preserves the two factors $\Gamma^w_{2n,1}$ and $\Gamma^w_{2n,2}$ in $\Lambda^w_{2n}$. Any such automorphism of $\Lambda^w_{2n}$ extends naturally to an automorphism of $\Lambda_{2n}$, because any permutation of the vertex set of $\Gamma^w_{2n,i}$ extends to an automorphism of $\Gamma_{2n,i}$ permuting the copies $\Gamma_{2n,i}[1],\dots,\Gamma_{2n,i}[2n]$. This gives the required extension of cocycles as in the assumption of Lemma~\ref{lem:extend2}. Now the corollary follows.
\end{proof}

\subsection{Auxiliary facts about star projections}\label{sec:factor-action-more}

Given a finite simplicial graph $\Gamma$ with $|\Out(G_\Gamma)|<+\infty$, we collect several facts about certain projections on $G_\Gamma$ for later use. Recall that we identify $G_\Gamma$ as the $0$-skeleton of $X_\Gamma$. 
Given $\sfv\in V\Gamma^e$, take a $\sfv$-line $\ell$. Then $\ell$ is the 0-skeleton of a convex subcomplex of $X_\Gamma$. 
Then there is nearest point projection $\pi_{\ell}:G_\Gamma\to \ell$ sending each point to the nearest point in $\ell$ with respect to the word distance \cite[Lemma~13.8]{HW}. 
It is known that if $\ell_1,\ell_2$ are standard lines with $\Delta(\ell_1)=\Delta(\ell_2)\notin \st(\sfv)$, then $\pi_\ell(\ell_i)$ is a single point for every $i\in\{1,2\}$, and $\pi_\ell(\ell_1)=\pi_\ell(\ell_2)$ (see \cite[Lemma~6.2]{Hua}).
This gives a well-defined map $\pi_{\sfv}:V(\Gamma^e\setminus \st(\sfv))\to \ell$.

Recall that the left action $G_\Gamma\actson \Gamma^e$ induces a projection map $\Gamma^e\to \Gamma$. We say a vertex $\sfv$ of $\Gamma^e$ is of \emph{type $v$} with $v\in\Gamma$, if the map $\Gamma^e\to \Gamma$ sends $\sfv$ to $v$.

Now we continue with the notations from the previous section, and assume in addition that $\Out(G_{\Gamma_1})$ and $\Out(G_{\Gamma_2})$ are finite. In particular, for $i\in\{1,2\}$, let $\Gamma_{n,i}$, $q_{n,i}:G_{\Gamma_{n,i}}\to G_{\Gamma_i}$ and $(q_{n,i})_*:\Gamma^e_{n,i}\to \Gamma_i^e$ be as in the previous section.

\begin{lemma}
	\label{lem:projection1}
	Assume that $\Out(G_{\Gamma_1})$ and $\Out(G_{\Gamma_2})$ are finite. Let $i\in\{1,2\}$. Let $\ell$ be a standard line in $G_{\Gamma_{n,i}}$ such that $\sfv=\Delta(\ell)$ is of type $v_{n,i}$. For $1\le j\neq k\le n$, let $\ell_j$ and $\ell_k$ be two standard lines in $G_{\Gamma_{n,i}}$ intersecting $P_\sfv$ non-trivially such that $\Delta(\ell_j)=\sfw_j$ is of type $w_{n,i}[j]$ and $\Delta(\ell_k)=\sfw_k$ is of type $w_{n,i}[k]$. Let  $\sfv'=(q_{n,i})_*(\sfv)$. Then
	\begin{enumerate}
		\item  $\pi_{\sfv'}((q_{n,i})_*(\sfw_j))\neq \pi_{\sfv'}((q_{n,i})_*(\sfw_k))$.
		\item Write $P_\sfv=gG_{\st(v_{n,i})}$ and let $\mathcal C$ be the collection of standard lines $\ell'$ which have non-trivial intersection with $gG_{\lk(v_{n,i})}$ and satisfy $\Delta(\ell')\notin\st(\sfv)$.
  Let $\mathcal W=\{\Delta(\ell')\}_{\ell'\in \mathcal C}$. Then $\pi_{\sfv'}((q_i)_*(\mathcal W))$ is a finite set of cardinality at most $n$.
	\end{enumerate}
\end{lemma}

\begin{proof}
	As $(q_{n,i})_\ast$ is $G_{\Gamma_{n,i}}$-equivariant, up to translation, we can assume that $P_\sfv=G_{\st(v_{n,i})}$, that $\ell_j=\langle w_{n,i}[j]\rangle$, and that $\ell_k=g_1v_{n,i}^m\langle w_{n,i}[k]\rangle$ for some $m\in\mathbb{Z}$ and some $g_1\in G_{\lk(v_{n,i})}$. Then $q_{n,i}(\ell_j)$ is Hausdorff close to the standard line $v_i^{j-1}\langle w_i\rangle$, and $q_{n,i}(\ell_k)$ is Hausdorff close to the standard line $g_1v_i^{mn+k-1}\langle w_i\rangle$. Identifying $Z_{\sfv'}$ with $\langle v_i\rangle$, we get that $\pi_{\sfv'}((q_{n,i})_*(\sfw_j))=v_i^{j-1}$ and $\pi_{\sfv'}((q_{n,i})_*(\sfw_k))=v_i^{nm+k-1}$, and Assertion 1 follows. Assertion 2 follows from a similar computation. 
\end{proof}

\begin{lemma}
	\label{lem:consistent}
 Let $\Gamma$ be a finite simplicial graph such that $|\Out(G_\Gamma)|<+\infty$. Suppose $H\actson G_\Gamma$ is an action by flat-preserving bijections. Take $\sfv\in V\Gamma^e$ and let $\alpha_\sfv:H_\sfv\actson Z_\sfv$ be the associated factor action.
	Let $\ell\subset G_\Gamma$ be a $\sfv$-line.  Take another standard line $\ell'$ with $\Delta(\ell')\notin \st(\sfv)$.

	Then $\alpha_{\sfv}(h)(\pi_{\sfv}(\Delta(\ell')))=\pi_{\sfv}(\Delta(h(\ell')))$ for any $h\in H_\sfv$.
\end{lemma}

\begin{proof}
As $P_\sfv$ is the vertex set of a convex subcomplex $C_\sfv$ of $X_\Gamma$, by \cite[Lemma~13.8]{HW}, there is a nearest point projection map $\pi_{P_\sfv}:G_\Gamma\to P_\sfv$.
The lemma follows from the fact, established below, that the nearest point projection map $\pi_{P_\sfv}:G_\Gamma\to P_\sfv$ is $H_\sfv$-equivariant.

Now we prove the fact. Let $x\in G_\Gamma$ and $h\in H_\sfv$. Let $y=\pi_{P_\sfv}(x)$. Let $\omega$ be a shortest edge path in the 1-skeleton of $X_\Gamma$ connecting the vertices $x$ and $y\in P_\sfv$. Let $x_0,\dots,x_n$ be
vertices in $\omega$ such that for $0\le i\le n-1$, $[x_i
, x_{i+1}]$ is a maximal sub-segment of $\omega$ that
is contained in a standard line ($x_0=x$ and $x_n=y$). Denote the corresponding
standard line by $\ell_i$. As $y$ can be alternatively characterized as the unique point in $P_\sfv$ such that any hyperplane separating $x$ and $y$ does not cross $C_\sfv$ (see \cite[Lemmas~13.1 and~13.8]{HW}),  $y=\pi_{P_\sfv}(x)$ if and only if  for any $\omega$ and $(\ell_i)$ as above, $\Delta(\ell_i)\notin\st(\sfv)$ for any $0\le i\le n-1$.

For $0\le i\le n-1$, let $\ell'_i$ be the standard line with $h(\ell_i)=\ell'_i$ and let $\omega'_i$ be a geodesic segment in $\ell'_i$ from $h(x_i)$ to $h(x_{i+1})$.   Let $\omega'$ be the concatenation of all the $\omega'_i$ with $0\le i\le n-1$. 
By considering the automorphism $h_*:\Gamma^e\to \Gamma^e$ induced by the flat-preserving bijection $h$, we see that $\Delta(\ell'_i)\notin \st(\sfv)$ for any $0\le i\le n-1$. Thus none of the hyperplanes that have nonempty intersection with $\omega'$ will intersect $C_{\sfv}$. As any hyperplane separating $h(x)$ and $h(y)$ is dual to an edge in $\omega'$, this hyperplane does not cross $C_\sfv$. Thus $h(y)=\pi_{P_\sfv}(h(x))$ by the previous paragraph.
\end{proof}

\subsection{Conclusion}\label{sec:end}

As usual, given an action $\alpha$ of a group $H$ by flat-preserving bijections on $G=G_\Gamma$, for every $\sfv\in V\Gamma^e$, we let $\alpha_\sfv:H_\sfv\actson Z_\sfv$ be the factor action.

\begin{prop}
	\label{prop:quasi-action}
Let $G$ be a non-cyclic right-angled Artin group with $|\Out(G)|<+\infty$. Let $\rho$ be a quasi-action of $H$ on $G$, and let 
	$\alpha:H\actson G$ be the unique $H$-action on $G$ by flat-preserving bijections which is equivalent to $\rho$ (see Corollary~\ref{cor:equivalent}).
	
 Then there exists $C\ge 0$ (depending on $\rho$) such that for every $\sfv\in V\Gamma^e$ and every subgroup $H'_\sfv\subseteq H_\sfv$, 
 if some $H'_\sfv$-orbit of the factor action $\alpha_{\sfv}$ is finite, then every $H'_\sfv$-orbit for this action has cardinality at most $C$. 
\end{prop}

\begin{proof}
	Note that there are constants $L,A>0$, depending on $\rho$, such that for any $\sfv\in V\Gamma^e$, the factor action $\alpha_\sfv:H_\sfv\actson Z_\sfv$ is by bijections which are $(L,A)$-quasi-isometries. Then by \cite[Proposition~6.3]{HK}, there exist a constant  $D$ depending only on $L$ and $A$, an isometric action $\alpha'_\sfv:H_\sfv\actson \mathbb Z$ and a surjective equivariant map $f: Z_\sfv\to \mathbb Z$ such that each point inverse of $f$ has cardinality at most $D$. If $(\alpha_\sfv)_{|H'_\sfv}$ has a finite orbit, then $(\alpha'_\sfv)_{|H'_\sfv}$ has a finite orbit, hence a fixed point. Hence any orbit of $(\alpha'_\sfv)_{|H'_\sfv}$ has cardinality at most $2$, and therefore any orbit of $(\alpha_\sfv)_{|H'_\sfv}$ has cardinality at most $2D$.
\end{proof}

 Let $\Gamma_1$ and $\Gamma_2$ be two finite simplicial graphs such that $\Out(G_{\Gamma_1})$ and $\Out(G_{\Gamma_2})$ are finite. Given this choice, let $\mathcal{C}\subseteq\mathbb{Z}_{\ge 0}$ and the groups $U_n,V_n$ be as in Corollary~\ref{cor:embedding properties}. Recall that $U_n$ acts properly and cocompactly on $X_{2n,1}\times X_{2n,2}$, and so does $G_{\Gamma_{2n,1}}\times G_{\Gamma_{2n,2}}$. In particular $U_n$ and $G_{\Lambda}=G_{\Gamma_1}\times G_{\Gamma_2}$ are strongly commable in the sense recalled before the statement of Theorem~\ref{theointro:nonrigidity} in the introduction. Therefore, Theorem~\ref{theointro:nonrigidity} is a consequence of the following statement -- recall that $U_n$ is torsion-free, so we can focus on lattice embeddings instead of lattice representations with finite kernel. 

\begin{theo}\label{theo:final}
	There does not exist a locally compact second countable topological group $\mathfrak{G}$ such that for each $n\in \mathcal C$, there exists a finite index subgroup $U'_n$ of $U_n$ which admits an embedding $U'_n\to \mathfrak{G}$ whose image is a lattice.
\end{theo}

\begin{proof}
 We argue by contradiction and assume such $\mathfrak{G}$ exists. Recall that $|\Out(G_\Lambda)|<+\infty$. Since $G_{\Lambda_{2n}}$ has finite index in $G_\Lambda$, and since $U_n$ and its finite-index subgroup $U'_n$ are uniform lattices in $\Aut(X_{2n,1}\times X_{2n,2})$, we can apply Theorem~\ref{theo:cocompact} and deduce that the lattice embedding $U'_n\to\mathfrak{G}$ is cocompact.

 Fix $n_0\in\mathcal{C}$. Since $U'_{n_0}$ is quasi-isometric to $G_\Lambda$, we can apply Proposition~\ref{prop:q-action} and get a proper, cobounded action $\alpha_1$ of $\mathfrak{G}$ on $G_\Lambda$ by flat-preserving bijective  quasi-isometries with uniform constants. 
 This gives, for every $n\in\mathcal{C}$, an action by flat-preserving bijections $\alpha_{1,n}:U'_n\actson G_\Lambda$ through the lattice embedding $U'_n\to \mathfrak{G}$. We apply Proposition~\ref{prop:quasi-action} to the action $\alpha_1:\mathfrak{G}\actson G_\Lambda$, and let $C$ be the resulting constant. %Note that $C$ only depends on $\mathfrak{G}$.
	
	On the other hand, for every $n\in\mathcal{C}$, the group $U'_n$ has a properly discontinuous cocompact action on $X_{2n,1}\times X_{2n,2}$. Let $q_{2n}:X_{2n,1}\times X_{2n,2}\to G_\Lambda$ be as in the previous section, after identifying the $0$-skeleton of $X_{2n,1}\times X_{2n,2}$ with $G_{\Lambda_{2n}}$. Then $q_{2n}$ gives a quasi-action of $U'_n$ on $G_\Lambda$, which is equivalent to a unique action by flat-preserving bijections 
	$\alpha_{2,n}:U'_n\actson G_\Lambda$ (see Corollary~\ref{cor:equivalent}). Note that both $\alpha_{1,n}$ and $\alpha_{2,n}$ are  proper and cobounded actions of $U'_n$ on $G_\Lambda$ by quasi-isometries with uniform constants. Thus $\alpha_{1,n}$ and $\alpha_{2,n}$ are quasi-conjugate (as any two proper and cobounded quasi-actions of the same finitely generated group on the same space are quasi-conjugate). By Theorem~\ref{theo:qi}, $\alpha_{1,n}$ and $\alpha_{2,n}$ are actually conjugate via a flat-preserving bijection of $G_\Lambda$. 

 Through $\alpha_{2,n}$, the group $U'_n$ acts on the extension graph $\Lambda^e$. Since the action of $U'_n$ on $X_{2n,1}\times X_{2n,2}$ sends standard flats to standard flats (Corollary~\ref{cor:embedding properties}), we also have an action of $U'_n$ on $\Lambda_{2n}^e$. The map $q_{2n}:X_{2n,1}\times X_{2n,2}\to G_\Lambda$ induces an isomorphism $(q_{2n})_{\ast}:\Lambda_{2n}^e\to\Lambda^e$, which is $U'_n$-equivariant with respect to the above actions.
 
	Let $\{\alpha_{2,n,\sfv}:U'_{n,\sfv}\actson Z_\sfv\}_{\sfv\in V\Lambda^e}$ be the collection of factor actions for $\alpha_{2,n}:U'_n\actson G_\Lambda$.  We observe that for every subgroup $U''_{n,\sfv}\subseteq U'_{n,\sfv}$, every finite orbit of $(\alpha_{2,n,\sfv})_{|U''_{n,\sfv}}$ has cardinality at most $C$: indeed, this follows from the same property  for the action $\alpha_{1,n,\sfv}$, which comes from our choice of constant $C$, and from the fact that the actions $\alpha_{1,n}$ and $\alpha_{2,n}$ are conjugate via a flat-preserving bijection. In the rest of the proof, we will show that as $n$ becomes larger and larger, we can find some $\sfv\in V\Lambda^e$ and some subgroup $U''_{n,\sfv}$ of $U'_{n,\sfv}$ acting on $Z_\sfv$ with finite orbits of larger and larger size, which will be a contradiction.

	Consider the simple subgroup $V_n$ of $U_n$ as above. Then $V_n\subseteq U'_n$. We view $T_n\times T_n$ as a $V_n$-invariant subcomplex of $X_{2n,1}\times X_{2n,2}$ via the embedding $\theta_n$ provided by Corollary~\ref{cor:embedding properties}. A \emph{vertical} (resp.\ \emph{horizontal}) \emph{$T_n$-copy} in $X_{2n,1}\times X_{2n,2}$ is the $\theta_n$-image of $\{z\}\times T_n$ (resp.\ $T_n\times \{z\}$) for some vertex $z$ in the first (resp.\ second) tree factor.  
	Let $\ell\subset X_{2n,1}\times X_{2n,2}$ be a standard line of type $v_{2n,1}$ such that $\ell$ intersects $T_n\times T_n$ in a vertex $x\in X_{2n,1}\times X_{2n,2}$. Let $p_1:T_n\times T_n\to T_n$ be the projection to the first factor. 
	Let $V_{n,x}$ be the $V_n$-stabilizer of $p_1(x)$ with respect to the action of $V_n$ on the first factor of $T_n\times T_n$. 
	
	We claim that each element in $V_{n,x}$ sends $\ell$ to a standard line which is parallel to $\ell$.
	Indeed, by Corollary~\ref{cor:embedding properties}(5), for any $g\in V_{n,x}$, edges in $g(\ell)$ are also labeled by $v_{2n,1}$. On the other hand, $x$ and $g(x)$ are connected by an edge path in a vertical $T_n$-copy whose edge labels commute with $v_{2n,1}$, thus $\ell$ and $g(\ell)$ are parallel.
	
	Let $P_\ell$ be the parallel set of $\ell$ in $G_{\Lambda_{2n}}$. We write $P_\ell=g_0G_{\st(v_{2n,1})}$ and assume without loss of generality that $x\in g_0G_{\lk(v_{2n,1})}$. By the previous paragraph, $V_{n,x}$ stabilizes both $P_\ell$ and $g_0G_{\lk(v_{2n,1})}$.
	
	Let $T_{n,x}^h$ be the horizontal $T_n$-copy that contains $x$.
	Let $\{\ell_1,\dots,\ell_k\}$ be the set of standard lines in $X_{2n,1}\times X_{2n,2}$ passing through $x$ whose intersection with $T_{n,x}^h$ contains at least one edge. By Corollary~\ref{cor:embedding properties}(4), we have $k=n$, and the types of the lines $\ell_1,\dots,\ell_n$ are contained in $\{w_{2n,1}[1],\ldots, w_{2n,1}[2n]\}$.  For every $i\in\{1,\dots,n\}$, let $e_i\subseteq\ell_i$ be an edge based at $x$ and contained in $\ell_i\cap T_{n,x}^h$. Since the action of $U_n$ on $X_{2n,1}\times X_{2n,2}$ sends standard flat to standard flat (Corollary~\ref{cor:embedding properties}), for every $g\in V_{n,x}$ and every $i\in\{1,\dots,n\}$, the image $g\ell_i$ is a standard line, and $g(e_i)$ is an edge based at $gx$ contained in $g\ell_i\cap T_{n,gx}^h$. It thus follows from Lemma~\ref{lemma:simple} and Corollary~\ref{cor:embedding properties}(4) that, up to permuting the lines $\ell_i$, the set of all types of $g\ell_1$ as $g$ varies in $V_{n,x}$, is a subset of $\{w_{2n,1}[1],\dots,w_{2n,1}[2n]\}$ of cardinality at least $n/4$.

	 Let $\ell'_1$ (resp.\ $\ell'$) be a standard line in $G_\Lambda$ at finite Hausdorff distance from $q_{2n}(\ell_1)$ (resp.\ $q_{2n}(\ell)$). Let $\sfv=\Delta(\ell')$. We consider the action $\alpha_{2,n,\sfv}:U'_{n,\sfv}\actson Z_\sfv$. The fact that every element of $V_{n,x}$ sends $\ell$ to a parallel line ensures that $V_{n,x}\subset U'_{n,\sfv}$. 
	
Notice that for every $g\in V_{n,x}$, we have $(q_{2n})_\ast(\Delta(g\ell_1))=\Delta(g\ell'_1)$. Thus, using Lemma~\ref{lem:consistent}, we deduce that the projections $\pi_\sfv((q_{2n})_*(\Delta(g\ell_1)))$, as $g$ varies in $V_{n,x}$, form exactly one orbit of the action $V_{n,x}\actson Z_\sfv$ under $\alpha_{2,n,\sfv}$.
As the types of the lines $g\ell_1$, with $g$ varying in $V_{n,x}$, form a subset of $\{w_{2n,1}[1],\ldots, w_{2n,1}[2n]\}$ of cardinality at least $n/4$, Lemma~\ref{lem:projection1}(1) implies that the cardinality of the set $\{\pi_\sfv((q_{2n})_*(\Delta(g\ell_1)))\}_{g\in V_{n,x}}$ is at least $n/4$. In addition, since $\Delta(\ell_1)$ and $\Delta(\ell)$ are not adjacent in $\Lambda_{2n}^e$, and $(q_{2n})_\ast$ is an isomorphism, we have $\Delta(\ell'_1)\notin\st(\Delta(\ell'))$. We can thus apply Lemma~\ref{lem:projection1}(2) and deduce that the above set (which as we explained is exactly one orbit of $V_{n,x}$ for the action $\alpha_{2,n,\sfv}$) is finite. We have thus found a subgroup of $U'_{n,\sfv}$ whose action on $Z_\sfv$ via $\alpha_{2,n,\sfv}$ has a finite orbit of size at least $n/4$, as desired.
\end{proof}

			\footnotesize
			
			\bibliographystyle{alpha}
			\bibliography{l1}

\newcommand{\etalchar}[1]{$^{#1}$}
\begin{thebibliography}{DKLMT22}

\bibitem[AB08]{AB}
P.~Abramenko and K.S. Brown.
\newblock {\em Buildings. Theory and applications}, volume 248 of {\em Graduate
  Texts in Mathematics}.
\newblock Springer, New York, 2008.

\bibitem[Ago13]{Ago}
I.~Agol.
\newblock The virtual {H}aken conjecture.
\newblock {\em Doc. Math.}, 18:1045--1087, 2013.

\bibitem[AT08]{AT}
A.~Arhangel'skii and M.~Tkachenko.
\newblock {\em Topological groups and related structures}, volume~1 of {\em
  Atlantis Studies in Mathematics}.
\newblock Atlantis Press, Paris; World Scientific Publishing Co. Pte. Ltd.,
  Hackensack, NJ, 2008.

\bibitem[Aus16a]{Aus2}
T.~Austin.
\newblock Behaviour of entropy under bounded and integrable orbit equivalence.
\newblock {\em Geom. Funct. Anal.}, 26(6):1483--1525, 2016.

\bibitem[Aus16b]{Aus}
T.~Austin.
\newblock Integrable measure equivalence for groups of polynomial growth.
\newblock {\em Groups Geom. Dyn.}, 10(1):117--154, 2016.

\bibitem[BC12]{BC}
J.~Behrstock and R.~Charney.
\newblock Divergence and quasimorphisms of right-angled {A}rtin groups.
\newblock {\em Math. Ann.}, 352(2):339--356, 2012.

\bibitem[BCG{\etalchar{+}}09]{BCGNW}
J.~Brodzki, S.J. Campbell, E.~Guentner, G.A. Niblo, and N.J. Wright.
\newblock Property {A} and $\mathrm{CAT}(0)$ cube complexes.
\newblock {\em J. Funct. Anal.}, 256(5):1408--1431, 2009.

\bibitem[BCGM19]{BCGM}
U.~Bader, P.-E. Caprace, T.~Gelander, and S.~Mozes.
\newblock Lattices in amenable groups.
\newblock {\em Fund. Math.}, 246(3):217--255, 2019.

\bibitem[Bel68]{Bel}
R.~M. Belinskaja.
\newblock Partitionings of a {L}ebesgue space into trajectories which may be
  defined by ergodic automorphisms.
\newblock {\em Funkcional. Anal. i Prilo\v{z}en.}, 2(3):4--16, 1968.

\bibitem[BFS13]{BFS}
U.~Bader, A.~Furman, and R.~Sauer.
\newblock Integrable measure equivalence and superrigidity of hyperbolic
  lattices.
\newblock {\em Invent. Math.}, 194(2):313--379, 2013.

\bibitem[BFS20]{BFS2}
U.~Bader, A.~Furman, and R.~Sauer.
\newblock Lattice envelopes.
\newblock {\em Duke Math. J.}, 169(2):213--278, 2020.

\bibitem[BH99]{BH}
M.R. Bridson and A.~Haefliger.
\newblock {\em Metric spaces of non-positive curvature}, volume 319 of {\em
  Grundlehren der mathematischen Wissenschaften}.
\newblock Springer-Verlag, Berlin, 1999.

\bibitem[BJN10]{BJN}
J.A. Behrstock, T.~Januszkiewicz, and W.D. Neumann.
\newblock Quasi-isometric classification of some high dimensional right-angled
  {A}rtin groups.
\newblock {\em Groups Geom. Dyn.}, 4(4):681--692, 2010.

\bibitem[BKMM12]{BKMM}
J.~Behrstock, B.~Kleiner, Y.~Minsky, and L.~Mosher.
\newblock Geometry and rigidity of mapping class groups.
\newblock {\em Geom. Topol.}, 16(2):781--888, 2012.

\bibitem[BKS08]{BKS}
M.~Bestvina, B.~Kleiner, and M.~Sageev.
\newblock The asymptotic geometry of right-angled {A}rtin groups. {I}.
\newblock {\em Geom. Topol.}, 12(3):1653--1699, 2008.

\bibitem[BM00a]{BM2}
M.~Burger and S.~Mozes.
\newblock Groups acting on trees: from local to global structure.
\newblock {\em Inst. Hautes \'{E}tudes Sci. Publ. Math.}, (92):113--150 (2001),
  2000.

\bibitem[BM00b]{BM}
M.~Burger and S.~Mozes.
\newblock Lattices in product of trees.
\newblock {\em Inst. Hautes \'{E}tudes Sci. Publ. Math.}, (92):151--194 (2001),
  2000.

\bibitem[BN08]{BN}
J.A. Behrstock and W.D. Neumann.
\newblock Quasi-isometric classification of graph manifold groups.
\newblock {\em Duke Math. J.}, 141(2):217--240, 2008.

\bibitem[Bro84]{Brown}
K.S. Brown.
\newblock Presentations for groups acting on simply-connected complexes.
\newblock {\em J. Pure Appl. Algebra}, 32(1):1--10, 1984.

\bibitem[CCV07]{CCV}
R.~Charney, J.~Crisp, and K.~Vogtmann.
\newblock Automorphisms of $2$-dimensional right-angled {A}rtin groups.
\newblock {\em Geom. Topol.}, 11:2227--2264, 2007.

\bibitem[CD95]{CD}
R.~Charney and M.W. Davis.
\newblock Finite {$K(\pi, 1)$}s for {A}rtin groups.
\newblock In {\em Prospects in topology ({P}rinceton, {NJ}, 1994)}, volume 138
  of {\em Ann. of Math. Stud.}, pages 110--124. Princeton Univ. Press,
  Princeton, NJ, 1995.

\bibitem[CdM24]{CdM}
P.-E. Caprace and T.~de~Medts.
\newblock Lattice envelopes of right-angled {A}rtin groups.
\newblock {\em arXiv:2401.15943}, 2024.

\bibitem[CF12]{CF}
R.~Charney and M.~Farber.
\newblock Random groups arising as graph products.
\newblock {\em Algebr. Geom. Topol.}, 12(2):979--995, 2012.

\bibitem[CH15]{CH}
P.-E. Caprace and D.~Hume.
\newblock Orthogonal forms of {K}ac-{M}oody groups are acylindrically
  hyperbolic.
\newblock {\em Ann. Inst. Fourier (Grenoble)}, 65(6):2613--2640, 2015.

\bibitem[Cha07]{Cha}
R.~Charney.
\newblock An introduction to right-angled {A}rtin groups.
\newblock {\em Geom. Dedic.}, 125(1):141--158, 2007.

\bibitem[CL10]{CL}
P.-E. Caprace and A.~Lytchak.
\newblock At infinity of finite-dimensional {CAT}(0) spaces.
\newblock {\em Math. Ann.}, 346(1):1--21, 2010.

\bibitem[Cor15]{Cor}
Y.~Cornulier.
\newblock Commability and focal locally compact groups.
\newblock {\em Indiana Univ. Math. J.}, 64(1):115--150, 2015.

\bibitem[Cor18]{Cor-quasi}
Y.~Cornulier.
\newblock On the quasi-isometric classification of locally compact groups.
\newblock In {\em New directions in locally compact groups}, volume 447 of {\em
  London Math. Soc. Lecture Note Ser.}, pages 275--342. Cambridge Univ. Press,
  Cambridge, 2018.

\bibitem[CS11]{CS}
P.-E. Caprace and M.~Sageev.
\newblock Rank rigidity for {CAT}(0) cube complexes.
\newblock {\em Geom. Funct. Anal.}, 21(4):851--891, 2011.

\bibitem[Dav98]{Dav}
M.W. Davis.
\newblock Buildings are $\mathrm{CAT}(0)$.
\newblock In {\em Geometry and cohomology in group theory (Durham, 1994)},
  volume 252 of {\em London Math. Soc. Lecture Note Ser.}, pages 108--123,
  Cambridge, 1998. Cambridge Univ. Press.

\bibitem[Day12]{Day}
M.B. Day.
\newblock Finiteness of outer automorphism groups of random right-angled
  {A}rtin groups.
\newblock {\em Algebr. Geom. Topol.}, 12(3):1553--1583, 2012.

\bibitem[DKLMT22]{DKLMT}
T.~Delabie, J.~Koivisto, F.~Le~Ma\^{i}tre, and R.~Tessera.
\newblock Quantitative measure equivalence between amenable groups.
\newblock {\em Ann. H. Lebesgue}, 5:1417--1487, 2022.

\bibitem[Dye59]{Dye1}
H.~A. Dye.
\newblock On groups of measure preserving transformations. {I}.
\newblock {\em Amer. J. Math.}, 81:119--159, 1959.

\bibitem[Dye63]{Dye2}
H.~A. Dye.
\newblock On groups of measure preserving transformations. {II}.
\newblock {\em Amer. J. Math.}, 85:551--576, 1963.

\bibitem[EF97]{EF}
A.~Eskin and B.~Farb.
\newblock Quasi-flats and rigidity in higher rank symmetric spaces.
\newblock {\em J. Amer. Math. Soc.}, 10(3):653--692, 1997.

\bibitem[EH24]{EH}
A.~Escalier and C.~Horbez.
\newblock Graph products and measure equivalence: classification, rigidity, and
  quantitative aspects.
\newblock {\em arXiv:2401.04635}, 2024.

\bibitem[Eng89]{Eng}
R.~Engelking.
\newblock {\em General topology}, volume~6 of {\em Sigma Series in Pure
  Mathematics}.
\newblock Heldermann Verlag, Berlin, second edition, 1989.

\bibitem[Esk98]{Esk}
A.~Eskin.
\newblock Quasi-isometric rigidity of nonuniform lattices in higher rank
  symmetric spaces.
\newblock {\em J. Amer. Math. Soc.}, 11(2):321--361, 1998.

\bibitem[Fer18]{Fer}
T.~Fern\'{o}s.
\newblock The {F}urstenberg-{P}oisson boundary and {${\rm CAT}(0)$} cube
  complexes.
\newblock {\em Ergodic Theory Dynam. Systems}, 38(6):2180--2223, 2018.

\bibitem[FHT11]{FHT}
B.~Farb, C.~Hruska, and A.~Thomas.
\newblock Problems on automorphism groups of nonpositively curved polyhedral
  complexes and their lattices.
\newblock In {\em Geometry, rigidity, and group actions}, Chicago Lectures in
  Math., pages 515--560. Univ. Chicago Press, Chicago, IL, 2011.

\bibitem[FLM18]{FLM}
T.~Fern\'os, J.~L\'ecureux, and F.~Math\'eus.
\newblock Random walks and boundaries of $\mathrm{CAT}(0)$ cubical complexes.
\newblock {\em Comment. Math. Helv.}, 93(2):291--333, 2018.

\bibitem[Fur63]{Furs}
H.~Furstenberg.
\newblock A {P}oisson formula for semi-simple {L}ie groups.
\newblock {\em Ann. of Math. (2)}, 77:335--386, 1963.

\bibitem[Fur99a]{Fur-me}
A.~Furman.
\newblock Gromov's measure equivalence and rigidity of higher-rank lattices.
\newblock {\em Ann.of Math. (2)}, 150(3):1059--1081, 1999.

\bibitem[Fur99b]{Fur-oe}
A.~Furman.
\newblock Orbit equivalence rigidity.
\newblock {\em Ann.of Math. (2)}, 150(3):1083--1108, 1999.

\bibitem[Fur01]{Fur-lattice}
A.~Furman.
\newblock Mostow-{M}argulis rigidity with locally compact targets.
\newblock {\em Geom. Funct. Anal.}, 11(1):30--59, 2001.

\bibitem[Gab02]{Gab-l2}
D.~Gaboriau.
\newblock Invariants {$l^2$} de relations d'\'{e}quivalence et de groupes.
\newblock {\em Publ. Math. Inst. Hautes \'{E}tudes Sci.}, (95):93--150, 2002.

\bibitem[GH21]{GH}
V.~Guirardel and C.~Horbez.
\newblock Measure equivalence rigidity of $\mathrm{Out}({F}_{N})$.
\newblock {\em arXiv:2103.03696}, 2021.

\bibitem[God03]{God}
E.~Godelle.
\newblock Parabolic subgroups of {A}rtin groups of type {FC}.
\newblock {\em Pacific J. Math.}, 208(2):243--254, 2003.

\bibitem[Gre90]{Gre}
E.R. Green.
\newblock {\em Graph products of groups}.
\newblock PhD thesis, University of Leeds, 1990.

\bibitem[Gro93]{Gro}
M.~Gromov.
\newblock Asymptotic invariants of infinite groups.
\newblock In {\em Geometric group theory, {V}ol. 2 ({S}ussex, 1991)}, volume
  182 of {\em London Math. Soc. Lecture Note Ser.}, pages 1--295. Cambridge
  Univ. Press, Cambridge, 1993.

\bibitem[Ham07]{Ham}
U.~Hamenstädt.
\newblock Geometry of the mapping class groups {III}: {Q}uasi-isometric
  rigidity.
\newblock {\em arXiv:math/0512429}, 2007.

\bibitem[HH20]{HH1}
C.~Horbez and J.~Huang.
\newblock Boundary amenability and measure equivalence rigidity among
  two-dimensional {A}rtin groups of hyperbolic type.
\newblock {\em arXiv:2004.09325}, 2020.

\bibitem[HH22a]{HH}
C.~Horbez and J.~Huang.
\newblock Measure equivalence classification of transvection-free right-angled
  {A}rtin groups.
\newblock {\em J. \'Ec. polytech. Math.}, 9:1021--1067, 2022.

\bibitem[HH22b]{HH-Higman}
C.~Horbez and J.~Huang.
\newblock Measure equivalence rigidity among the {H}igman groups.
\newblock {\em arXiv:2206.00884}, 2022.

\bibitem[HHI23]{HHI}
C.~Horbez, J.~Huang, and A.~Ioana.
\newblock Orbit equivalence rigidity of irreducible actions of right-angled
  {A}rtin groups.
\newblock {\em Compos. Math.}, 159(4):860--887, 2023.

\bibitem[HK18]{HK}
J.~Huang and B.~Kleiner.
\newblock Groups quasi-isometric to right-angled {A}rtin groups.
\newblock {\em Duke Math. J.}, 167(3):537--602, 2018.

\bibitem[HM23]{HM}
J.~Huang and M.~Mj.
\newblock Indiscrete common commensurators.
\newblock {\em arXiv:2310.04876}, 2023.

\bibitem[Hua16]{Hua2}
J.~Huang.
\newblock Quasi-isometry classification of right-angled {A}rtin groups {II}:
  several infinite out cases.
\newblock {\em arXiv:1603.02372}, 2016.

\bibitem[Hua17]{Hua}
J.~Huang.
\newblock Quasi-isometric classification of right-angled {A}rtin groups {I}:
  the finite out case.
\newblock {\em Geom. Topol.}, 21(6):3467--3537, 2017.

\bibitem[Hua18]{Hua-QI}
J.~Huang.
\newblock Commensurability of groups quasi-isometric to raags.
\newblock {\em Invent. Math.}, 213(3):1179--1247, 2018.

\bibitem[Hug21]{Hug}
S.~Hughes.
\newblock Graphs and complexes of lattices.
\newblock {\em arXiv:2104.13728}, 2021.

\bibitem[HW08]{HW}
F.~Haglund and D.T. Wise.
\newblock Special cube complexes.
\newblock {\em Geom. Funct. Anal.}, 17(5):1551--1620, 2008.

\bibitem[HW12]{HW2}
F.~Haglund and D.T. Wise.
\newblock A combination theorem for special cube complexes.
\newblock {\em Ann. of Math. (2)}, 176(3):1427--1482, 2012.

\bibitem[Iva97]{Iva}
N.V. Ivanov.
\newblock Automorphism of complexes of curves and of {T}eichm\"{u}ller spaces.
\newblock {\em Internat. Math. Res. Notices}, (14):651--666, 1997.

\bibitem[Jus71]{Jus}
J.~Justin.
\newblock Groupes et semi-groupes \`a croissance lin\'{e}aire.
\newblock {\em C. R. Acad. Sci. Paris S\'{e}r. A-B}, 273:A212--A214, 1971.

\bibitem[Kec95]{Kec}
A.S. Kechris.
\newblock {\em Classical descriptive set theory}, volume 156 of {\em Graduate
  Texts in Mathematics}.
\newblock Springer-Verlag, New York, 1995.

\bibitem[Kid09]{Kid-survey}
Y.~Kida.
\newblock Introduction to measurable rigidity of mapping class groups.
\newblock In {\em Handbook of Teichmüller theory, volume II}, volume~13 of
  {\em IRMA Lect. Math. Theor. Phys.}, pages 297--367. Eur. Math. Soc., 2009.

\bibitem[Kid10]{Kid-me}
Y.~Kida.
\newblock Measure equivalence rigidity of the mapping class group.
\newblock {\em Ann. of Math. (2)}, 171(3):1851--1901, 2010.

\bibitem[Kid11]{Kid-amalgam}
Y.~Kida.
\newblock Rigidity of amalgamated free products in measure equivalence.
\newblock {\em J. Topol.}, 4(3):687--735, 2011.

\bibitem[KK13]{KK}
S.-h. Kim and T.~Koberda.
\newblock Embedability between right-angled {A}rtin groups.
\newblock {\em Geom. Topol.}, 17(1):493--530, 2013.

\bibitem[KL97]{KL}
B.~Kleiner and B.~Leeb.
\newblock Rigidity of quasi-isometries for symmetric spaces and {E}uclidean
  buildings.
\newblock {\em Inst. Hautes \'{E}tudes Sci. Publ. Math.}, (86):115--197 (1998),
  1997.

\bibitem[KL21]{KLi}
D.~Kerr and H.~Li.
\newblock Entropy, {S}hannon orbit equivalence, and sparse connectivity.
\newblock {\em Math. Ann.}, 380(3-4):1497--1562, 2021.

\bibitem[KL24]{KLi2}
D.~Kerr and H.~Li.
\newblock Entropy, virtual {A}belianness and {S}hannon orbit equivalence.
\newblock {\em Ergodic Theory Dynam. Systems}, 44(12):3481--3500, 2024.

\bibitem[KS16]{KS}
A.~Kar and M.~Sageev.
\newblock Ping pong on {$\rm CAT(0)$} cube complexes.
\newblock {\em Comment. Math. Helv.}, 91(3):543--561, 2016.

\bibitem[Lau95]{Lau}
M.R. Laurence.
\newblock A generating set for the automorphism group of a graph group.
\newblock {\em J. London Math. Soc. (2)}, 52(2):318--334, 1995.

\bibitem[LLM23]{LLM}
N.~Lazarovich, I.~Levcovitz, and A.~Margolis.
\newblock Counting lattices in products of trees.
\newblock {\em Comment. Math. Helv.}, 98(3):597--630, 2023.

\bibitem[Mar91]{Mar}
G.~A. Margulis.
\newblock {\em Discrete subgroups of semisimple {L}ie groups}, volume~17 of
  {\em Ergebnisse der Mathematik und ihrer Grenzgebiete (3) [Results in
  Mathematics and Related Areas (3)]}.
\newblock Springer-Verlag, Berlin, 1991.

\bibitem[Mar20]{Marg}
A.~Margolis.
\newblock Quasi-isometry classification of right-angled {A}rtin groups that
  split over cyclic subgroups.
\newblock {\em Groups Geom. Dyn.}, 14(4):1351--1417, 2020.

\bibitem[MS06]{MS}
N.~Monod and Y.~Shalom.
\newblock Orbit equivalence rigidity and bounded cohomology.
\newblock {\em Ann. of Math. (2)}, 164(3):825--878, 2006.

\bibitem[OW80]{OW}
D.S. Ornstein and B.~Weiss.
\newblock Ergodic theory of amenable group actions. {I}. {T}he {R}ohlin lemma.
\newblock {\em Bull. Amer. Math. Soc. (N.S.)}, 2(1):161--164, 1980.

\bibitem[Rad20]{Rad}
N.~Radu.
\newblock New simple lattices in products of trees and their projections.
\newblock {\em Canad. J. Math.}, 72(6):1624--1690, 2020.

\bibitem[Rag72]{Rag}
M.~S. Raghunathan.
\newblock {\em Discrete subgroups of {L}ie groups.}
\newblock Springer-Verlag, New York-Heidelberg,, 1972.

\bibitem[Rol98]{Rol}
M.A. Roller.
\newblock Poc sets, median algebras and group actions. an extended study of
  {D}unwoody's construction and {S}ageev's theorem.
\newblock 1998.
\newblock preprint.

\bibitem[Sag14]{Sag}
M.~Sageev.
\newblock {$\rm CAT(0)$} cube complexes and groups.
\newblock In {\em Geometric group theory}, volume~21 of {\em IAS/Park City
  Math. Ser.}, pages 7--54. Amer. Math. Soc., Providence, RI, 2014.

\bibitem[Sal87]{Sal}
M.~Salvetti.
\newblock Topology of the complement of real hyperplanes in {${\bf C}^N$}.
\newblock {\em Invent. Math.}, 88(3):603--618, 1987.

\bibitem[Ser89]{Ser}
H.~Servatius.
\newblock Automorphisms of graph groups.
\newblock {\em J. Algebra}, 126(1):34--60, 1989.

\bibitem[Sha00]{Sha2}
Y.~Shalom.
\newblock Rigidity of commensurators and irreducible lattices.
\newblock {\em Invent. Math.}, 141(1):1--54, 2000.

\bibitem[Sha04]{Sha}
Y.~Shalom.
\newblock Harmonic analysis, cohomology, and the large-scale geometry of
  amenable groups.
\newblock {\em Acta Math.}, 192(2):119--185, 2004.

\bibitem[She24]{She}
S.~Shepherd.
\newblock Commensurability of lattices in right-angled buildings.
\newblock {\em Adv. Math.}, 441:Paper No. 109522, 55, 2024.

\bibitem[Wal67]{Wal}
C.~T.~C. Wall.
\newblock Poincar\'{e} complexes. {I}.
\newblock {\em Ann. of Math. (2)}, 86:213--245, 1967.

\bibitem[Zim80]{Zim2}
R.J. Zimmer.
\newblock Strong rigidity for ergodic actions of semisimple {L}ie groups.
\newblock {\em Ann. of Math. (2)}, 112(3):511--529, 1980.

\bibitem[Zim84]{Zim}
R.J. Zimmer.
\newblock {\em Ergodic {T}heory and {S}emisimple {G}roups}, volume~81 of {\em
  Monographs in Mathematics}.
\newblock Birkhäuser, Basel; Boston; Stuttgart, 1984.

\end{thebibliography}

\begin{flushleft}
Camille Horbez\\ 
Universit\'e Paris-Saclay, CNRS,  Laboratoire de math\'ematiques d'Orsay, 91405, Orsay, France \\
\emph{e-mail:~}\texttt{camille.horbez@universite-paris-saclay.fr}\\[4mm]
\end{flushleft}

\begin{flushleft}
Jingyin Huang\\
Department of Mathematics\\
The Ohio State University, 100 Math Tower\\
231 W 18th Ave, Columbus, OH 43210, U.S.\\
\emph{e-mail:~}\texttt{huang.929@osu.edu}\\
\end{flushleft}
~
\begin{flushleft}
This work is openly licensed via CC-BY 4.0 (https://creativecommons.org/licenses/by/4.0/).
\end{flushleft}

		\end{document}